\documentclass{Hybrid_Arxiv}

\usepackage[super]{cite}
\usepackage{url}

\usepackage{textcomp}
\usepackage{subfigure}
\usepackage{hyperref}

\newtheorem{thm}{Theorem}[section]
\newtheorem{prop}[thm]{Proposition}

\newtheorem{cor}[thm]{Corollary}

\renewcommand{\O}{\Omega}

\newcommand{\ds}{\,\mbox{d}s}

\newcommand\E{{E}}  % element
\newcommand\dE{{\partial \E}}

\newcommand\Th{{\mathcal T}_h}

\newcommand\Eh{{\mathcal E}_h}

\newcommand\R{\mathbb{R}}
\newcommand\C{\mathbb{C}}
\newcommand\D{\mathbb{D}}

\renewcommand{\P}{{\mathcal P}}  % polynomials

\def\decapita#1{}
\def\grecomultibold#1#2{\grecobolddef#1\def\secondobold{#2}%
	\ifx#2\finemultibold\let\next\relax\let\secondobold\relax
	\else\let\next\grecomultibold
	\fi\expandafter\next\secondobold}

\def\grecobolddef#1{%
	\edef\dadef{bf\expandafter\decapita\string#1}%
	\expandafter\def\csname\dadef\endcsname{{\neretto #1}}}

\def\neretto#1{\setbox0=\hbox{\mathsurround=0pt$#1$}%
	\kern.02em\copy0 \kern-\wd0
	\kern-.02em\copy0 \kern-\wd0
	\raise.03em\box0 \kern.02em}
\grecomultibold\alpha\phi\beta\omega\gamma\delta\xi\sigma\tau\eta\theta\varepsilon\rho\psi\pi\varphi\varepsilon\lambda\mu\finemultibold
\def\bdiv{\mathop{\bf div}\nolimits}

\def\teps{{\bfvarepsilon}}

\def\IE{{\mathcal I}_E}

\def\wbox#1;#2;{\vbox{\hrule\hbox{\vrule height#1mm\kern#2mm\vrule
			height#1mm}\hrule}}

\let\phi\varphi

\newcommand{\bbc}   {\mathbf{c}}

\newcommand{\bbf}   {\mathbf{f}}
\newcommand{\bbg}   {\mathbf{g}}

\newcommand{\bbn}   {\mathbf{n}}

\newcommand{\bbp}   {\mathbf{p}}
\newcommand{\bbq}   {\mathbf{q}}
\newcommand{\bbr}   {\mathbf{r}}

\newcommand{\bbt}   {\mathbf{t}}
\newcommand{\bbu}   {\mathbf{u}}
\newcommand{\bbv}   {\mathbf{v}}
\newcommand{\bbw}   {\mathbf{w}}
\newcommand{\bbx}   {\mathbf{x}}

\def\hpoint#1.#2.#3{{\underline{#1}}_{#2}\cdot
	{\underline{\mathop{\smash{#3}\vphantom{{#1}_{#2}}}}}}

\def\npointp#1.#2{{\underline{#1}}\cdot
	{\underline{\mathop{\smash{#2}\vphantom{{#1}}}}}}

\def\beq{\begin{equation}}
\def\enq{\end{equation}}

\def\bfzero{{\bf 0}}

\newcommand{\dO}    {\text{d}\Omega}
\newcommand{\dEl}    {\text{d} E}

\newcommand{\ljump}{\text{\textlbrackdbl}}
\newcommand{\rjump}{\text{\textrbrackdbl}}

\begin{document}

%\markboth{Authors' Names}{Instructions for Typing Manuscripts (Paper's Title)}

%%%%%%%%%%%%%%%%%%% Publisher's Area please ignore %%%%%%%%%%%%%%%%%%%%%%%
%
%\catchline{}{}{}{}{}
%
%%%%%%%%%%%%%%%%%%%%%%%%%%%%%%%%%%%%%%%%%%%%%%%%%%%%%%%%%%%%%%%%%%%%%%%%%%
%\title{Instructions for typesetting manuscripts\\
%	using \LaTeX2e\footnote{For the title, try not to
%		use more than 3 lines. Typeset the title in 10 pt
%		Times Roman  and boldface.} }

%\author{First Author\footnote{Typeset names in 8 pt Roman. Use the footnote to indicate the
%present or permanent address of the author.}}
%
%\address{University Department, University Name, Address\\
%City, State ZIP/Zone,
%Country\footnote{State completely without abbreviations, the
%affiliation and mailing address, including country. Typeset in 8 pt
%Times italic.}\\
%first\_author@university.edu}
%
%\author{Second Author}
%\address{Group, Laboratory, Address\\
%	City, State ZIP/Zone, Country\\
%	second\_author@group.com}
%\markboth{F. Dassi, C. Lovadina \& M. Visinoni}{Hybridization of the Virtual Element Method for linear elasticity problems}

%%%%%%%%%%%%%%%%%%% Publisher's Area please ignore %%%%%%%%%%%%%%%%%%%%%%%
%
%\catchline{}{}{}{}{}
%
%%%%%%%%%%%%%%%%%%%%%%%%%%%%%%%%%%%%%%%%%%%%%%%%%%%%%%%%%%%%%%%%%%%%%%%%%%
\title{Hybridization of the Virtual Element Method\\for linear elasticity problems}

\author{Franco Dassi}
\address{Dipartimento di Matematica e Applicazioni,\\
Universit\`a degli studi di Milano Bicocca,\\
Via Roberto Cozzi 55 - I-20125 Milano, Italy\\
franco.dassi@unimib.it}

\author{Carlo Lovadina\footnote{Corresponding author}}
\address{Dipartimento di Matematica, Universit\`a di Milano,\\
Via Saldini 50, I-20133 Milano, Italy\\
\vspace*{0.2cm}
IMATI del CNR, Via Ferrata 1, 27100 Pavia, Italy\\
carlo.lovadina@unimi.it}
%%
%% UNIMI
%%
\author{Michele Visinoni}
\address{Dipartimento di Matematica, Universit\`a di Milano,\\
Via Saldini 50, I-20133 Milano, Italy\\
michele.visinoni@unimi.it}
%%
%% BICOCCA
%%
%\author{Michele Visinoni}
%\address{Dipartimento di Matematica e Applicazioni,\\
%Universit\`a degli studi di Milano Bicocca,\\
%Via Roberto Cozzi 55 - I-20125 Milano, Italy
%michele.visinoni@unimib.it}

\maketitle

%\begin{history}
%\received{(Day Month Year)}
%\revised{(Day Month Year)}
%%\accepted{(Day Month Year)}
%\comby{(xxxxxxxxxx)}
%\end{history}

\begin{abstract}
\section*{Abstract}
We extend the hybridization procedure proposed in Ref.~\refcite{ArnoldBrezzi_1985} to the Virtual Element Method for linear elasticity problems based on the Hellinger-Reissner principle.
To illustrate such a technique, we focus on the 2D case, but other methods and 3D problems can be considered as well. We also show how to design a better approximation of the displacement field using a straightforward post-processing procedure. 
The numerical experiments confirm the theory for both two and three-dimensional problems.
\end{abstract}

\keywords{Virtual Element Methods; Elasticity Problems; Hybridization.}

\ccode{AMS Subject Classification: 65N30, 65N12, 65N15}

%%%%%%%%%%%%%%%%%%%%%%%%%%%%%%%%%%%%%%%%%%%%%%%%%%%%%
%% Introduction
%%
%%%%%%%%%%%%%%%%%%%%%%%%%%
\section{Introduction}
The Virtual Element Method (VEM) is a generalization of the Finite Element Method (FEM), which allows to deal with general polytopal meshes, also including non-convex or distorted elements, as well as hanging nodes, see Refs.~\refcite{volley,hitchhikers}.
The fundamental idea of this technology is hidden behind the definition of the approximation spaces:
% to guarantee this flexibility, we abandon the \emph{polynomial} approximation setting. 
VEM spaces contain suitable polynomials as FEM, but also some other functions, that are solutions of a local PDE.
Thereby, we do not know the discrete functions pointwise, but we know them only by means of a limited set of information, i.e., the degrees of freedom. Nevertheless, the available information is sufficient to construct the stiffness matrix and the right-hand side to set up and solve the discrete scheme.
In Structural Mechanics and elasticity fields the flexibility in handling general polygonal and polyhedral meshes ensures an high success of VEM in such communities. Here we mention, as a representative non-exhaustive sample, a brief list of papers in the framework of structural mechanics problems: Refs.~\refcite{ABLS_part_I,ABLS_part_II,ARTIOLI2020112667,ARTIOLI_RICOVERYVEM,BeiraodaVeiga-Brezzi-Marini:2013,BeiraoLovaMora,CHI2017148,Paulino-VEM,Brezzi-Marini:2012,wriggers2017,DALTRI2021113663,Hudobivnik2019,ZHANG20181}.
Some examples of other numerical methods for the elasticity problem that can handle polytopal meshes are Refs.~\refcite{BOTTI201996,CockburnFu,CockburnShi,Di-Pietro.Ern:15*2}.
In this paper we focus on the numerical approximation of the elasticity problem, using the Hellinger-Reissner variational formulation (see Ref. \refcite{BoffiBrezziFortin}, for example), and in particular we consider the so-called hybridization procedure.
The idea behind the hybridization strategy dates back to 1965, see Ref.~\refcite{Freajis_de_Veubeke}, and it has been used as an implementation technique (for FEM) to solve 2nd order differential problems in mixed form (see Ref. \refcite{BoffiBrezziFortin}, for instance). Essentially, this procedure consists in using Lagrange multipliers to impose the required continuity constraints across the inter-elements, rather than enforcing them directly in the approximation spaces. Then, a static condensation technique is employed to obtain a symmetric and positive definite linear system. Consequently, the resolution of a large indefinite linear system is replaced by the resolution of a lower dimensional positive definite one. Other than such an advantage, a suitable post-processing procedure of the discrete solution and the Lagrange multipliers leads to an improved approximation of the displacement variable, see Ref.~\refcite{ArnoldBrezzi_1985}.
We here present and analyse the hybridization technique applied to
Hellinger-Reissner conforming VEMs. In contrast to FEMs, the flexibility of VEMs allows us to design cheap and conforming schemes with a-priori symmetric Cauchy stresses both in 2D and 3D, see Refs.\refcite{ARTIOLI2017155}, \refcite{ARTIOLI2018978} and \refcite{DLV}. One interesting aspect of the proposed VEM scheme is that it does not exploit point values at the mesh vertices, so the hybridization procedure becomes more straightforward.
However, we wish to recall that several Hellinger-Reissner FEM schemes have been proposed: as a few examples, we cite the conforming one presented in Ref. \refcite{ArnoldWinther}, the one based on composite elements studied in Ref. \refcite{JohnsonMercier}, the one based on the symmetry reduction detailed in Ref. \refcite{Arnold1984}, and the one based on the recent interesting approach analysed in Ref. \refcite{Schoeberl2,Schoeberl1}.

A brief outline of the paper is as follows. In Sec.~\ref{section:modelProblem} we present the continuous elasticity problem.
Sec.~\ref{section:hybrid} introduces the hybridization technique with its computational aspects. The 2D low-order VEM scheme studied in Ref. \refcite{ARTIOLI2017155} has been selected to illustrate the procedure.
In Sec.~\ref{section:h_errorAnalysis} an error analysis is developed for the above-mentioned method, both recalling known results, and proving new estimates regarding the Lagrange multipliers.
In Sec.~\ref{section:post-processing} we propose and study a post-processed displacement solution which exploits the information provided by the computed Lagrange multipliers. We highlight that in several points our analysis follows the guidelines detailed in Ref. \refcite{ArnoldBrezzi_1985} for the laplacian problem in mixed form. However, the VEM approach here requires peculiar technical tools which often differ from the typical ones of FEMs. 
In Sec.~\ref{section:hybridNumer} we present some experiments to give numerical evidence of the proposed VEM approach for both two and three dimensional problems. In the last section we draw some conclusions.

\paragraph{Space notation.} In this paper we will use the standard notation regarding Sobolev spaces, norms and seminorms, see for instance Ref.~\refcite{Lions-Magenes}. Given two quantities $a$ and $b$, we write $a \lesssim b$ when there exists a constant $C$, independent of the mesh size (but possibly dependent on the regularity of the continuous elastic problem), such that $a\leq C b$. Moreover, given any subset $A\subset \R^d$ and an integer $k\geq 0$, we denote by $\P_k(A)$ the space of polynomials up to degree $k$, defined on $A$; whereas, given a functional space $X$, we indicate with $\left[X\right]_s^{d\times d}$ the $d\times d$ symmetric tensor whose components belong to the space $X$.

\paragraph{Mesh notation.} Given a polygon $E$ with $n_e^E$ edges, we denote its area, diameter and barycenter by $|E|$, $h_E$ and $\bbx_E$. Moreover, we denote by $|e|$ and $\bbx_e$ the length and the middle point of an edge $e$, respectively.

%%%%%%%%%%%%%%%%%%%%%%%%%%
%%
%% The Hellinger-Reissner elasticity problem
%%
%%%%%%%%%%%%%%%%%%%%%%%%%%
\section{The Hellinger-Reissner elasticity problem}~\label{section:modelProblem}
In the present section, we introduce the elasticity problem which ensues from the Hellinger-Reissner principle, see Refs.~\refcite{BoffiBrezziFortin,Braess:book}. Let $\O$ be a polytopal domain in $\R^d$, $d=2,3$ and we consider the following elasticity problem:
\begin{equation}\label{problem:cont-strong}
\left\lbrace{
	\begin{aligned}
	&\mbox{Find } (\bfsigma,\bbu)~\mbox{such that}\\
	&-\bdiv \bfsigma= \bbf\  &\mbox{in $\Omega$}\\
	& \bfsigma = \C \teps(\bbu)\ &\mbox{in $\Omega$}\\
	&\bbu=\bbg\ &\mbox{in $\partial\Omega$}
	\end{aligned}	
} \right. ,
\end{equation}
where $\bfsigma$ and $\bbu$ represent the stress and the displacement field, respectively. 
In addition, $\bbf$ is a function in $\left[L^2(\O)\right]^d$ which represents the loading term and $\bbg\in \left[H^{1/2}(\partial \O)\right]^d$ is the displacement on boundary. Furthermore, we assume that the elasticity fourth-order symmetric tensor $\C$ is uniformly-bounded, positive-definite and sufficiently smooth.
To set the variational formulation of Problem~\eqref{problem:cont-strong}, we define the spaces
\begin{equation}\label{eq:globalspaces}
U:=\left[L^2(\Omega)\right]^d, \quad\quad 
\Sigma:=\left\{ \bftau\in H(\bdiv;\Omega)\ :\ \bftau \mbox{ is symmetric} \right\}
\end{equation}
with standard norms.
As usual, $H(\bdiv;\Omega)$ is the space of tensor in $\left[L^2(\O)\right]^{d\times d}$ whose divergence is the vector-valued operator in $\left[L^2(\O)\right]^d$. 
We define the bilinear forms $a(\cdot,\cdot): \Sigma \times \Sigma \rightarrow \R$ and $b(\cdot,\cdot): \Sigma \times U \rightarrow \R$ as
\begin{equation}
a(\bfsigma,\bftau):=\int_{\O}\D \bfsigma:\bftau~\dO, \qquad
 b(\bfsigma,\bbu):=\int_{\O}\bdiv\bfsigma\cdot \bbu~\dO,
\end{equation}
where $\D:=\C^{-1}$ is the inverse of the Cauchy tensor. Then the mixed weak formulation of Problem~\eqref{problem:cont-strong} reads
\begin{equation}\label{problem:cont-weak}
\left\lbrace{
	\begin{aligned}
	&\mbox{Find } (\bfsigma,\bbu)\in \Sigma\times U~\mbox{such that}\\
	&a(\bfsigma,\bftau) + b(\bftau, \bbu)=<\bbg,\bftau\,\bbn>  &\forall \bftau\in \Sigma\\
	& b(\bfsigma, \bbv) = -(\bbf,\bbv) &\forall \bbv\in U
	\end{aligned}
} \right. ,
\end{equation}
where  $(\cdot,\cdot)$ is the inner product in  $\left[L^2(\O)\right]^d$, whereas $<\cdot,\cdot>$ is the duality product between $\left[H^{1/2}(\partial \O)\right]^d$ and $\left[H^{-1/2}(\partial \O)\right]^d$. 
% aggiungi alle notazioni
It is well-known that Problem~\eqref{problem:cont-weak} is well posed, see for instance Ref.~\refcite{BoffiBrezziFortin}. %In particular, it holds
%$$
%||\bfsigma||_{\Sigma}+||\bbu||_U\leq C\left(||\bbf||_0+||\bbg||_{1/2}\right),
%$$
%where $C$ is a constant depending on $\O$ and on the material tensor $\D$, which does not degenerate in the incompressible limit.
%%%%%%%%%%%%%%%%%%%%%%%%%%%%%%%%%%%%%%%%%%%%%%%%
%%
%% HYBRID PROCEDURE
%%
%%%%%%%%%%%%%%%%%%%%%%%%%%%%%%%%%%%%%%%%%%%%%%%%
\section{Hybridization procedure}~\label{section:hybrid}
In this section we present the hybridization technique for the mixed approximation of Problem~\eqref{problem:cont-weak}, cf. Ref.~\refcite{Freajis_de_Veubeke}. 
First of all, we recall a typical discrete formulation of Problem~\eqref{problem:cont-weak}:
\begin{equation}\label{problem:descretised-formulation}
\left\lbrace{
	\begin{aligned}
	&\mbox{Find } (\bfsigma_h,\bbu_h)\in \Sigma_h\times U_h~\mbox{such that}\\
	&a_h(\bfsigma_h,\bftau_h) + b_h(\bftau_h, \bbu_h)=<\bbg,\bftau_h\,\bbn>_h  &\forall \bftau_h\in \Sigma_h\\
	& b_h(\bfsigma_h, \bbv_h) = -(\bbf,\bbv_h)_h &\forall \bbv_h\in U_h,
	\end{aligned}
} \right.
\end{equation}
where $\Sigma_h$ and $U_h$ are the global discrete spaces for the stress and displacement field, respectively. Moreover, $a_h(\cdot,\cdot)$, $b_h(\cdot,\cdot)$, $<\bbg,\cdot>_h$ and $(\bbf,\cdot)_h$ are suitable approximations of the corresponding bilinear and linear forms. 
More details about possible choices of the spaces for a conforming low-order VEM can be found in Refs.~\refcite{ARTIOLI2017155,DLV}. 
%The matrix of the linear system associated with the discrete formulation of the Problem~\eqref{problem:cont-weak} has the following saddle-point form
The linear system associated with~\eqref{problem:descretised-formulation} has the following form
\begin{equation}\label{eq:linearSystem_matrix}
%\left\{
\begin{pmatrix}
A & B\\
B^T & 0   
\end{pmatrix}
\begin{pmatrix}
\bfsigma_h \\
\bbu_h  
\end{pmatrix}
=
\begin{pmatrix}
G \\
F   
\end{pmatrix}
\end{equation}
whose matrix is indefinite.
%%----------
%% procedure 
%%----------
The hybridization procedure is an implementation technique which leads to solve a linear system with a symmetric and positive definite matrix instead of the original indefinite one~\eqref{eq:linearSystem_matrix}.
%a symmetric and positive definite linear system instead of the original indefinite one~\eqref{eq:linearSystem_matrix}.
The procedure is split into two different steps: the imposition of the stress $H(\bdiv)$-conformity requirement through the introduction of suitable Lagrange multipliers, and the static condensation algorithm.
\begin{remark}%ke algebras and their represproblem:cont-weak
	It is worth noticing that the possibility to perform hybridization highly depends on the particular features of the discrete scheme, and it is not always possible. With this respect, the structure of the discrete space $\Sigma_h$ is essential.
\end{remark}
In what follows, we illustrate the hybridization procedure using the 2D VEM scheme presented in Ref. \refcite{ARTIOLI2017155}. Accordingly, we recall the approximation spaces and the bilinear and linear forms involved in the method. Such quantities are, as usual, defined locally on each element. Afterwards, all the contributions are glued together to form the discrete problem. 
Such procedure can be extended to the 3D case when using the scheme proposed in Ref. \refcite{DLV} (in fact numerical results for such an instance are presented in Sec.~\ref{section:hybridNumer}).
\subsection{A low-order VEM scheme}
Let $\{\Th\}_h$ be a sequence of decompositions of $\O$ into general polygons $E$ and set 
$
h := \sup_{E \in \mathcal{T}_h} h_E$.
We denote by $\Eh$ the set of the edges of $\Th$, while $\mathcal{E}_h^{I}$ and $\mathcal{E}_h^{B}$ are the set of internal and boundary edges of the skeleton $\Eh$, respectively.
For all $h$, we say that $\{\Th\}_h$ is a regular polygonal decomposition if the following assumptions are satisfied:
\begin{itemize}
	\item $\mathbf{(A1)}$  for every edge $e\in\partial E$ we have: $h_e \ge \, \gamma \, h_E$,
	\item $\mathbf{(A2)}$  $E$ is star-shaped with respect to a ball of radius $ \ge\, \gamma \, h_E$,
\end{itemize}
where $\gamma$ is positive constant. The hypotheses above, and in particular $\mathbf{(A2)}$, may be relaxed, see Refs.~\refcite{BLRXX,BrennerSungSmallEdges,CaoChen}. Moreover, we assume that the material tensor $\D$ is piecewise constant with respect to the decomposition $\Th$. This regularity is enough for our low-order method, cf. Ref.~\refcite{ARTIOLI2017155}.
%%-----------------
%% LOCAL SPACE
%%-----------------
%

%
To describe the local spaces employed in our hybrid VEM scheme, we need to introduce these two elementary (for our scheme) spaces: $RM(E)$ and $R(e)$.
\paragraph{Space $RM(E)$.} It is the space of local infinitesimal rigid body motions:
\begin{equation}\label{eq:h_rigid}
RM(E):=\left\{ \bbr(\bbx) = \bfalpha + \beta\big(\bbx -\bbx_E\big)^{\perp}\quad \text{s.t.}\quad \bfalpha \in\R^2\text{ and } \beta\in\R \right\},
\end{equation}
where if $\bbc=(c_1,c_2)^T$ is a generic vector in $\R^2$, we denote by $\bbc^{\perp}=(c_2,-c_1)^T$ its counterclockwise rotation. The dimension of $RM(E)$ is $3$.

% indeed we need three constants, $\bfalpha$ and $\beta$, to uniquely identify a function in this space.

\paragraph{Space $R(e)$.} For each edge $e\in\partial E$, we introduce 
\begin{equation}~\label{eq:h_space_Re}
R(e)=\left\{\bfpsi(s)=c\, \bbt_e +p_1(s)\, \bbn_e\ \: \ c\in\R, \quad p_1(s)\in\P_1(e) \right\}
\end{equation}
where $\bbn_e$ is the outward normal to the edge $e$, and $\bbt_e$ is the tangent vector to the edge $e$, in accordance with its direction. The dimension of such space is $3$. 

%indeed:
%\begin{itemize}
%	\item the tangential component is determined by a single scalar value $c\in\R$;
%	\item the polynomial $p_1(s)\in\P_1(e)$ is one variable polynomial with respect to the local edge coordinate system so it is determined by two parameters, as illustrated here below:
%	\begin{equation}\label{eq:h_dofP1}
%	p_1(s):= d_1+d_2(s-s_e).
%	\end{equation}
%\end{itemize}
%Therefore, this space consists of vector functions whose tangential component is constant (first term of~\eqref{eq:h_space_Re}), while the normal component is a linear (one-variable) polynomial (the last term of~\eqref{eq:h_space_Re}).

\paragraph{Stress space.} Starting from $RM(E)$ and $R(e)$ we can define our local approximation space for the stress field:
\begin{equation}\label{eq:h_local_stress_space}
\begin{aligned}
\Sigma_{h}(E) =\{\bftau_h\  |\ \bftau_h \in H(\bdiv;E)\ :\ &\exists \bbw^\ast\in \left[H^1(E)\right]^2 \mbox{ such that } \bftau_h=\C\teps(\bbw^\ast);\\  &(\bftau_h\,\bbn)_{|e}\in R(e) \quad \forall e\in \partial  E;
\\&\bdiv\bftau_h\in RM(E)\}.
\end{aligned}
\end{equation}
%Accordingly, every local virtual function $\bfsigma_h\in\Sigma_{h}(E)$ is uniquely determined by the following degrees of freedom, see also~\eqref{eq:h_space_Re}.
%\begin{itemize}
%	\item For each edge $e$ of the element $E$, the degree of freedom which determines the tangential component of the tractions:
%	\begin{equation}
%		\bftau_h\rightarrow \int_e (\bftau_h\,\bbn)_{|e} \cdot \bbt_e~\ds.
%	\end{equation}
%	\item For each edge $e$ of the element $E$, the two degrees of freedom which determine the normal component of the tractions:
%	\begin{equation}
%		\bftau_h\rightarrow \int_e (\bftau_h\,\bbn)_{|e} \cdot p_1(s)\, \bbn_e~\ds \qquad\forall p_1(s)\in\P_1(e).
%	\end{equation}
%\end{itemize}
It is easy to see, cf. Ref. \refcite{ARTIOLI2017155}, that $\bdiv \bftau_h\in RM(E)$ is completely determined by the $(\bftau_h\,\bbn)_{|e}$'s.
%where $e$ is an edge of $E$. Indeed, denoting  $\bfvarphi: \partial E \rightarrow \R^2$ such that $\bfvarphi_{|e}=c_e\, \bbt_e +p_{1,e}(s)\, \bbn_e$, integration by parts
%\begin{equation}\label{eq:h_compat}
%\int_E \bdiv\bftau_h\cdot \bbr =
%\int_{\partial E}\bfvarphi\cdot \bbr \qquad \forall \bbr\in RM(E) ,
%\end{equation}
%%
%allows to compute $\bdiv\bftau_h$ using the information on the boundary. More precisely, setting (cf \eqref{eq:h_rigid})
%%
%\begin{equation}\label{eq:h_div1}
%\bdiv\bftau_h = \bfalpha_E + \beta_E (\bbx -\bbx_C)^\perp ,
%\end{equation}
%%
%from \eqref{eq:h_compat} we infer that the vector $\bfalpha_E$ can be computed as follows 
%%
%%
%\begin{equation}
%\bfalpha_E =\frac{1}{|E|}\int_{\partial E}\bfvarphi~\ds = \frac{1}{|E|}\sum_{e\in\partial E}\int_e (c_e\, \bbt_e +d_{1,e}\, \bbn_e)~\ds,
%\end{equation}
%while the coefficient $\beta_E$, in front of the rotational term, is given by
%\begin{equation}
%\begin{aligned}
%\beta_E &= \frac{1}{\int_E | \bbx -\bbx_E |^2~dE}\int_{\partial E} \bfvarphi\cdot (\bbx -\bbx_E)^\perp~ds\\ 
%&= \frac{1}{\int_E | \bbx -\bbx_E |^2~dE}\sum_{e\in\partial E}\int_e (c_e\bbt_e+ p_{1,e}(s)\,\bbn_e)\cdot (\bbx -\bbx_E)^\perp~\ds.
%\end{aligned}
%\end{equation}
Therefore, see Fig.~\ref{fig:dofsStandard}, we infer that the dimension of this space is
$3n_e^E$.
\begin{figure}[!ht]
	\centering
	\includegraphics[width=0.55\textwidth]{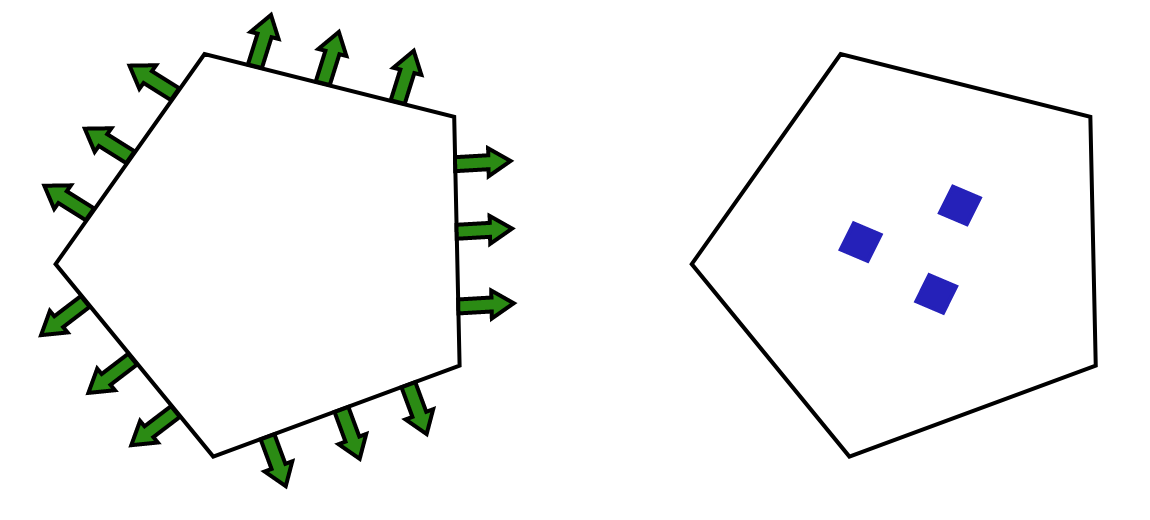}
	\caption{Schematic description of the local degrees of freedom:
		stresses (left); displacements (right).}~\label{fig:dofsStandard}
\end{figure}
\paragraph{Displacement space.}
The local approximation space for the displacement field is defined by, see~\eqref{eq:h_rigid}
\begin{equation}\label{eq:h_local_displacement_space}
U_h(E) =\left\{ \bbv_h \in \left[L^2(\E)\right]^2 \ : \ \bbv_h\in RM(E) \right\}.
\end{equation}
%Accordingly, for the local space $U_h(E)$ the following degrees of freedom can be taken:
%\begin{equation}
%\bbv_h\rightarrow \int_E \bbv_h\cdot\bbr~\dEl  \quad \forall \bbr\in RM(E)
%\end{equation}
It follows, see Fig.~\ref{fig:dofsStandard}, that
$
\text{dim}(U_h(E))=3.
$
%\begin{figure}[!ht]
%	\centering
%	\includegraphics[width=0.55\textwidth]{FiguresHybrid/LocalDofs.png}
%	\caption{Schematic description of the local degrees of freedom:
%		stresses (left); displacements (right).}~\label{fig:dofsStandard}
%\end{figure}
%% ----------------
%% The local forms
%% ----------------
% \subsubsection{The local forms}
%
We now introduce the local bilinear and linear forms involved in the method.
%
% b_E(\cdot,\cdot)
%
\paragraph{The local bilinear form $b_E(\cdot,\cdot)$.}
Given an element $E\in\Th$, we notice that, for every $\bftau_h\in\Sigma_h(E)$ and $\bbv_h\in U_h(E)$, the term
\begin{equation}
b_E(\bftau_h,\bbv_h)=\int_E \bdiv \bftau_h \cdot\bbv_h~\dEl
\end{equation}
is computable thanks to the local degrees of freedom. Therefore, there is not need to introduce any approximation of the global term $b(\bftau, \bbv)$ (hence $b_h(\cdot,\cdot)=b(\cdot,\cdot)$).
%
% a_E(\cdot,\cdot)
%
\paragraph{The local bilinear form $a_E(\cdot,\cdot)$.} 
The local bilinear form
\begin{equation}
a_E(\bfsigma_h,\bftau_h)=\int_E \D \bfsigma_h:\bftau_h~\dEl
\end{equation}
is not computable for a general couple $(\bfsigma_h,\bftau_h)\in{\Sigma}_h(E)\times{\Sigma}_h(E)$. We proceed as in the standard VEM setting. We define a suitable projection operator onto the local polynomial functions. 
We introduce $$\Pi_E:{\Sigma}_h(E)\rightarrow\left[\P_0(E)\right]_s^{2\times 2}$$ as follows
\begin{equation}~\label{eq:h_proj}
a_E(\Pi_E \bftau_h,\bfpi_0 )=a_E(\bftau_h,\bfpi_0 ) \quad \forall\bfpi_0\in\left[\P_0(E)\right]^{2\times 2}_s.
\end{equation}
Therefore, $\Pi_E$
is a projection operator onto the constant symmetric tensor functions and it is computable from the degrees of freedom.
Indeed, using the divergence theorem and the fact that each $\bfpi_0 \in \left[\P_0(E)\right]^{2\times 2}_s$ can be written as $\bfpi_0 = \C \teps(\bbp_1)$, with $\bbp_1\in \left[\P_1(E)\right]^{2}$, we rewrite the right-hand side of~\eqref{eq:h_proj} as
\begin{equation}
\begin{aligned}
a_E(\bftau_h,\bfpi_0 ) & =\int_E \D \bftau_h : \bfpi_0~\dEl\\
& =\int_E \D \bftau_h : \C \teps(\bbp_1)~\dEl
=\int_E  \bftau_h : \teps(\bbp_1)~\dEl\\
& = -\int_E  \bdiv \bftau_h\cdot \bbp_1~\dEl+\int_{\partial E} (\bftau_h\, \bbn) \cdot \bbp_1~\ds
\end{aligned}
\end{equation}
which is clearly computable. Then, the approximation of $a_E(\cdot,\cdot)$ reads: 
\begin{equation}\label{eq:h_ah1}
\begin{aligned}
a_E^h(\bfsigma_h,\bftau_h)  &:=
a_E(\Pi_E\,\bfsigma_h,\Pi_E\bftau_h) + s_E\left( (I-\Pi_E)\bfsigma_h, (I-\Pi_E)\bftau_h \right)\\
&=\int_E \D (\Pi_E\bfsigma_h) : (\Pi_E\bftau_h)~\dEl + s_E\left( (I-\Pi_E)\bfsigma_h, (I-\Pi_E)\bftau_h \right) ,
\end{aligned}
\end{equation}
where $s_E(\cdot,\cdot)$ is a symmetric and positive definite bilinear form. We propose the following choice:
\begin{equation}\label{eq:h_stab1}
s_E(\bfsigma_h,\bftau_h) : = \kappa_E\, h_E\int_{\partial E} \bfsigma_h\,\bbn\cdot \bftau_h\,\bbn~\ds,
\end{equation}
where $\kappa_E$ is a positive constant to be chosen according to $\D$. For instance, in the numerical examples of Sec.~\ref{section:hybridNumer}, $\kappa_E$ is set equal to $\frac{1}{2} {\rm tr}(\D_{|E})$. Other choice of~\eqref{eq:h_stab1} can be found in~\refcite{ARTIOLI2017155}.
%
% c(\cdot,\cdot)
%
%
% boundary terms
%
\paragraph{The loading terms.}
Let us start with the body loading term. This term can be split on each element as follows
\begin{equation}
(\bbf,\bbv_h)=\int_{\Omega}\bbf\cdot\bbv_h~\dO=\sum_{E\in\Th}\int_{E}\bbf\cdot\bbv_h~\dEl
\end{equation}
and since $\bbv_h\in RM(E)$, it is computable via quadrature rules for polygonal domains.
Similarly, the boundary term, for a sufficiently regular function $\bbg$,  can be split on each edge  $e\in\Eh^B$ as follows
%Similarly the same thing happens for the boundary term. As a matter of fact, this term can be split on each edge $e\in\Eh^B$ as follows
\begin{equation}
<\bbg,\bftau_h\, \bbn>=\int_{\partial \Omega}\bbg\cdot\bftau_h\, \bbn ~\ds=\sum_{e\in\Eh^B}\int_{e}\bbg\cdot\bftau_h\,\bbn_e~\ds.
\end{equation}
Since  $\bftau_h\in{\Sigma}_h(E)$ and in particular $\bftau_h\,\bbn_e$ is a polynomial function, this term is computable.

\paragraph{Discrete problem.} 
With all the above ingredients the discretization of Problem \eqref{problem:cont-weak} can be defined.
As for standard VEM and FEM schemes, the global spaces are built by gluing local ones and the global forms are obtained summing all the local ones. Thus, we set
\begin{equation}\label{eq:space_sigma_h}
{\Sigma}_{h}=\left\{\bftau_h\in \Sigma\ : \ \bftau_{h_{|E}}\in\Sigma_h(E) \quad \forall E\in\Th\right\}
\end{equation}
and

\begin{equation}\label{eq:space_u_h}
U_h=\left\{\bbv_h\in \left[L^2(\O)\right]^2\ : \ \bbv_{h_{|E}}\in U_h(E) \quad \forall E\in\Th\right\}.
\end{equation}

\noindent Then the discrete scheme reads:
\begin{equation}\label{problem:ourVEM} %\label{problem:descretised-formulation}
\left\lbrace{
	\begin{aligned}
	&\mbox{Find } (\bfsigma_h,\bbu_h)\in \Sigma_h\times U_h~\mbox{such that}\\
	&a_h(\bfsigma_h,\bftau_h) + b(\bftau_h, \bbu_h)=<\bbg,\bftau_h \,\bbn> &\forall \bftau_h\in \Sigma_h\\
	& b(\bfsigma_h, \bbv_h) = -(\bbf,\bbv_h) &\forall \bbv_h\in U_h.
	\end{aligned}
} \right.
\end{equation}
%We are now ready to hybridize the method.
%\subsection{Imposing $H(\bdiv)$-conformity via Lagrange multiplier}
\subsection{Imposing \texorpdfstring{$H(\bdiv)$}{Lg}-conformity via Lagrange multiplier}
We first note that the space \eqref{eq:space_sigma_h} can be considered as a subspace of the following:

\begin{equation}\label{eq:h_GlobalStressSpaceNonConforming}
\tilde{\Sigma}_{h}(\Th)=\left\{\bftau_h\in \left[L^2(\O)\right]^{2\times 2}\ : \ \bftau_{h_{|E}}\in\Sigma_h(E) \quad \forall E\in\Th\right\}.
\end{equation}
Indeed, we have:
$
\Sigma_{h}=\tilde{\Sigma}_{h}(\Th)\cap H(\bdiv,\O)$.
However, one could try to impose the conformity $\Sigma_{h}\subseteq H(\bdiv;\O)$ using Lagrange multipliers, instead of forcing the regularity directly in the subspace definition. In this VEM setting we proceed as follows.

Given $\Eh^I$, the set of the internal edges of $\Th$, we define the space of the Lagrange multipliers by (cf.~\eqref{eq:h_space_Re}):
\begin{equation}
\Lambda_h(\Eh^I):=\left\{\bfmu_h\in \left[L^2(\Eh^I) \right]^2  :\, \bfmu_{h|e}\in R(e)\quad \forall\, e\in \Eh^I \right\},
\end{equation}
where, with a little abuse of notation, we denote with $L^2(\Eh^I)$ the $L^2$ space defined on the interior skeleton of $\Th$, i.e., the union of $e\in\Eh^I$. 
We observe that the Lagrange multipliers are defined only on the internal edges $\Eh^I$ because their role will be 
to match the normal stresses insisting on interior interfaces, see Fig.~\ref{fig:dofsHybrid}. Indeed a tensor field $\bftau$ is $H(\bdiv)$-regular if its normal component $\bftau\,\bbn$ does not jump across any interface inside $\O$.
\begin{figure}[!ht]
	\centering
	\includegraphics[width=0.55\textwidth]{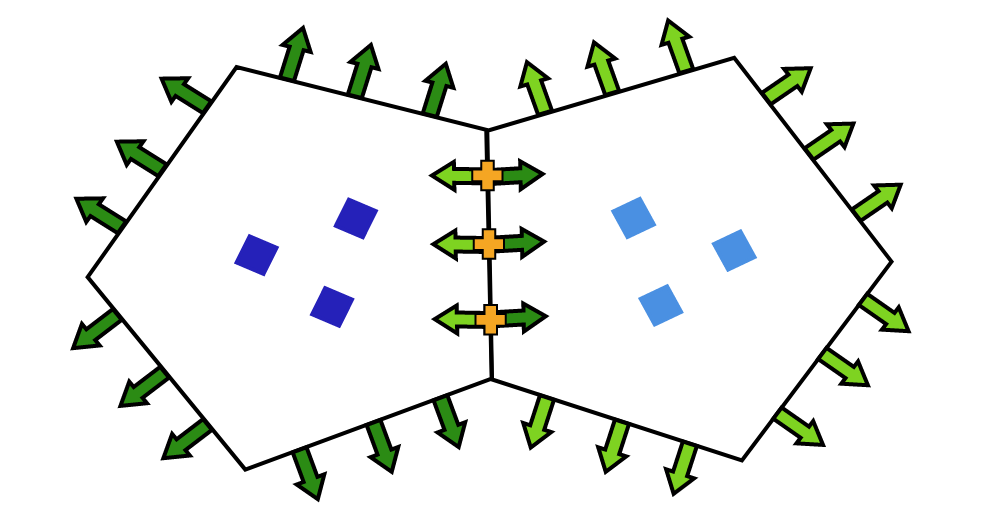}
	\caption{Overview of the degrees of freedom. The degrees of freedom are denoted as follows: square for displacement, arrow for stress and cross for Lagrange multiplier. Dark colors are referred to the left element, while light ones to the right.}~\label{fig:dofsHybrid}
\end{figure}
To force such a continuity we consider the bilinear form 
\begin{equation}
c_h(\cdot,\cdot):\tilde{\Sigma}_{h}(\Th)\times \Lambda_h(\Eh^I)\rightarrow \R
\end{equation}
defined as: 
\begin{equation}\label{eq:c_h_form}
c_h(\bftau_h,\bfmu_h):= -  \sum_{E\in\Th}\int_{\partial E^I}\bfmu_h\cdot \bftau_h\,\bbn~\ds \quad\quad \forall\bftau_h\in\tilde{\Sigma}_{h}(\Th),\  \forall\bfmu_h\in\Lambda_h(\Eh^I),
\end{equation}
where $\partial E^I=\partial E \cap \Eh^I$.
We observe that although $\bftau_h$ is virtual, such bilinear form is computable. Indeed, we are integrating over edges where both $\bfmu_h$ and $\bftau_h\,\bbn$ are polynomials.
We are now ready to set the hybrid version of Problem \eqref{problem:ourVEM}:
%We present the discrete scheme that we find once we have introduced the Lagrange multipliers into the hybrid procedure~\eqref{eq:resultingLinearSystem_matrix}. So our discrete problem is defined as follows\\
%
%
%	Find $(\bfsigma_h,\bbu_h,\bflambda_h)\in\tilde{\Sigma}_{h}(\Th)\times U_h(\Th)\times \Lambda_h(\Eh^I)$ such that
\begin{equation}~\label{HybridProblem}
\left\{
\begin{aligned}
&\mbox{Find } (\bfsigma_h,\bbu_h,\bflambda_h)\in\tilde{\Sigma}_{h}(\Th)\times U_h\times \Lambda_h(\Eh^I)~\mbox{such that}\\
&a_h(\bfsigma_h,\bftau_h) + b(\bftau_h, \bbu_h)+ c_h(\bftau_h,\bflambda_h)=<\bbg,\bftau_h\, \bbn>  &\forall \bftau_h\in \tilde{\Sigma}_h(\Th),\\
& b(\bfsigma_h, \bbv_h) = -(\bbf,\bbv_h) &\forall \bbv\in U_h,\\
& c_h(\bfsigma_h,\bfmu_h)=0  &\forall \bfmu_h\in \Lambda_h(\Eh^I).
\end{aligned}
\right.
\end{equation}
It is easy to prove that the two discrete Problems~\eqref{problem:ourVEM} and~\eqref{HybridProblem} are equivalent. In this particular case, equilavence means that if $(\bfsigma_h,\bbu_h,\bflambda_h)\in\tilde{\Sigma}_{h}(\Th)\times U_h\times \Lambda_h(\Eh^I)$ solves Problem \eqref{HybridProblem}, then  
$(\bfsigma_h,\bbu_h)\in{\Sigma}_{h}\times U_h$ and is the solution of Problem \eqref{problem:ourVEM}. Moreover if $(\bfsigma_h,\bbu_h)\in{\Sigma}_{h}\times U_h$ is the solution of Problem \eqref{problem:ourVEM}, then there is a unique $\bflambda_h\in\Lambda_h(\Eh^I)$ such that $(\bfsigma_h,\bbu_h,\bflambda_h)$ is the solution of Problem \eqref{HybridProblem}.

%The proof exploits the following Lemma, which guarantees the $H(\bdiv)$-conformity of the stress field via an appropriate choice of the multiplier space~\refcite{BoffiBrezziFortin}.
%\begin{lem}
%	If $\bftau_h\in\tilde{\Sigma}_{h}(\Th)$, then $\bftau_h\in\Sigma_h$ iff
%	\begin{equation}\label{eq:h_lemmaCondition}
%	c_h(\bftau_h,\bfmu_h) =0 \qquad \forall\bfmu_h\in\Lambda_h(\Eh^I).
%	\end{equation}
%\end{lem}
%\begin{figure}[!ht]
%	\centering
%	\includegraphics[width=0.6\textwidth]{FiguresHybrid/GlobalDofs1.png}
%	\caption{Overview of the degrees of freedom. The degrees of freedom are denoted as follows: square for displacement, arrows for stress and cross for Lagrange multipliers. Dark colors are referred to the left element, while light ones to the right.}~\label{fig:dofsHybrid}
%\end{figure}
%
%\begin{remark}
%	A reasonable choice of the space of multipliers leads to obtain the same discrete solution of the linear system~\eqref{eq:linearSystem_matrix}, although the dimensions of the vector $\tilde{\bfsigma}_h$ and $\bfsigma_h$ are obviously different.
%\end{remark}
%We notice that, in this first step, we have an increase in the dimension of our linear system, because we add a new set of unknown as the Lagrange multipliers, and we double the number of stress degrees of freedom on the internal interfaces. Therefore, if one stops now, we do not see any improvement in the hybridization technique, but on the contrary, we worsen the situation. So, the next step is fundamental.

\subsection{Static condensation of stresses and displacements}\label{ss:static_cond}
The matrix form of Problem \eqref{HybridProblem} can be written as

\begin{equation}\label{eq:resultingLinearSystem_matrix}
%\left\{
\begin{pmatrix}
\tilde{A} & \tilde{B} & \tilde{C}\\
\tilde{B}^T & O & O\\
\tilde{C}^T & O & O
\end{pmatrix}
\begin{pmatrix}
\bfsigma_h \\
{\bbu}_h   \\
{\bflambda}_h
\end{pmatrix}
=
\begin{pmatrix}
\tilde{G} \\
\tilde{F}  \\
O
\end{pmatrix}
\end{equation}
where the symbol $\sim$ here highlights that the quantity under consideration refers to the (discontinuous) space \eqref{eq:h_GlobalStressSpaceNonConforming}, rather than the conforming one \eqref{eq:space_sigma_h}. 
Moreover, we remark that the third equation 
\begin{equation}\label{eq:continuityConstraints}
\tilde{C}^T \bfsigma_h = O
\end{equation}
represents the traction continuity for the stress field by means of the multipliers. 
We observe that one of the advantages of having discontinuous stress degrees of freedom is that the matrices $\tilde{A}$ and $\tilde{B}$, corresponding to the discrete bilinear form $a_h(\cdot,\cdot)$ and the mixed term $b(\cdot,\cdot)$ are block matrices. Each block corresponds to the information of a single element in our discretization. Hence, the matrix $\tilde{A}$ is a block diagonal matrix, whose inverse can be found in a fast and cheap way. Then we can compute $\bfsigma_h$ via: 
\begin{equation}\label{eq:first_equation_of_hybrid_system}
\bfsigma_h =\tilde{A}^{-1}(\tilde{G} - \tilde{B} {\bbu}_h - \tilde{C} \bflambda_h).
\end{equation}
Subtracting~\eqref{eq:first_equation_of_hybrid_system} into the second and third equations of~\eqref{eq:resultingLinearSystem_matrix} we have
\begin{equation}\label{eq:second_third_substituting_first}
\begin{pmatrix}
\tilde{B}^T \tilde{A}^{-1} \tilde{B} & \tilde{B}^T \tilde{A}^{-1} \tilde{C}\\
\tilde{C}^T \tilde{A}^{-1} \tilde{B} & \tilde{C}^T \tilde{A}^{-1} \tilde{C} 
\end{pmatrix}
\begin{pmatrix}
{\bbu}_h   \\
\bflambda_h
\end{pmatrix}
=
\begin{pmatrix}
\tilde{B}^T \tilde{A}^{-1} \tilde{G} -\tilde{F}\\
\tilde{C}^T \tilde{A}^{-1} \tilde{G}
\end{pmatrix}
\end{equation}
which is symmetric and positive definite. \\
Now, recalling again that $\tilde{A}$ and $\tilde{B}$ are block matrices, we have that $\tilde{B}^T \tilde{A}^{-1} \tilde{B}$ is a block diagonal matrix, too. As before, it can be inverted in a straightforward way and we get
\begin{equation}~\label{eq:displacementHybridEquation}
{\bbu}_h = (\tilde{B}^T \tilde{A}^{-1} \tilde{B})^{-1}\left[(\tilde{B}^T \tilde{A}^{-1} \tilde{C})\bflambda_h +\tilde{B}^T \tilde{A}^{-1} \tilde{G} -\tilde{F} \right].
\end{equation}
Now, substituting ${\bbu}_h$ in the third equation, our system has the following form  
%We are finally left with a system of the form
\begin{equation}~\label{eq:finalLinearSystem}
H \bflambda_h = R
\end{equation}
where
\begin{equation}
H= (-\tilde{C}^T\tilde{A}^{-1}\tilde{B})(\tilde{B}^T\tilde{A}^{-1}\tilde{B})^{-1}(\tilde{B}^T\tilde{A}^{-1}\tilde{C})-\tilde{C}^T\tilde{A}^{-1}\tilde{C}
\end{equation}
and
\begin{equation}
R= -\tilde{C}^T\tilde{A}^{-1}\tilde G+(C^T\tilde{A}^{-1}\tilde{B})(\tilde{B}^T\tilde{A}^{-1}\tilde{B})^{-1}(\tilde{B}^T\tilde{A}^{-1}\tilde G-\tilde F).
\end{equation}
The matrix $H$ is symmetric and positive definite. This is an advantage from a computational viewpoint. Indeed, one can use an ``ad-hoc'' procedure to solve~\eqref{eq:finalLinearSystem}, for instance Cholesky decomposition. Once we have $\bflambda_h$, the displacement and then the stress vectors can be obtained explicitly via matrix-vector multiplication, see~\eqref{eq:displacementHybridEquation} and~\eqref{eq:first_equation_of_hybrid_system}. 
\begin{remark}
	The Lagrange multiplier field has the physical interpretation of (generalized) displacements. As we will see in Sec.~\ref{section:post-processing}, we will use it to design a higher-order (non-conforming) approximation of the displacement field.
\end{remark}
%
%%%%%%%%
%%%%%%%%%%%%%%%%%%%%%%%%%%%%%%%%%%%%%%%%%%%%%%%%%%
%%
%% Error analysis	
%%
%%%%%%%%%%%%%%%%%%%%%%%%%%%%%%%%%%%%%%%%%%%%%%%%%%%
\section{Error analysis}\label{section:h_errorAnalysis}
Since the hybridization technique of Sec.~\ref{section:hybrid} can be seen as a computational way to solve the original linear system stemming from the discrete problem \eqref{problem:ourVEM} (cf.~\eqref{eq:linearSystem_matrix}), the error estimates developed in Ref. \refcite{ARTIOLI2017155} hold also for the stress and displacement solutions of the equivalent problem \eqref{HybridProblem}. 
In particular, the following result holds true.

\begin{thm}
	Let $(\bfsigma,\bbu)\in\Sigma\times U$ be the solution of Problem~\eqref{problem:cont-weak}, and let
	$(\bfsigma_h,\bbu_h)\in\tilde\Sigma_h(\Th)\times U_h$ be the discrete stress and displacement solution of Problem~\eqref{HybridProblem}. Under assumptions $\mathbf{(A1)}$ and $\mathbf{(A2)}$ on the mesh, and supposing $(\bfsigma,\bbu)$ sufficiently regular, the following estimate holds true:
	\begin{equation}\label{theorem:standardConvergence}
	||\bfsigma-\bfsigma_h||_{\Sigma} + ||\bbu-\bbu_h||_U \lesssim h.
	\end{equation}
%	where $C$ is independent of $h$ but depends on the domain $\O$ and on the Sobolev regularity of $\bfsigma$ and $\bbu$.
\end{thm}
It remains to study the convergence to $\bbu$ of the Lagrange multipliers $\bflambda_h$ (we recall that the multipliers are physically a displacement field).
It is useful to recall, see Ref. \refcite{ARTIOLI2017155}, that there exist an interpolation operator
$$\mathcal{I}_h: W^r(\O) \rightarrow \Sigma_h$$
where
\begin{equation}
W^r(\O):=\left\{\bftau \ : \ \bftau\in \left[L^r(\O)\right]^{2\times 2}_s \quad\text{s.t.}\quad \bdiv\bftau \in \left[L^2(\O)\right]^{2}  \right\}.
\end{equation}
Such an operator is obtained by glueing the local contributions.
We define the local interpolator $\mathcal{I}_E: W^r(E)\rightarrow\Sigma_h(E)$ as
\begin{equation}\label{eq:h_interpolationOperator}
\int_{\partial E} (\mathcal{I}_E\, \bftau)\bbn \cdot \bfphi_*~\ds = \int_{\partial E} \bftau\, \bbn \cdot \bfphi_*~\ds\quad\forall \bfphi_h\in R_*(\partial E),
\end{equation}
where
%\begin{equation*}
%R_*(\partial E):=\left\{\bfphi_*\in \left[L^2(\partial E)\right]^2 \, :\, (\bfphi_*)_{|e}=\bfgamma_e+\delta_e(\bbx-\bbx_E)^{\perp} \,\, \bfgamma_e\in\R^2, \delta_e\in\R, \, \forall e \in\partial E \right\}.
%\end{equation*}
\begin{equation*}
\begin{aligned}
R_*(\partial E):=\Big\{\bfphi_*\in \left[L^2(\partial E)\right]^2 \, :\, (\bfphi_*)_{|e}=&\bfgamma_e+\delta_e(\bbx-\bbx_E)^{\perp} \,\, \\& \bfgamma_e\in\R^2, \delta_e\in\R, \, \forall e \in\partial E \Big\}.
\end{aligned}
\end{equation*}

The operator $\mathcal{I}_h$ satisfies the following commuting diagram property:
\begin{equation}~\label{eq:commutingDiagram}
\bdiv (\mathcal{I}_h \bftau) = \Pi_{RM}(\bdiv \bftau) \quad \forall \bftau \in W^r(\O),
\end{equation}
where $\Pi_{RM}$ denotes the $L^2$-projection onto the space of the rigid body motions.
Furthermore, the following error estimates hold true.
\begin{prop}~\label{prop:h_error_estimates_interpolations}
	Under the standard mesh assumptions $\mathbf{(A1)}$ and $\mathbf{(A2)}$, for the interpolation operator $\mathcal{I}_E$ defined in~\eqref{eq:h_interpolationOperator} and for each $\bftau$ sufficiently regular we have 
	\begin{equation}
	\left\{
	\begin{aligned}
	&|| \bftau -\IE\bftau||_{0,E}\lesssim h_E |\bftau|_{1,E}\\
	&|| \bdiv(\bftau -\IE\bftau)||_{0,E}  \lesssim h_E |\bdiv\bftau|_{1,E}.
	\end{aligned}
	\right.
	\end{equation}
\end{prop}
%
%% -----------
%% ERROR ESTIMATES
%% ------------
\subsection{A superconvergence result}
Henceforth, we will suppose $\Omega$ to be a convex polyogn (or a domain sufficiently regular for the application of the shift theorem); moreover, in Problem~\eqref{problem:cont-strong} we consider $\bbg=\bfzero$.
Our aim is to prove that the $L^2$-projection of $\bbu$ onto the rigid body motion,
\begin{equation}\label{eq:proj_rm}
\bar{\bbu}_h=\Pi_{RM}\bbu,
\end{equation}
superconverges to $\bbu$.
\begin{thm}\label{theorem:superconvergence}
	Let $(\bfsigma,\bbu)\in\Sigma\times U$ be the solution of Problem~\eqref{problem:cont-strong}, and let $(\bfsigma_h,\bbu_h)\in\Sigma_h\times U_h$ be the solution of the discrete Problem~\eqref{problem:ourVEM}. Then, assuming that the solution is sufficiently regular and that the mesh assumptions  $\mathbf{(A1)}$ and $\mathbf{(A2)}$ are satisfied, the following estimate holds true:
	\begin{equation}\label{eq:superconvergence}
	||\bar{\bbu}_h-\bbu_h||_0\lesssim h^2.
	\end{equation}
\end{thm}
%\coorM{Ho seguito la dimostrazione che c'è nel boffi-brezzi-fortin a pagina 432-433}
\begin{proof}
	%but in general one can use some abstract duality results,  see~\refcite{BoffiBrezziFortin}.
	Let $\bfphi\in\left[H^2(\O)\right]^2\cap \left[H^1_0(\O)\right]^2$ be the solution of the linear elasticity problem:
	\begin{equation}~\label{problem:primal}
	\left\{
	\begin{aligned}
	&\bdiv(\C \bfvarepsilon(\bfphi))=\bar{\bbu}_h-\bbu_h\qquad &\text{in } \O\\
	&\bfphi =\bfzero \qquad &\text{on } \partial \O.
	\end{aligned}
	\right.
	\end{equation}
	Due to standard regularity results ($\O$ is supposed to be convex), we have  
	\begin{equation}~\label{eq:estimate1}
	||\bfphi||_2 \lesssim ||\bar{\bbu}_h-\bbu_h||_0.
	\end{equation}
Set $\bfxi = \C \bfvarepsilon(\bfphi)$, and let $\mathcal{I}_h\,\bfxi$ be the interpolation of $\bfxi$ defined in~\eqref{eq:h_interpolationOperator}. Using~\eqref{eq:commutingDiagram}, \eqref{problem:primal} and recalling that $\left( \bar{\bbu}_h-\bbu_h\right)_{|E}\in RM(E)$, we get
	\begin{equation}~\label{eq:divInterpolateDelta}
	\bdiv (\mathcal{I}_h\,\bfxi) := \Pi_{RM}\left(\bdiv \bfxi\right)=\Pi_{RM}\left( \bar{\bbu}_h-\bbu_h\right) = \bar{\bbu}_h-\bbu_h.
	\end{equation}
	Therefore, using \eqref{eq:proj_rm} and the definition of the $L^2$-projection on rigid body motion, we have 
	\begin{equation}~\label{eq:ChainOfEquality1}
	\begin{aligned}
	||\bar{\bbu}_h-\bbu_h||_0^2
	& =  \int_{\O} \left(\bar{\bbu}_h - \bbu_h \right) \cdot \left(\bar{\bbu}_h - \bbu_h \right)~\dO 
	= \int_{\O} \bdiv (\mathcal{I}_h\,\bfxi) \cdot \left(\bar{\bbu}_h - \bbu_h \right)~\dO \\
	& = \int_{\O} \bdiv (\mathcal{I}_h\,\bfxi) \cdot \left( \bbu - \bbu_h \right)~\dO.
	\end{aligned}
	\end{equation}
	From~\eqref{problem:cont-strong}, \eqref{problem:ourVEM} and \eqref{eq:ChainOfEquality1} we infer
	\begin{equation}~\label{eq:h_usingcontinuousproblem}
	||\bar{\bbu}_h-\bbu_h||_0^2=\int_{\O} \bdiv (\mathcal{I}_h\,\bfxi) \cdot \left( \bbu - \bbu_h \right)~\dO =a_h(\bfsigma_h,\mathcal{I}_h\bfxi)-a(\bfsigma,\mathcal{I}_h\bfxi).
	\end{equation}
	Now, employing the definition of the projection operator $\Pi_E$ (cf.~\eqref{eq:h_proj}), we get
	\begin{equation}~\label{eq:T1+T2+T3}
	\begin{aligned}
	a_h&(\bfsigma_h,\mathcal{I}_h\bfxi)-a(\bfsigma,\mathcal{I}_h\bfxi)
	\\&=\sum_{E\in\Th} \left[a_E^h(\bfsigma_h, \mathcal{I}_E\bfxi)-a_E(\bfsigma,  \mathcal{I}_E\bfxi)\right]
	\\&=\sum_{E\in\Th} \left[a_E(\Pi_E\bfsigma_h, \Pi_E(\mathcal{I}_E\bfxi))-a_E(\bfsigma,  \mathcal{I}_E\bfxi)+ s_E\left( (I-\Pi_E)\bfsigma_h, (I-\Pi_E)\mathcal{I}_E\bfxi \right)\right]
	\\&=\sum_{E\in\Th} \left[a_E\left(\Pi_E\bfsigma_h,\mathcal{I}_E\bfxi\right)-a_E(\bfsigma,  \mathcal{I}_E\bfxi)+s_E\left( (I-\Pi_E)\bfsigma_h, (I-\Pi_E)\mathcal{I}_E\bfxi \right)\right]
	\\&=\sum_{E\in\Th} [a_E\left((\Pi_E-I)\bfsigma_h,\mathcal{I}_E\bfxi\right)+a_E(\bfsigma_h-\bfsigma, \mathcal{I}_E\bfxi)+s_E\left( (I-\Pi_E)\bfsigma_h, (I-\Pi_E)\mathcal{I}_E\bfxi \right)]
	\\&=T_1+T_2+T_3.
	\end{aligned}
	\end{equation}
%	\begin{equation}~\label{eq:T1+T2+T3}
%	\begin{aligned}
%	a_h&(\bfsigma_h,\mathcal{I}_h\bfxi)-a(\bfsigma,\mathcal{I}_h\bfxi)
%	\\&=\sum_{E\in\Th} \left[a_E^h(\bfsigma_h, \mathcal{I}_E\bfxi)-a_E(\bfsigma,  \mathcal{I}_E\bfxi)\right]
%	\\&=\sum_{E\in\Th} \left[a_E(\Pi_E\bfsigma_h, \Pi_E(\mathcal{I}_E\bfxi))-a_E(\bfsigma,  \mathcal{I}_E\bfxi)+ s_E\left( (I-\Pi_E)\bfsigma_h, (I-\Pi_E)\mathcal{I}_E\bfxi \right)\right]
%	\\&=\sum_{E\in\Th} \left[a_E\left(\Pi_E\bfsigma_h,\mathcal{I}_E\bfxi\right)-a_E(\bfsigma,  \mathcal{I}_E\bfxi)+s_E\left( (I-\Pi_E)\bfsigma_h, (I-\Pi_E)\mathcal{I}_E\bfxi \right)\right]
%	\\&=\sum_{E\in\Th} a_E\left((\Pi_E-I)\bfsigma_h,\mathcal{I}_E\bfxi\right)+\sum_{E\in\Th}a_E(\bfsigma_h-\bfsigma, \mathcal{I}_E\bfxi)\\&+\sum_{E\in\Th}s_E\left( (I-\Pi_E)\bfsigma_h, (I-\Pi_E)\mathcal{I}_E\bfxi \right)
%	\\&=T_1+T_2+T_3.
%	\end{aligned}
%	\end{equation}
	
	We  bound  the  three  terms $T_1$, $T_2$ and $T_3$ in~\eqref{eq:T1+T2+T3} separately. 
	
	To estimate the term $T_1$, we first write:
%	\begin{equation}
%	\begin{aligned}
%	T_1:&= \sum_{E\in\Th}a_E\left((\Pi_E-I)\bfsigma_h,\mathcal{I}_E\bfxi\right)	
%	\\&=\sum_{E\in\Th}\left[a_E\left((\Pi_E-I)(\bfsigma_h- \bfsigma),\mathcal{I}_E\bfxi\right)+a_E\left((\Pi_E-I)\bfsigma,\mathcal{I}_E\bfxi\right) \right]
%	\\&=\sum_{E\in\Th}a_E\left((\Pi_E-I)(\bfsigma_h- \bfsigma),\mathcal{I}_E\bfxi-\bfxi\right)+\sum_{E\in\Th}a_E\left((\Pi_E-I)(\bfsigma_h- \bfsigma),\bfxi-\Pi_E\bfxi\right)\\
%	&+\sum_{E\in\Th}a_E\left((\Pi_E-I)\bfsigma,\mathcal{I}_E\bfxi-\bfxi\right)+\sum_{E\in\Th}a_E\left((\Pi_E-I)\bfsigma,\bfxi-\Pi_E\bfxi\right).
%	\end{aligned}
%	\end{equation}
	\begin{equation}
	\begin{aligned}
	T_1:&= \sum_{E\in\Th}a_E\left((\Pi_E-I)\bfsigma_h,\mathcal{I}_E\bfxi\right)	
	\\&=\sum_{E\in\Th}\left[a_E\left((\Pi_E-I)(\bfsigma_h- \bfsigma),\mathcal{I}_E\bfxi\right)+a_E\left((\Pi_E-I)\bfsigma,\mathcal{I}_E\bfxi\right) \right]
	\\&=\sum_{E\in\Th}a_E\left((\Pi_E-I)(\bfsigma_h- \bfsigma),\mathcal{I}_E\bfxi-\bfxi\right)+\sum_{E\in\Th}a_E\left((\Pi_E-I)\bfsigma,\mathcal{I}_E\bfxi-\bfxi\right)\\
	&+\sum_{E\in\Th}a_E\left((\Pi_E-I)(\bfsigma_h- \bfsigma),\bfxi-\Pi_E\bfxi\right)+\sum_{E\in\Th}a_E\left((\Pi_E-I)\bfsigma,\bfxi-\Pi_E\bfxi\right).
	\end{aligned}
	\end{equation}
	Now, by employing the continuity of $a_E(\cdot,\cdot)$, standard polynomial approximation results, Proposition~\ref{prop:h_error_estimates_interpolations} and estimate~\eqref{theorem:standardConvergence}, we have
	%\begin{equation}
	%\begin{aligned}
	%&T_3\lesssim||\bfsigma_h-\bfsigma||_{0}||\mathcal{I}_h\bfxi-\bfxi||_{0}
	%\\&T_4\lesssim||\bfsigma_{\pi}-\bfsigma||_{0}||\mathcal{I}_h\bfxi-\bfxi||_{0}
	%\\&T_5\lesssim||\bfsigma_h-\bfsigma||_{0}||\bfxi-\bfxi_{\pi}||_{0}
	%\\&T_6\lesssim||\bfsigma_{\pi}-\bfsigma||_{0}||\bfxi-\bfxi_{\pi}||_{0}
	%\end{aligned}
	%\end{equation}
	%then
	\begin{equation}~\label{eq:estimate_T_1}
	\begin{aligned}
	T_1&\lesssim  \left(||\bfsigma_h-\bfsigma||_{0}+||\bfsigma - \Pi_h \bfsigma ||_{0}\right)\left(||\mathcal{I}_h\bfxi-\bfxi||_{0}+||\bfxi-\Pi_h\bfxi||_{0}\right)
	%	\\&\lesssim h\left(||\bfsigma_h-\bfsigma||_{0}+||\bfsigma_{\pi}-\bfsigma||_{0}\right) ||\bfvarphi||_2
	\\&\lesssim h\left(||\mathcal{I}_h\bfxi-\bfxi||_{0}+||\bfxi-\Pi_h\bfxi||_{0}\right)\lesssim h^2 |\bfxi|_1
	\\& \lesssim h^2||\bfvarphi||_2,
	%	 \left(||\bfsigma_h-\bfsigma||_{0}+||\bfsigma_{\pi}-\bfsigma||_{0}\right)h
	\end{aligned}
	\end{equation}
	where $\Pi_h$ is the operator that locally coincides with $\Pi_E$, for every $E\in\Th$.
	
	To  estimate the term $T_2$, we recall that $\bfxi : = \C \teps(\bfvarphi)$ to write 
	\begin{equation*}
	\begin{aligned}
	T_2:=a(\bfsigma_h-\bfsigma,  \mathcal{I}_h\bfxi)
	&=\int_{\O} \D\left(\bfsigma_h-\bfsigma\right): \mathcal{I}_h\bfxi~\dO
	\\&= \int_{\O} \D\left(\bfsigma_h-\bfsigma\right): \left(\mathcal{I}_h\bfxi-\bfxi\right)~\dO + \int_{\O} \D\left(\bfsigma_h-\bfsigma\right): \bfxi~\dO\\
	& = \int_{\O} \D\left(\bfsigma_h-\bfsigma\right): \left(\mathcal{I}_h\bfxi-\bfxi\right)~\dO + \int_{\O} \D\left(\bfsigma_h-\bfsigma\right): \C\varepsilon(\bfphi)~\dO
	\\& = \int_{\O} \D\left(\bfsigma_h-\bfsigma\right): \left(\mathcal{I}_h\bfxi-\bfxi\right)~\dO - \int_{\O} \bdiv\left(\bfsigma_h-\bfsigma\right)\cdot\bfphi~\dO,
	\end{aligned}
	\end{equation*}	
	where an integration by parts has been used in the last step.
	Now, we recall (see~\eqref{eq:commutingDiagram}) that 
	$$(\bdiv (\bfsigma_h-\bfsigma),\bbq_h)=0, \quad\forall\bbq_h\in U_h.$$
	Hence, taking $\bbq_h=\bar{\bfphi}_h :=\Pi_{RM}\bfphi$, 
	we obtain
	\begin{equation}~\label{eq:estimate2}
	T_2= \int_{\O} \D\left(\bfsigma_h-\bfsigma\right): \left(\mathcal{I}_h\bfxi-\bfxi\right)~\dO - \int_{\O} \bdiv\left(\bfsigma_h-\bfsigma\right)\cdot (\bfphi-\bar{\bfphi}_h)~\dO.
	\end{equation}
	Employing Proposition~\ref{prop:h_error_estimates_interpolations} and~\eqref{theorem:standardConvergence} we have 
	\begin{equation}~\label{eq:estimate_T_2}
	\begin{aligned}
	T_2&\lesssim ||\bfsigma_h-\bfsigma||_{\Sigma}(||\bfxi-\mathcal{I}_h\,\bfxi||_0+||\bfphi - \bar{\bfphi}_h||_0)
	\\ &\lesssim h (||\bfxi-\mathcal{I}_h\,\bfxi||_0+||\bfphi - \bar{\bfphi}_h||_0)\lesssim h^2 \left( |\bfxi|_1 + |\bfvarphi|_1\right)
	\lesssim h^2||\bfphi||_2.
	\end{aligned}
	\end{equation}
	Concerning the term $T_3$, it holds:
	\begin{equation}
	\begin{aligned}~\label{eq:TS_first}
	T_3&=\sum_{E\in\Th} s_E\left( (I-\Pi_E)\bfsigma_h, (I-\Pi_E)\mathcal{I}_E\bfxi \right)
	\\&=\sum_{E\in\Th} \kappa_E h_E \int_{\partial E} \left[(I-\Pi_E)\bfsigma_h\bbn\right] \left[(I-\Pi_E)(\mathcal{I}_E\bfxi)\bbn\right]~\ds
	\\&\lesssim \sum_{E\in\Th}h_E^{1/2}||(I-\Pi_E)\bfsigma_h\bbn||_{0,\partial E}\, h_E^{1/2} ||(I-\Pi_E)(\mathcal{I}_E\bfxi)\bbn||_{0,\partial E}.
	\end{aligned}
	\end{equation}
	Under assumption $\mathbf{(A1)}$ and $\mathbf{(A2)}$, using the same technique developed in Ref. \refcite{ARTIOLI2017155,BLRXX}, we have that
	\begin{equation}~\label{eq:estimatesForStability}
	h_E^{1/2}||\bftau_h\bbn||_{0,\partial E}\lesssim ||\bftau_h\bbn||_{-1/2,\partial E}\lesssim ||\bftau_h||_{0,E}+ h_E||\bdiv \bftau_h||_{0,E} \qquad \forall \bftau_h\in\Sigma_h(E).
	\end{equation}
	From~\eqref{eq:TS_first} and~\eqref{eq:estimatesForStability} we then deduce
	\begin{equation}~\label{eq:estimateForStability1}
	\begin{aligned}
	T_3\lesssim &\left(\sum_{E\in\Th}\left[||(I-\Pi_E)\bfsigma_h||^2_{0,E}+ h_E^2||\bdiv \bftau_h||^2_{0,E}\right] \right)^{1/2}\\ &\left(\sum_{E\in\Th}\left[||(I-\Pi_E)\mathcal{I}_E\bfxi||^2_{0,E}+ h_E^2||\bdiv (\mathcal{I}_E\bfxi)||^2_{0,E}\right]\right)^{1/2}.
	\end{aligned}
	\end{equation}
	It holds
	\begin{equation}~\label{eq:estimateForStability2}
	\begin{aligned}
	||(I-\Pi_E)\bfsigma_h||_{0,E}^2&=||(\bfsigma_h -\bfsigma) +(\bfsigma - \Pi_E\bfsigma_h)||_{0,E}^2
	\\&\lesssim||\bfsigma_h -\bfsigma||_{0,E}^2 +||\bfsigma - \Pi_E\bfsigma_h||_{0,E}^2
	\end{aligned}
	\end{equation}
	and 
	\begin{equation}~\label{eq:estimateForStability3}
	\begin{aligned}
	||(I-\Pi_E)\mathcal{I}_E\bfxi||_{0,E}^2&=||(\mathcal{I}_E\bfxi - \bfxi) + ( \bfxi -  \Pi_E( \mathcal{I}_E\bfxi) )||_{0,E}^2 
	\\& = ||(\mathcal{I}_E\bfxi - \bfxi) + ( \bfxi - \Pi_E\bfxi ) +\Pi_E(\bfxi - \mathcal{I}_E\bfxi )||_{0,E}^2
	\\&\lesssim 
	|| \mathcal{I}_E\bfxi - \bfxi ||_{0,E}^2 + || \bfxi - \Pi_E\bfxi ||_{0,E}^2 +
	|| \Pi_E(\bfxi - \mathcal{I}_E\bfxi ) ||_{0,E}^2.
	\end{aligned}
	\end{equation}
	Therefore, we use~\eqref{eq:estimateForStability2}, \eqref{eq:estimateForStability3}, the continuity of $\Pi_h$, Proposition~\ref{prop:h_error_estimates_interpolations} and~\eqref{theorem:standardConvergence}, to get
	\begin{equation}~\label{eq:estimate_T_3}
	\begin{aligned}
	T_3 &\lesssim \left( ||\bfsigma-\bfsigma_h||_{0}+||\bfsigma-\Pi_h\bfsigma||_{0}+h||\bdiv \bfsigma_h||_{0}\right)\cdot\\&\qquad\qquad\  \left(||\bfxi-\mathcal{I}_h\bfxi||_{0}+||\bfxi-\Pi_h\bfxi||_{0}+h||\bdiv(\mathcal{I}_h\bfxi)||_{0}\right)
	\\&\lesssim h^2 |\bfxi|_1 \lesssim h^2||\bfvarphi||_2.
	\end{aligned}	
	\end{equation}
	Above, we have also used the estimate $||\bdiv \bfsigma_h||_0\lesssim 1$.
	Now estimate~\eqref{eq:superconvergence} follows from~\eqref{eq:estimate1},~\eqref{eq:ChainOfEquality1},~\eqref{eq:T1+T2+T3},~\eqref{eq:estimate_T_1},~\eqref{eq:estimate_T_2} and~\eqref{eq:estimate_T_3}.
\end{proof}

\subsection{Error estimate for the Lagrangre multipliers}
The next result gives some information about the convergence of the Lagrange multipliers. To this end we introduce the following two norms on $\Lambda_h(\mathcal{E}_h^I)$:
\begin{align}
|\bfmu_h|^2_{0,h}&=\sum_{e\in\Eh^I}||\bfmu_h||^2_{0,e} ~\label{eq:boundaryNormL2}\\
|\bfmu_h|^2_{-1/2,h}&= \sum_{e\in\Eh^I}h_e\, ||\bfmu_h||^2_{0,e}.~\label{eq:boundarySeminorm}
\end{align}
%\begin{equation}\label{eq:boundaryNormL2}
%|\bfmu_h|^2_{0,h}=\sum_{e\in\Eh^I}||\bfmu_h||_{0,e}
%\end{equation}
%and
%\begin{equation}\label{eq:boundarySeminorm}
%|\bfmu_h|^2_{-1/2,h}=h_e \sum_{e\in\Eh^I}||\bfmu_h||_{0,e}.
%\end{equation}
We also need to define the $L^2$-projection operator $$\Pi_{0}^{\partial}: \left[L^2(\Eh^I)\right]^2 \rightarrow\left[\P_0(\Eh^I)\right]^2\subseteq\Lambda_h(\mathcal{E}_h^I),$$ such that 
\begin{equation}
\int_e \Pi_{0}^{\partial}\bbu\cdot \bbp~\ds =\int_e \bbu \cdot \bbp~\ds \quad \forall \bbp\in\left[\P_0(e)\right]^2, \, \forall e\in\Eh^I.
\end{equation}
%With an abuse of notation, we will use the same notation for the $L^2$-projection of the Lagrange multipliers  $\bflambda_h$ onto constant vectors.
%Now, we introduce the subspace of the constant Lagrange multipliers as follows
%\begin{equation}
%	\Lambda_h^0(\mathcal{E}_h^I):=\left\{ \bfmu_h\in\Lambda_h(\Eh^I): \bfmu_h\in\left[\mathbb{P}_0(e)\right]^2 \ \forall e\in\Eh^e \right\}\subseteq\Lambda_h(\mathcal{E}_h^I).
%\end{equation}
%Now, we introduce the following two boundary norms on $\Lambda_h(\mathcal{E}_h^I)$:
%\subsection{Boundary estimates}
%\coorM{pagina 13-14 Arnold-Brezzi sull'ibrido e dimostrazione presa dall'articolo\\}
%The aim of this subsection is now to derive some useful error bounds. Let's start to define the following two boundary norms on $\Lambda_h(\mathcal{E}_h^I)$
%
%
%The idea is to compare the constant Lagrange multipliers $\bflambda_h$ with $\P^{\partial}_{0}$ which is the orthogonal projection of $\bbu$ onto $\left[\mathbb{P}_0(e)\right]^2$
%
%Here we want to compare the Lagrange multipliers $\bflambda_h$ with  in the norm~\eqref{eq:boundaryNormL2}.
%The following results gives us an interesting estimate of the multipliers.
%
%We can prove the following result.
%
\begin{thm}~\label{theorem: boundaryEstimateL2}	For every element $E\in\Th$ and edge $e\in\partial E\cap \Eh^I$, if $\left\{ \Th \right\}_h$ is regular, it holds 
	\begin{equation}~\label{eq:boundaryEstimateL2}
	||\Pi^{\partial}_{0}(\bflambda_h-\bbu)||_{0,e}\lesssim h_E^{1/2}||\bfsigma-\bfsigma_h||_{0,E} +h_E^{-1/2}||\bar{\bbu}_h-\bbu_h||_{0,E},
	\end{equation}
where $\bar{\bbu}_h :=\Pi_{RM}\bbu$, see \eqref{eq:proj_rm}.
\end{thm}

\begin{proof}
%	Since the Lagrange multipliers are defined only on $\Eh^I$, we extend when we consider a boundary edge $e\in\Eh^B$ we take $\Pi_{0}^{\partial}(\bflambda_h - \bbu) = \bfzero$.
%	%
%	
	%
	%		 Since the Lagrange multipliers are defined only on the set of internal edges $\Eh^I$, when we consider a boundary edge $e\in\Eh^B$ we take $\Pi_{0}^{\partial}(\bflambda_h - \bbu) = \bfzero$ that, in the homogeneous case, it means  $\bflambda_h=\bfzero$. \\
	Given an element $E\in\Th$, we fix an edge $e\in\partial E\cap \Eh^I$. 
	Using the unisolvence of the degrees of freedom of $\Sigma_{h}(E)$, we infer that there exists a unique function $\tilde{\bftau}_h\in\Sigma_{h}(E)$ such that 
	\begin{equation}\label{eq:defintionBoundaryEstimate}
	\left\{
	\begin{aligned}
	\tilde{\bftau}_h \bbn_{e} &=  \Pi^{\partial}_{0}\left(\bflambda_h-\bbu\right),	& \text{on } e\\
	\tilde{\bftau}_h \bbn_{\tilde{e}} &= \bfzero & \forall \tilde{e}\neq e.
	\end{aligned}
	\right.
	\end{equation}
	Then, recalling that $\bdiv \tilde{\bftau}_h\in RM(E)$, an integration by parts, equations \eqref{eq:defintionBoundaryEstimate} and an inverse estimate for polynomials, give
	\begin{equation}\label{eq:bound_sc1}
	\begin{aligned}
	||\bdiv \tilde{\bftau}_h  ||_{0,E}^2 &= \int_E \bdiv \tilde{\bftau}_h\cdot \bdiv \tilde{\bftau}_h \dEl = \int_{\partial E} \tilde{\bftau}_h\bbn\cdot \bdiv\tilde{\bftau}_h\ds\\
	& \le 
	|| \Pi^{\partial}_{0}\left(\bflambda_h-\bbu\right)||_{0,e} || \bdiv \tilde{\bftau}_h||_{0,e}\\
	& \lesssim || \Pi^{\partial}_{0}\left(\bflambda_h-\bbu\right)||_{0,e} h_E^{-1/2}|| \bdiv \tilde{\bftau}_h||_{0,E} .
	\end{aligned}
	\end{equation}
Hence, we get
\begin{equation}\label{eq:bound_sc2}
h_E ||\bdiv \tilde{\bftau}_h  ||_{0,E} \lesssim h_E^{1/2} || \Pi^{\partial}_{0}\left(\bflambda_h-\bbu\right)||_{0,e} .
\end{equation}
Using Lemma 5.1 of Ref. \refcite{ARTIOLI2017155}, from \eqref{eq:defintionBoundaryEstimate} and \eqref{eq:bound_sc2} we obtain
	\begin{equation}~\label{eq:boundaryScalingArguments}
	h_E ||\bdiv\tilde{\bftau}_h||_{0,E} + ||\tilde{\bftau}_h||_{0,E}\lesssim  h^{1/2}_E ||\Pi^{\partial}_{0}\left(\bflambda_h-\bbu\right)||_{0,e}.
	\end{equation}
	Now, in the first equation of~\eqref{HybridProblem} we take $\bftau_h\in\tilde{\Sigma}_{h}(\Th)$ such that
	\begin{equation}
	\bftau_h=\tilde{\bftau}_h \quad \text{in } E,\quad \text{and }\quad \bftau_h=\bfzero\quad \text{in } \O\setminus E,
	\end{equation}
	and using~\eqref{eq:defintionBoundaryEstimate} we have
	\begin{equation}~\label{eq:integral1}
	\int_E \D \bfsigma_h:\tilde{\bftau}_h~\dEl+\int_E\bbu_h\cdot\bdiv\tilde{\bftau}_h~\dEl-\int_e\bflambda_h\cdot\Pi^{\partial}_{0}(\bflambda_h-\bbu)~\ds=0.
	\end{equation}
	On the other hand, employing the constitutive law in~\eqref{problem:cont-strong} and the Green's formula, we infer	\begin{equation}~\label{eq:integral2}
	\int_E \D \bfsigma :\tilde{\bftau}_h~\dEl+	\int_E \bbu\cdot \bdiv\tilde{\bftau}_h~\dEl - \int_e \bbu\cdot\Pi^{\partial}_{0}(\bflambda_h-\bbu)~\ds=0.
	\end{equation}
	Using~\eqref{eq:integral1} and~\eqref{eq:integral2} and recalling the fact that $\bdiv \tilde{\bftau}_h\in RM(E)$ we get 
	\begin{equation}~\label{eq:integralFinal}
	\begin{aligned}
	||\Pi^{\partial}_{0}(\bflambda_h-\bbu)||_{0,e}^2&=\int_e\Pi^{\partial}_{0}(\bflambda_h-\bbu)\cdot\Pi^{\partial}_{0}(\bflambda_h-\bbu)~\ds \\
	&=\int_e(\bflambda_h-\bbu)\cdot\Pi^{\partial}_{0}(\bflambda_h-\bbu)~\ds \\	
	&=\int_{E}\D (\bfsigma_h-\bfsigma):\tilde{\bftau}_h~\dEl+\int_E(\bbu_h-\bbu)\cdot\bdiv\tilde{\bftau}_h~\dEl\\
	&=\int_{E}\D (\bfsigma_h-\bfsigma):\tilde{\bftau}_h~\dEl+\int_E(\bbu_h-\bar{\bbu}_h)\cdot\bdiv\tilde{\bftau}_h~\dEl.
	\end{aligned}
	\end{equation}
	Finally~\eqref{eq:integralFinal} and~\eqref{eq:boundaryScalingArguments} give~\eqref{eq:boundaryEstimateL2}.
\end{proof}
As a consequence of the Theorem above, we have the following corollary, whose proof is immediate (cf.~\eqref{eq:boundaryEstimateL2}, \eqref{theorem:standardConvergence} and~\eqref{boundaryCorollary1}).
\begin{cor}\label{coroll:multipliers}
	For each element $E\in\Th$ and for every edge $e\in\partial E$, we have
	\begin{equation}~\label{boundaryCorollary1}
	|\Pi^{\partial}_{0}(\bflambda_h-\bbu)|_{-1/2,h}\lesssim  h||\bfsigma-\bfsigma_h||_{0,\O} +||\bar{\bbu}_h-\bbu_h||_{0,\O}  
	\end{equation}
	and  
	\begin{equation}~\label{boundaryCorollary2}
	|\Pi^{\partial}_{0}(\bflambda_h-\bbu)|_{-1/2,h}\lesssim h^2.
	\end{equation}
\end{cor}
\begin{remark}
The same results of Theorem \ref{theorem: boundaryEstimateL2} and Corollary \ref{coroll:multipliers} can be obtained replacing $\Pi_{0}^{\partial}$ with the  
$L^2$-projection operator $$\Pi_{R}^{\partial}: \left[L^2(\Eh^I)\right]^2 \rightarrow \Lambda_h(\mathcal{E}_h^I),$$ defined by (cf. \eqref{eq:h_space_Re})
\begin{equation}
\int_e \Pi_{R}^{\partial}\bbu\cdot \bbq~\ds =\int_e \bbu \cdot \bbq~\ds \quad \forall \bbq \in R(e)\ , \, \forall e\in\Eh^I .
\end{equation}

\end{remark}

%%%%%%%%%%%%%%%%%%%%%%%%%%%%%%%%%%%%%%%%%%%%%%%%%%%%%%%%
%%%
%%%	Post-processing
%%%
%%%%%%%%%%%%%%%%%%%%%%%%%%%%%%%%%%%%%%%%%%%%%%%%%%%%%%%%
\section{Post-processing}~\label{section:post-processing}	
In the present section, we introduce a post-processing procedure which leads to achieve a better approximation for the displacement field. More precisely, we will employ the Lagrange multipliers $\bflambda_h$ to construct a non-conforming VEM approximation $\bbu^*_h$ converging to $\bbu$ faster than $\bbu_h$. 

Let us start to present the non-conforming VEM spaces, see Ref.~\refcite{AyusoLipnikovManzini} for more details.
% In order to define $\bbu_h^*$ we need to introduce the following non-conforming VEM framework, see~\refcite{AyusoLipnikovManzini} for more details.
\subsection{Non-conforming Sobolev spaces}
Given $\{\Th\}_h$, a sequence of regular decomposition of $\O$, we define the broken $H^1$ space on $\Th$
as
\begin{equation}
H^1(\Th):= \prod_{E\in\Th} H^1(E)=\left\{v \in L^2(\O) : v_{|E}\in H^1(E) \right\}.
\end{equation}
Then, in particular
\begin{equation}\label{eq:brokenVectorSpace}
\left[H^1(\Th)\right]^2:=\prod_{E\in\Th} \left[ H^1(E)\right]^2
\end{equation}
is the space of vector-valued functions that, are locally in $\left[ H^1(E)\right]^2$.
For the vector space~\eqref{eq:brokenVectorSpace}, we introduce the corresponding broken seminorm and norm
\begin{equation}
|\bbv|^2_{1,\Th}:=\sum_{E \in \Th}||\nabla\bbv||^2_{0,E}, \qquad ||\bbv||^2_{1,\Th}:=\sum_{E \in \Th}||\bbv||^2_{1,E}.
\end{equation}

In order to define non-conforming Sobolev spaces associated with a polygonal decomposition, we need to fix some additional notation. Let $e$ be an edge in $\Eh^I$. Then, there are two adjacent elements $E^\pm$ which share the same edge $e$. We write $\bbn_{E^{+}}$, $\bbn_{E^{-}}$ for the exterior unit normal on $\partial E^{+}$ and $\partial E^{-}$, respectively. Then, for $\bbv\in \left[H^1(\Th)\right]^2$, we define the jump operator across an edge $e\in\Eh$ as 
\begin{equation}
\ljump \bbv \rjump:=\left\{
\begin{aligned}
&\bbv^{+}\otimes\bbn_{E^+}+\bbv^{-}\otimes\bbn_{E^-} \quad &\text{ on } e\in \Eh^I\\
&\bbv\otimes\bbn_{e}  & \text{ on }e\in \Eh^B,
\end{aligned}
\right.
\end{equation}
where $\otimes$ denotes the usual tensor product of vectors.
We now introduce the global non-conforming $H^1$ space as follows
%	Finally, always considering homogeneous boundary conditions, we introduce the global non-conforming $H^1$ space as follows
\begin{equation}
H^{1,nc}_0(\Th):=\left\{\bbv \in \left[H^1(\Th)\right]^2 \,: \, \int_e\ljump \bbv \rjump~\ds =0 \quad  \forall e \in\Eh  \right\}.
\end{equation}
We remark that the seminorm $|\cdot|_{1,\Th}$ is a norm for functions in $H^{1,nc}_0(\Th)$ and that the following Poincar\'e inequality holds true (see Refs.~\refcite{AyusoLipnikovManzini,ncHVEM}):
\begin{equation}
||\bbv||_{0}\lesssim |\bbv|_{1,\Th} \quad \forall \bbv \in H^{1,nc}_0(\Th).
\end{equation}
\subsection{A low-order non-conforming Virtual Element Method}
We briefly recall the main features of the low-order non-conforming VEM studied in Refs.~\refcite{AyusoLipnikovManzini,ncHVEM}. Given a polygon $E\in\Th$, we define the local non-conforming virtual space as  
\begin{equation}~\label{eq:nonConformingSpace}
U_h^*(E):=\left\{ \bbv_h^*\in \left[H^1(E)\right]^{2} \, :\, \dfrac{\partial\bbv_h^*}{\partial \bbn} = \nabla\bbv_h^*\bbn  \in \left[\P_0(e)\right]^2\quad \forall e\in\partial E, \quad \Delta \bbv_h^* = \bfzero \right\}.
\end{equation}
Accordingly, for the local spaces $U_h^*(E)$, we can take the following degrees of freedom:
%
%
%
%For $0 \leq l\leq 1$, we denote by $\mathcal{M}^l(E)$, the space of scaled monomials up to degree $l$ on $E$, defined by
%\begin{equation}
%	\mathcal{M}^l(E):=\left\{ \left(\frac{\bbx -\bbx_E}{h_E}\right)^s, \quad |s|\leq l \right\}.
%\end{equation}.
%In $U_h^*$ we can choose the following degrees of freedom for each edge $e\in\Eh^I$:
\begin{equation}~\label{eq:nonConformingDofs}
\bbv^*_h \rightarrow \frac{1}{|e|}\int_e \bbv^*_h~\ds.
\end{equation}
Therefore, we infer that the dimension of space~\eqref{eq:nonConformingSpace} is
\begin{equation}
\dim(U_h^*(E)) = 2n_e^E,
\end{equation}
where we recall that $n_e^E$ is the number of element edges.
The unisolvence of the degrees of freedom defined in~\eqref{eq:nonConformingDofs} is given by the following proposition, whose proof can be found in Ref.~\refcite{AyusoLipnikovManzini}.
\begin{prop}~\label{prop:h_unisolvenceDofsNC}
	Let $E$ be a simple polygon with $n_e^E$ edges, and let $U^*_h(E)$ be the space defined in~\eqref{eq:nonConformingSpace}. The degrees of freedom~\eqref{eq:nonConformingDofs} are unisolvent for $U^*_h(E)$.
\end{prop}
The global non-conforming virtual element space is given by
\begin{equation}
U_h^*(\Th):=\left\{\bbv^*_h\in H^{1,nc}_0(\Th) \, :\, \bbv^*_{h_{|E}}\in U_h^*(E) \quad \forall\ E\in\Th  \right\}.
\end{equation}
We also need to recall the projection operator $\Pi^{\nabla}: \left[H^1(E)\right]^2 \rightarrow \left[\mathcal{P}_1(E)\right]^2$, defined by 
\begin{equation}
\begin{aligned}
\label{eq:h_projection_post-processing}
& \int_E \nabla(\Pi^{\nabla}\bbv^*_h): \nabla \bbq~\dEl =\int_E \nabla\bbv^*_h:\nabla \bbq~\dEl \quad \forall \bbq\in\left[\P_1(E)\right]^2\\
&
\int_{\partial E} \Pi^{\nabla}\bbv^*_h~\dEl =\int_{\partial E} \bbv^*_h~\dEl.
\end{aligned}
\end{equation}

Furthermore, the following estimates will be useful in the sequel.

\begin{prop}\label{pr:investBroken}
	Under assumptions $\mathbf{(A1)}$ and $\mathbf{(A2)}$, for every $E\in\Th$ and every $\bbv_h^*\in U^*_h(E)$, it holds
	
	\begin{equation}\label{eq:invest_broken}
	| \bbv_h^* |_{1,E} \lesssim h_E^{-1}  || \bbv_h^* ||_{0,E} 
	\end{equation}
	and
	
	\begin{equation}\label{eq:invest_l2}
	|| \bbv_h^* ||_{0,E} \lesssim h_E^{1/2}  || \Pi^\partial_0 \bbv_h^* ||_{0,\dE} .
	\end{equation}

\end{prop}

\begin{proof}
	We first notice that, since $\bbv_h^*\in U_h^*(E)$ is harmonic in $E$, we have
	
	\begin{equation}\label{eq:inverse1}
	| \bbv_h^* |_{1,E}^2 = \int_{\partial E} \nabla \bbv_h^*\bbn\cdot \bbv_h^* \ds \leq || \nabla\bbv_h^*\bbn||_{0,\partial E}\, || \bbv_h^*||_{0,\partial E}.
	\end{equation}
	Recalling that $(\nabla\bbv_h^*\bbn)_{|\partial E}$ is a piecewise constant vectorial function, under assumptions $\mathbf{(A1)}$ and $\mathbf{(A2)}$, the 1D inverse estimate 	
	$$
	|| \nabla\bbv_h^*\bbn||_{0,\partial E}\lesssim h_E^{-1/2} || \nabla\bbv_h^*\bbn||_{-1/2,\partial E}
	$$
	holds true. Therefore, we get (cf. Ref. \refcite{ARTIOLI2017155} and recall again that $\bdiv\nabla\bbv_h^* = 0$)
	
	\begin{equation}\label{eq:inverse2}
	|| \nabla\bbv_h^*\bbn||_{0,\partial E}\lesssim h_E^{-1/2} || \nabla\bbv_h^*||_{0,E} = h_E^{-1/2} | \bbv_h^*|_{1,E}.
	\end{equation}
	Hence, from \eqref{eq:inverse1} we get
	
	\begin{equation}\label{eq:inverse3}
	| \bbv_h^* |_{1,E} \lesssim h_E^{-1/2} || \bbv_h^*||_{0,\partial E}.
	\end{equation}
	We then exploit a scaled trace inequality, see for instance Ref. \refcite{Brenner-Scott:2008}, to infer that it holds
	
	\begin{equation}\label{eq:inverse4}
	| \bbv_h^* |_{1,E} \lesssim h_E^{-1/2} || \bbv_h^*||_{0,E}^{1/2} \left(  | \bbv_h^* |_{1,E}^2 + h_E^{-2} || \bbv_h^*||_{0,E}^2  \right)^{1/4}.
	\end{equation}
	Hence, we get
	
	\begin{equation}\label{eq:inverse5}
	| \bbv_h^* |_{1,E} \lesssim h_E^{-1/2} || \bbv_h^*||_{0,E}^{1/2} \,  | \bbv_h^* |_{1,E}^{1/2} + h_E^{-1} || \bbv_h^*||_{0,E}.
	\end{equation}
	Using the Young's inequality, we obtain
	\begin{equation}\label{eq:inverse6}
	| \bbv_h^* |_{1,E} \lesssim
	\frac{1}{2\delta}
	h_E^{-1} || \bbv_h^*||_{0,E} + \frac{\delta}{2}  | \bbv_h^* |_{1,E} + h_E^{-1} || \bbv_h^*||_{0,E},
	\end{equation}
	where $\delta > 0$ is at our disposal. We now choose $\delta$ sufficiently small to absorb in the left-hand side the second term of the right-hand side, and thus get \eqref{eq:invest_broken}.
	
	To prove \eqref{eq:invest_l2}, we first split $\bbv_h^*\in U^*_h(E)$ as
	\begin{equation}\label{eq:split_l2-1}
	\bbv_h^* = (\bbv_h^* - \bar\bbv_h^*) + \bar\bbv_h^* = \bbw_h^* + \bar\bbv_h^* ,
	\end{equation}
	where the constant vector $\bar\bbv_h^*$ is defined by 
	$$
	\bar\bbv_h^* = \frac{1}{|\partial E|}\int_{\partial E} \bbv_h^*\ds.
	$$
	and $\bbw_h^* : = \bbv_h^* - \bar\bbv_h^*$.
	Then, a direct computation shows that
	\begin{equation}\label{eq:invest_l2_2}
	|| \bbv_h^* ||_{0,E} \leq || \bbw_h^* ||_{0,E} + || \bar\bbv_h^* ||_{0,E} \lesssim || \bbw_h^* ||_{0,E} + h_E^{1/2}  ||  \bar \bbv_h^* ||_{0,\dE} .
	\end{equation}
	To estimate $|| \bbw_h^* ||_{0,E}$, we notice that $\bbw_h^*$ has zero mean value on $\dE$. Therefore, a Poincar\'e-type estimate gives, see for instance Ref. \refcite{NazarovRepin}:
	\begin{equation}\label{eq:invest_l2_3}
	|| \bbw_h^* ||_{0,E} \lesssim h_E  | \bbw_h^* |_{1,E} .
	\end{equation}
	Using that $\nabla\bbw_h^*\bbn$ is piecewise constant on $\dE$, we get (cf. also \eqref{eq:inverse2})
	\begin{equation}\label{eq:invest_l2_4}
	\begin{aligned}
	| \bbw_h^* |_{1,E}^2 =
	\int_\dE \nabla\bbw_h^*\bbn \cdot \bbw_h^*\ds &=  \int_\dE \nabla\bbw_h^*\bbn \cdot \Pi^\partial_0\bbw_h^*\ds \leq
	|| \nabla\bbw_h^*\bbn ||_{0,\dE} || \Pi^\partial_0\bbw_h^* ||_{0,\dE}\\
	& \lesssim
	h_E^{-1/2} || \Pi^\partial_0\bbw_h^* ||_{0,\dE}
	| \bbw_h^* |_{1,E} .
	\end{aligned}
	\end{equation}
	Therefore, we obtain
	
	\begin{equation}\label{eq:invest_l2_5}
	| \bbw_h^* |_{1,E} \lesssim
	h_E^{-1/2} || \Pi^\partial_0\bbw_h^* ||_{0,\dE}
	.
	\end{equation}
	Combining \eqref{eq:invest_l2_2}, \eqref{eq:invest_l2_3} and \eqref{eq:invest_l2_5}, we infer
	
	\begin{equation}\label{eq:invest_l2_6}
	|| \bbv_h^* ||_{0,E} \lesssim h_E^{1/2} \left( || \Pi^\partial_0\bbw_h^* ||_{0,\dE} +  || \bar\bbv_h^* ||_{0,\dE} \right).
	\end{equation}
	We now notice that, since
	
	$$
	\int_\dE \Pi^\partial_0\bbw_h^* \cdot \bar\bbv_h^* \ds = 0,
	$$
	it holds 
	
	\begin{equation}\label{eq:invest_l2_7}
	|| \Pi^\partial_0\bbw_h^* ||_{0,\dE} +  || \bar\bbv_h^* ||_{0,\dE} \lesssim || \Pi^\partial_0\bbw_h^* + \bar\bbv_h^* ||_{0,\dE} = 
	|| \Pi^\partial_0\bbv_h^* ||_{0,\dE} . 
	\end{equation}
	Now estimate \eqref{eq:invest_l2} follows from \eqref{eq:invest_l2_6} and \eqref{eq:invest_l2_7}.
\end{proof}

We are ready to prove the following convergence result for a suitable non-conforming post-processed displacement field. 

\begin{thm}~\label{theorem:post-processing_result}
	Let $(\bfsigma,\bbu)$ be the solution of continuous Problem~\eqref{problem:cont-strong} and  $(\bfsigma_h,\bbu_h,\bflambda_h)$ be the discrete solution of Problem~\eqref{HybridProblem}.
	Define $\bbu_h^* \in U_h^*(\Th)$ such that it holds:
	\begin{equation}~\label{eq:projectionOnBoundaryNonConformingApproximation}
	\Pi_{0}^{\partial} (\bbu_h^*- \bflambda_h)=0.
	\end{equation}
	Then we have
	\begin{equation}~\label{eq:inequalityMainTheorem}
	||\bbu-\bbu_h^*||_{0}\lesssim  h^{2}.
	%	||\bbu-\bbu_h^*||_{0}\leq  C h^{2} \left(||\bbu||_{3}+||\bfsigma||_{1} \right)
	\end{equation}
	In addition, if the family of meshes $\left\{ \Th \right\}_h$ is also quasi-uniform, it holds 
	\begin{equation}~\label{eq:H1inequalityMainTheorem}
	|\bbu-\bbu_h^*|_{1,\Th}\lesssim  h.
	%	||\bbu-\bbu_h^*||_{0}\leq  C h^{2} \left(||\bbu||_{3}+||\bfsigma||_{1} \right)
	\end{equation}
	
\end{thm}
\begin{proof}
	For the displacement field $\bbu$, we define the non-conforming interpolant $\tilde{\bbu}_h^*\in U_h^*(\Th)$
	imposing:
	\begin{equation}~\label{eq:projectionOnBoundaryNonConformingInterpolant}
	\Pi^{\partial}_{0}(\tilde{\bbu}^*_h - \bbu)= \bfzero
	\end{equation}
	for each edge $e\in\Eh$. Due to Proposition~\ref{prop:h_unisolvenceDofsNC}), $\tilde{\bbu}_h^*$ is well-defined. Similarly, $\bbu_h^* \in U_h^*(\Th)$ is well-defined by \eqref{eq:projectionOnBoundaryNonConformingApproximation}.
	Writing now 
	\begin{equation}
	\bbu -\bbu_h^* =(\bbu - \tilde{\bbu}_h^*)+ (\tilde{\bbu}_h^*-\bbu_h^*)
	\end{equation}
	and using the triangle inequality, we have
	\begin{equation}\label{eq:split_tria}
	||\bbu-\bbu_h^*||_{0}\leq 	||\bbu-\tilde{\bbu}_h^*||_{0} +	||\tilde{\bbu}_h^*-\bbu_h^*||_{0}.
	\end{equation}
	By standard arguments, see Refs.~\refcite{AyusoLipnikovManzini,Brenner-Scott:2008}, we get
	\begin{equation}\label{eq:errorNC}
	||\bbu -\tilde{\bbu}_h^*||_0 \lesssim h^2 . % |\bbu|_2.
	\end{equation}
	To estimate $||\tilde{\bbu}_h^*-\bbu_h^*||_{0}$, we notice that from~\eqref{eq:projectionOnBoundaryNonConformingApproximation} and~\eqref{eq:projectionOnBoundaryNonConformingInterpolant}, we have 
	\begin{equation}\label{eq:stand}
	\Pi_{0}^{\partial}\left(\bbu_h^*-\tilde{\bbu}_h^*\right)=\Pi_{0}^{\partial}(\bflambda_h-\bbu).
	\end{equation}
	Fix an element $E\in\Th$; due to estimate \eqref{eq:invest_l2} of Proposition \ref{pr:investBroken} and to \eqref{eq:stand}, we get
	\begin{equation}~\label{eq:scalingArguments}
	||\bbu_h^*-\tilde{\bbu}_h^*||_{0,E}\lesssim  h_E^{1/2} ||\Pi_{0}^{\partial}(\bbu_h^*-\tilde{\bbu}_h^*)||_{0,\partial E}
	=  h_E^{1/2} ||\Pi_{0}^{\partial}(\bflambda_h-\bbu)||_{0,\partial E} .
	\end{equation}
	Summing all the local estimates \eqref{eq:scalingArguments} and combining with Corollary \ref{coroll:multipliers}, we get 
	\begin{equation}\label{eq:errrorDiscrNC}
	||\bbu_h^*-\tilde{\bbu}_h^*||_{0}\lesssim  h^2 .  %(||\bbu||_3+||\bfsigma||_2).
	\end{equation}
	Estimate \eqref{eq:inequalityMainTheorem} now follows from \eqref{eq:split_tria}, \eqref{eq:errorNC} and \eqref{eq:errrorDiscrNC}.
	To prove \eqref{eq:H1inequalityMainTheorem}, we observe that 
	\begin{equation}\label{eq:H1split_tria}
	|\bbu-\bbu_h^*|_{1,\Th}\leq 	|\bbu-\tilde{\bbu}_h^*|_{1,\Th} +	|\tilde{\bbu}_h^*-\bbu_h^*|_{1,\Th}.
	\end{equation}
	By standard arguments, we have
	\begin{equation}\label{eq:H1_1}
	|\bbu-\tilde{\bbu}_h^*|_{1,\Th}
	\lesssim h.
	\end{equation}
	Using the inverse estimate \eqref{eq:invest_broken} of Proposition \ref{pr:investBroken}, we get
	\begin{equation}\label{eq:H1_2}
	\begin{aligned}
	|\tilde{\bbu}_h^*-\bbu_h^*|_{1,\Th} = \left(\sum_{E\in\Th} |\tilde{\bbu}_h^*-\bbu_h^*|_{1,E}^2\right)^{1/2} &\lesssim 
	\left(\sum_{E\in\Th} h_E^{-2}||\tilde{\bbu}_h^*-\bbu_h^*||_{0,E}^2\right)^{1/2}\\
	&\lesssim h^{-1}||\tilde{\bbu}_h^*-\bbu_h^*||_{0},
	\end{aligned}
	\end{equation}
	where in the last step we have used that the family of meshes is quasi-uniform. Since 
	$$
	||\tilde{\bbu}_h^*-\bbu_h^*||_{0}\le ||\tilde{\bbu}_h^*-\bbu||_{0} + ||\bbu-\bbu_h^*||_{0},
	$$
	from \eqref{eq:stand} and \eqref{eq:inequalityMainTheorem}, estimate \eqref{eq:H1_2} leads to
	\begin{equation}\label{eq:H1_3}
	|\tilde{\bbu}_h^*-\bbu_h^*|_{1,\Th}
	\lesssim h.
	\end{equation}
	Estimate \eqref{eq:H1inequalityMainTheorem} now follows from \eqref{eq:H1split_tria}, \eqref{eq:H1_1} and \eqref{eq:H1_3}.
\end{proof}
%%%%%%%%%%%%%%%%%%%%%%%%%%%%%%%%%%%%%%%%%%%%%%%%%%
%%%
%%% 	Numerical Results
%%%
%%%%%%%%%%%%%%%%%%%%%%%%%%%%%%%%%%%%%%%%%%%%%%%%%%
\section{Numerical Results}~\label{section:hybridNumer}
In this section we validate the proposed VEM hybridized approach through some numerical experiments. We first give numerical evidence of the theoretical results. Then, we compare the solving time of the conforming and hybridized VE method, showing the better performace of this latter procedure, especially for the 3D case.
We will consider the following two test problems.
\paragraph{Test case 2D.}
Given $\O_1=[0,1]^2$ the unit square, we consider the following analytical solution
\begin{eqnarray}
\bbu:=
\left(
\begin{array}{l}
0.5 (\sin(2 \pi x))^2\sin(2 \pi y) \cos(2 \pi y)\\
-0.5 (\sin(2 \pi y))^2 \sin(2 \pi x)\cos(2 \pi x)
\end{array}
\right).
\end{eqnarray}
The loading term $\bbf$ is computed accordingly. For this problem we consider a homogeneous and isotropic material with Lam\'e coefficients $\lambda = 10^5$ and $\mu=0.5$ (nearly incompressible material). 
\paragraph{Test case 3D.}
Give the unit cube $\O_2=[0,1]^3$, we consider a 3D elastic problem with the following exact displacement solution and load term:
\begin{equation}
\left\{
\begin{aligned}
&u_1 = u_2 = u_3 = 10 S(x, y, z)\\
&f_1 = -10\pi^2 ((\lambda + \mu) cos(\pi x) sin(\pi y + \pi z) - (\lambda + 4\mu)S(x, y, z))\\
&f_2 = -10\pi^2 ((\lambda + \mu) cos(\pi y) sin(\pi x + \pi z) - (\lambda + 4\mu)S(x, y, z))\\
&f_3 = -10\pi^2 ((\lambda + \mu) cos(\pi z) sin(\pi x + \pi y) - (\lambda + 4\mu)S(x, y, z))
\end{aligned}
\right.
\end{equation}
where $S(x, y, z)=\sin(\pi x)\sin(\pi y)\sin(\pi z)$. In this case, we consider a compressible material where the Lam\'e constants are $\lambda=1$ and $\mu=1$.
%%
%% mesh
%%
\paragraph{Mesh.}
In order to test our problems we consider two packages of meshes of four types each, see Figure~\ref{fig:h_meshes}:
\begin{itemize}
	\item \textbf{2D meshes}: the unit square $\O_1$ is discretized as follows : i) \texttt{Square}, a uniform mesh composed by standard structured squares; ii)  \texttt{Tria}, a Delanuay triangolation of the domain $\O_1$~\refcite{Triangle}; iii) \texttt{CVT}, a centroidal Voronoi tessellation~\refcite{Du:Faber99} generated with \textbf{Polymesher}~\refcite{TPPM12}; iv) \texttt{Rand}, random polygons.
	\item \textbf{3D meshes}: for the unit cube $\O_2$, we take: a) \texttt{Cube}, a uniform mesh composed by standard structured cubes; b)  \texttt{Tetra}, a Delanuay tetrahedralization of the domain $\O_2$~\refcite{tetgen}; c) \texttt{CVT}, a centroidal Voronoi tessellation~\refcite{Du:Faber99}; d) \texttt{Rand}, random polyhedra thanks to Voronoi tessellation achieved with random control points.
\end{itemize}
We remark that the meshes \texttt{CVT} and \texttt{Rand} have interesting features which challenge the robustness of the virtual element approach. Indeed, they could have some elements
with tiny faces and edges, and we remark that such a situation is not covered by the developed theory, i.e., the assumptions $\mathbf{(A1)}$ and $\mathbf{(A2)}$ for the two-dimensional case are not both satisfied (the same occurs for the 3D case).
In order to assess the convergence rate, for each type of mesh, we define the following mesh-size $h$:
$$
h:= \frac{1}{N_E}\sum_{i=1}^{N_E}h_E
$$
where we recall that $N_E$ is the number of elements in the mesh, and $h_E$ is the diameter of the polytopal element $E$.
\newcommand{\sizeGraphHybrid}{0.35}
\newcommand{\sizeMeshHybrid}{0.24}
\begin{figure}[!ht]
	\centering
	\renewcommand{\thesubfigure}{}
	\subfigure[i) \texttt{Square}]{\includegraphics[width=\sizeMeshHybrid\textwidth,trim = 0mm 0mm 0mm 0mm, clip]{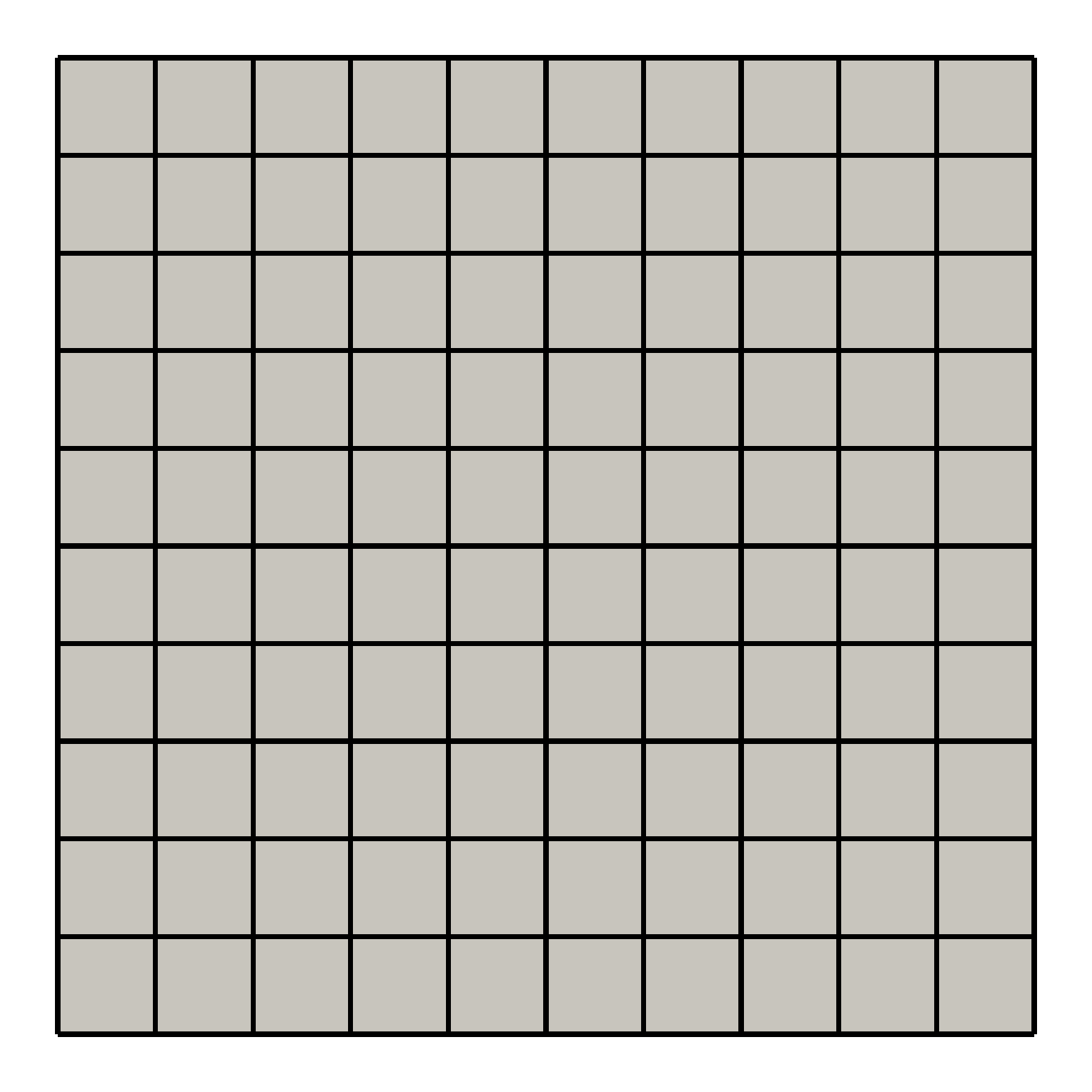}}
	\subfigure[ii) \texttt{Tria}]{\includegraphics[width=\sizeMeshHybrid\textwidth,trim = 0mm 0mm 0mm 0mm, clip]{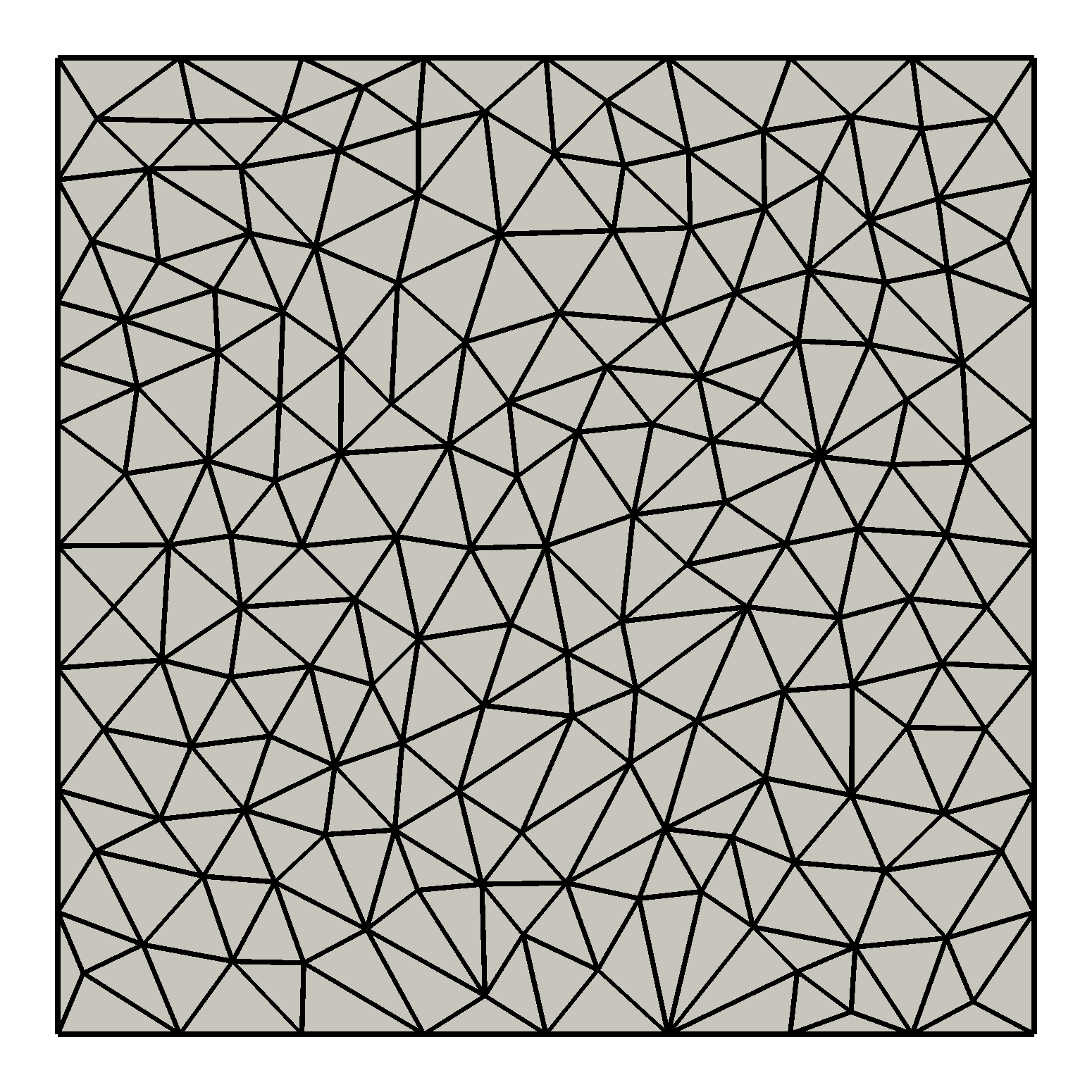}}
	\subfigure[iii) \texttt{CVT}]{\includegraphics[width=\sizeMeshHybrid\textwidth,trim = 0mm 0mm 0mm 0mm, clip]{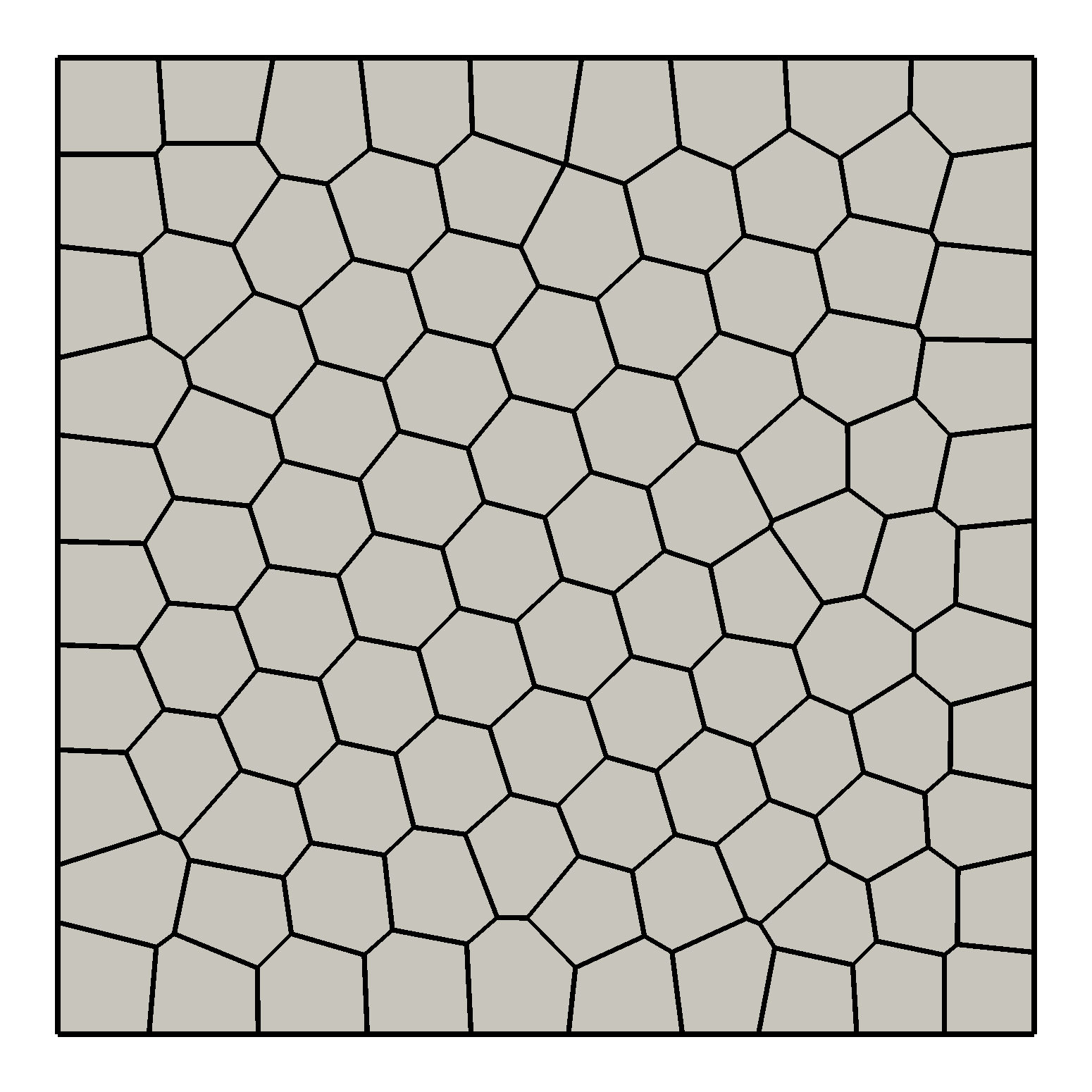}}
	\subfigure[iv) \texttt{Rand}]{\includegraphics[width=\sizeMeshHybrid\textwidth,trim = 0mm 0mm 0mm 0mm, clip]{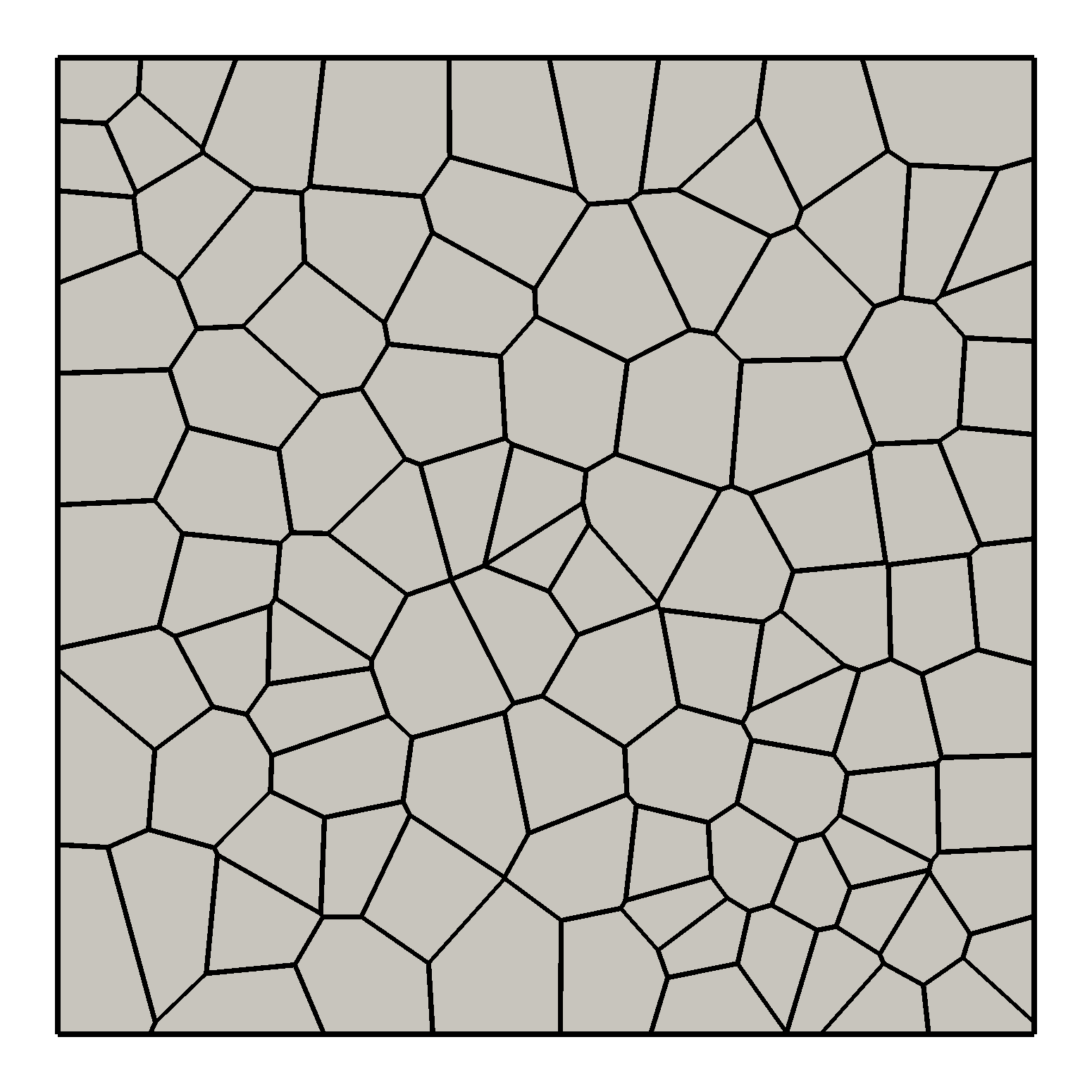}}\\
	\subfigure[a) \texttt{Cube}]{\includegraphics[width=\sizeMeshHybrid\textwidth,trim = 0mm 0mm 0mm 0mm, clip]{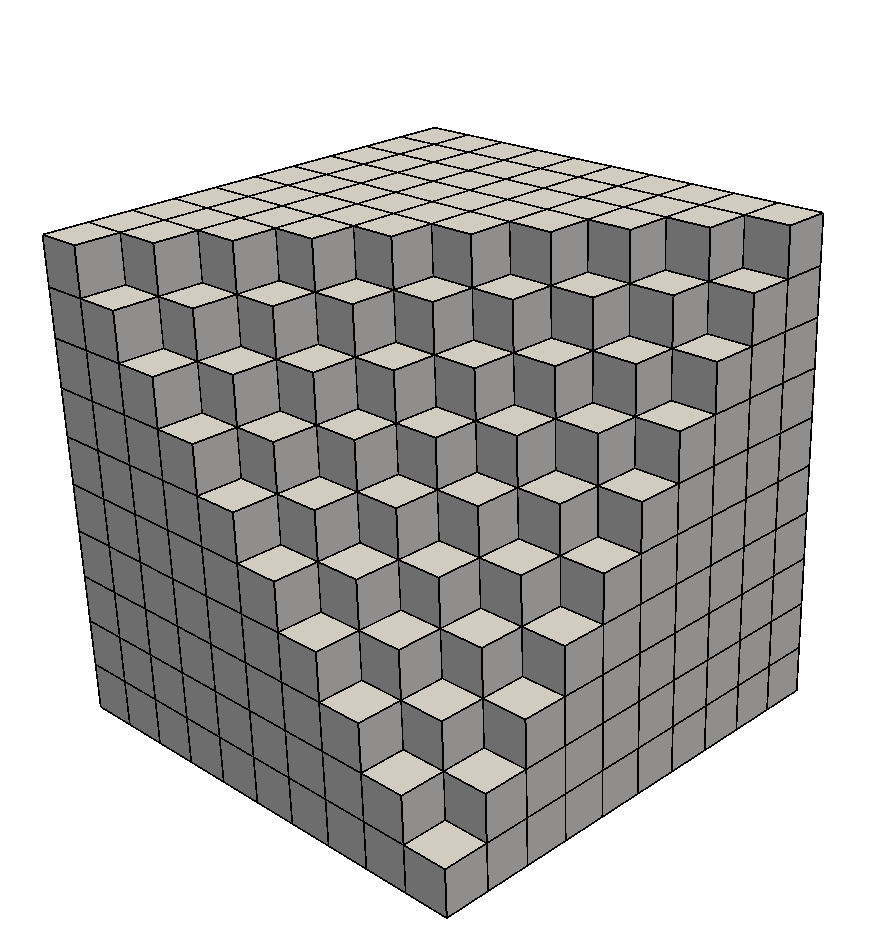}}
	\subfigure[b) \texttt{Tetra}]{\includegraphics[width=\sizeMeshHybrid\textwidth,trim = 0mm 0mm 0mm 0mm, clip]{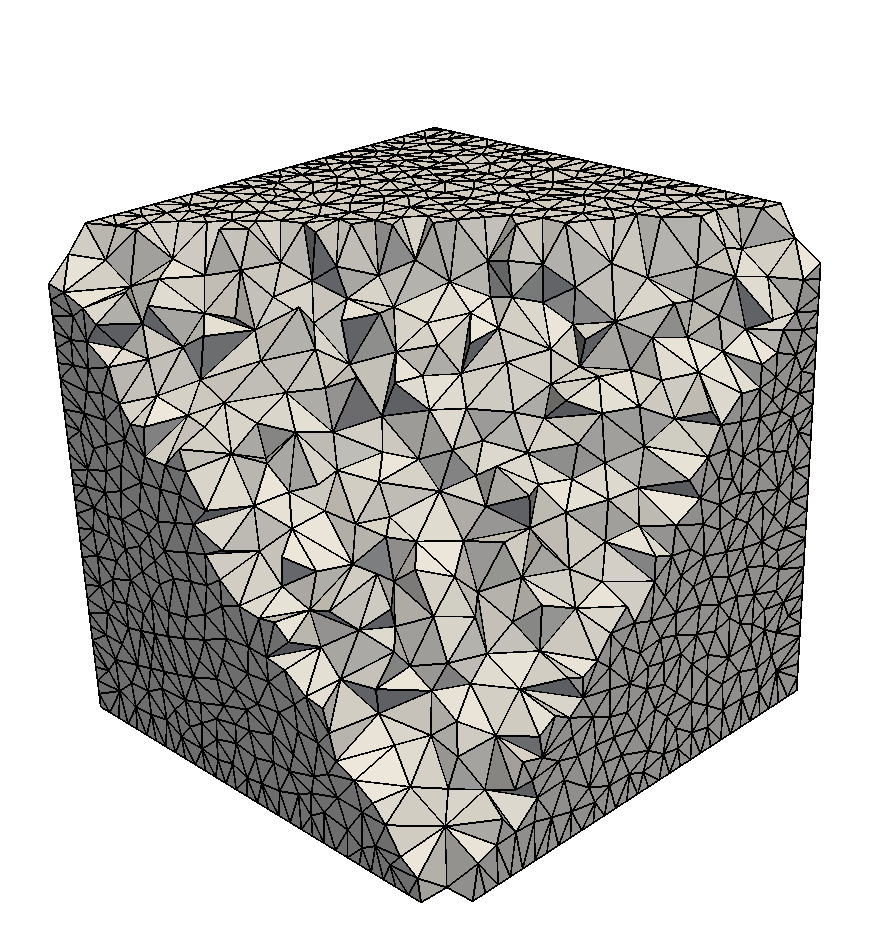}}
	\subfigure[c) \texttt{CVT}]{\includegraphics[width=\sizeMeshHybrid\textwidth,trim = 0mm 0mm 0mm 0mm, clip]{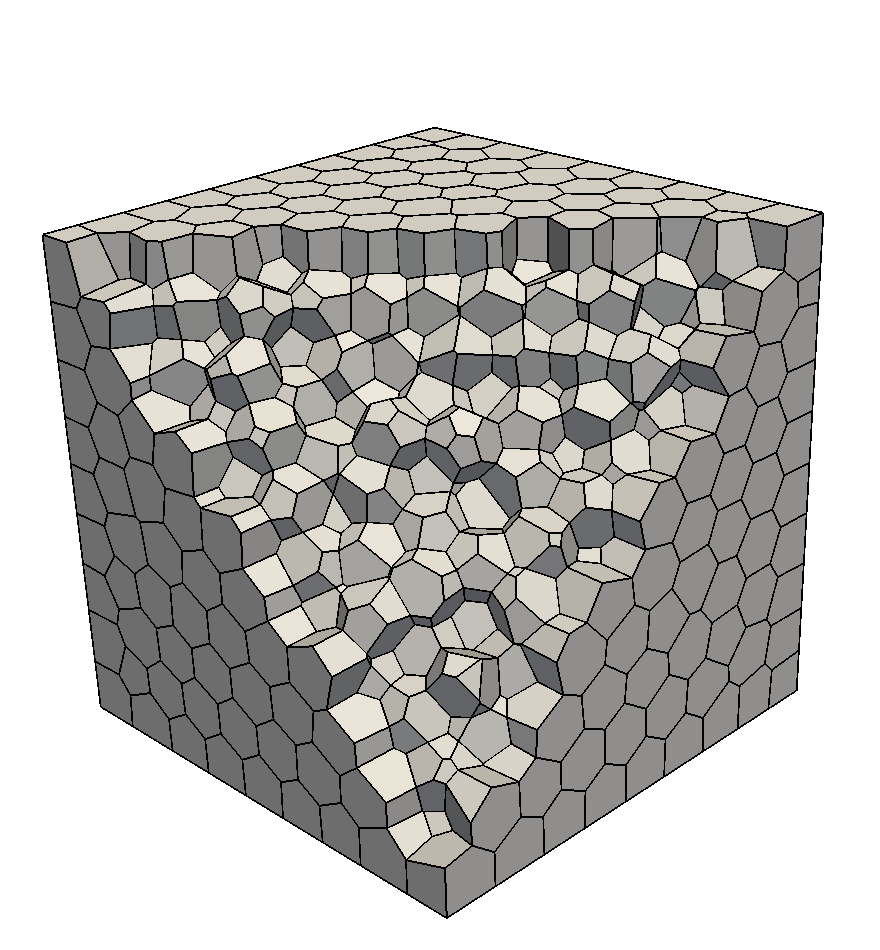}}
	\subfigure[d) \texttt{Rand}]{\includegraphics[width=\sizeMeshHybrid\textwidth,trim = 0mm 0mm 0mm 0mm, clip]{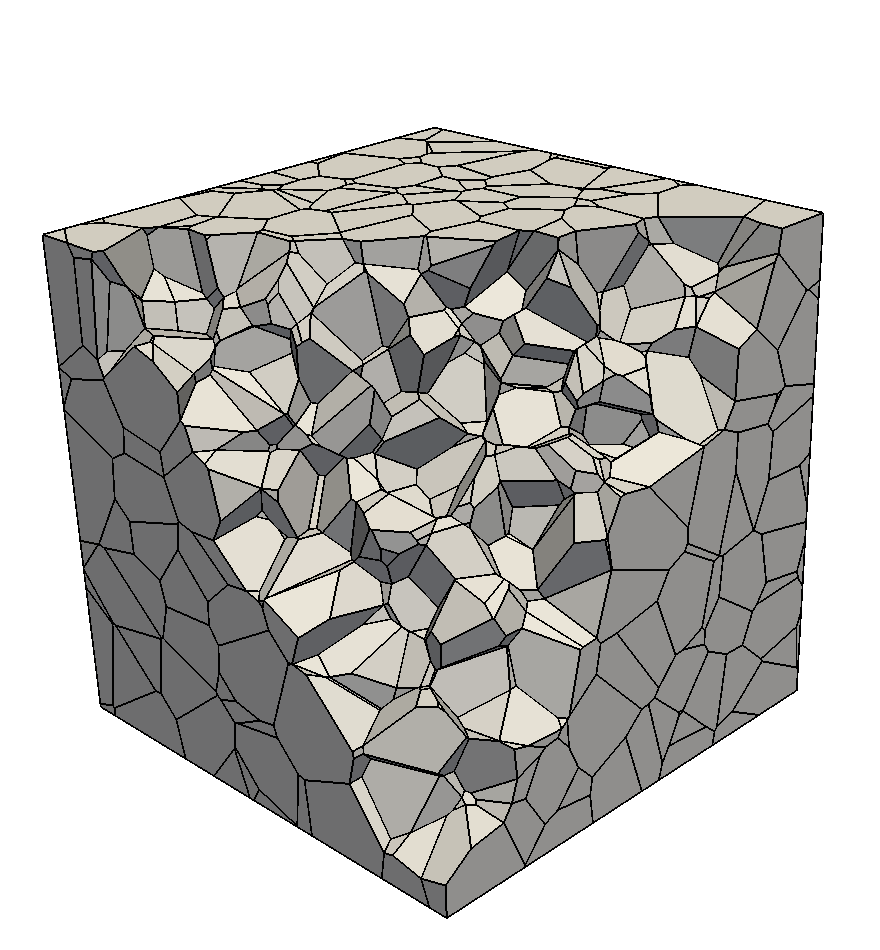}}					
	\caption{Overview of adopted meshes: in the first row the meshes for test case 2D, while in the second row the meshes for test case 3D.}
	\label{fig:h_meshes}
\end{figure}
%%
%% Convergence results
%%
\subsection{Convergence results}~\label{ss:convergenceResultsHybrid}
The first numerical results focus on the accuracy of the proposed VEM method using the hybridized procedure on the previous two test cases. To carry out this assessment, we use the following error norms:	
\begin{itemize}
	\item[$\bullet$] $L^2$ error norm for the displacement field:
	$ ||\bbu - \bbu_h ||_0$.
	
	\item[$\bullet$] $L^2$ error on the divergence:
	$ ||\bdiv \bfsigma - \bdiv \bfsigma_h ||_0$.
	
	\item[$\bullet$] $L^2$ error on the projection:
$ || \bfsigma - \Pi_h\bfsigma_h ||_0$.

	\item[$\bullet$] Discrete error norms for the stress field:
	\begin{equation*}
	E_{\bfsigma}   :=\left( \sum_{e \in \Eh} h_e\int_{e} \kappa\,| (\bfsigma - \bfsigma_h)\bbn  |^2~\ds\right)^{1/2} ,
	\end{equation*}
	where $\kappa=\frac{1}{2} {\rm tr}(\D)$ (the material is here homogeneous).
%	The quantities above are described in the 2D setting, but for the 3D case we simply need to replaces the mesh edges with faces. Furthermore, we remark that for the internal edges (faces for 3D) $\bfsigma_h\bbn$ is the mean value between the two contributes derived by the adjacent elements.
	%	Here, with abuse of notation, we indicate by $\Eh$ the set of the edges/faces for $\Th$.
\end{itemize}
We will give the numerical evidence that all the above quantities behave as $O(h)$.
%
%\begin{figure}[!ht]
%	\centering
%	\subfigure[]{\includegraphics[width=\sizeGraphHybrid\textwidth,trim = 0mm 0mm 0mm 0mm, clip]{FiguresHybrid/NearlyIncompressible_errorL2Displacement_2DAll}}
%	\subfigure[]{\includegraphics[width=\sizeGraphHybrid\textwidth,trim = 0mm 0mm 0mm 0mm, clip]{FiguresHybrid/NearlyIncompressible_errorL2Divergence_2DAll}}\\
%	\subfigure[]{\includegraphics[width=\sizeGraphHybrid\textwidth,trim = 0mm 0mm 0mm 0mm, clip]{FiguresHybrid/NearlyIncompressible_errorl2Discrete_2DAll}}
%	\subfigure[]{\includegraphics[width=\sizeGraphHybrid\textwidth,trim = 0mm 0mm 0mm 0mm, clip]{FiguresHybrid/NearlyIncompressible_errorL2Projection_2DAll}}
%	\caption{Convergence results. $h$-convergence results for the test case 2D and for all meshes.}\label{fig:convergence2D}
%\end{figure}
%\begin{figure}[!ht]
%	\centering
%	\subfigure[]{\includegraphics[width=\sizeGraphHybrid\textwidth,trim = 0mm 0mm 0mm 0mm, clip]{FiguresHybrid/Compressible_errorL2Displacement_3DAll}}
%	\subfigure[]{\includegraphics[width=\sizeGraphHybrid\textwidth,trim = 0mm 0mm 0mm 0mm, clip]{FiguresHybrid/Compressible_errorL2Divergence_3DAll}}\\
%	\subfigure[]{\includegraphics[width=\sizeGraphHybrid\textwidth,trim = 0mm 0mm 0mm 0mm, clip]{FiguresHybrid/Compressible_errorl2Discrete_3DAll}}
%	\subfigure[]{\includegraphics[width=\sizeGraphHybrid\textwidth,trim = 0mm 0mm 0mm 0mm, clip]{FiguresHybrid/Compressible_errorL2Projection_3DAll}}
%	\caption{Convergence results. $h$-convergence results for the test case 3D and for all meshes.}\label{fig:convergence3D}
%\end{figure}
\begin{figure}[!ht]
	\centering
	\includegraphics[width=\sizeGraphHybrid\textwidth,trim = 0mm 0mm 0mm 0mm, clip]{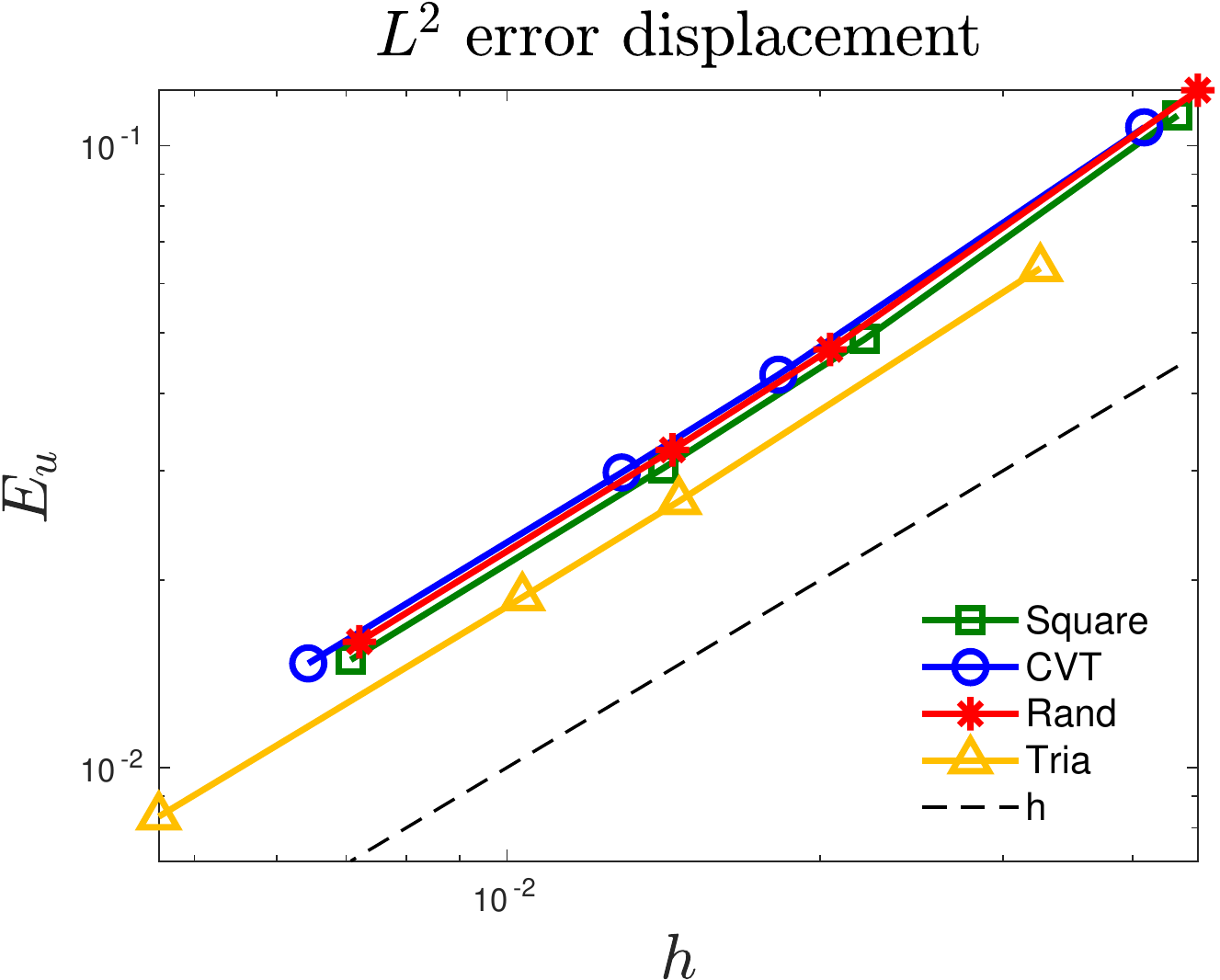}
	\hspace*{10pt}
	\includegraphics[width=\sizeGraphHybrid\textwidth,trim = 0mm 0mm 0mm 0mm, clip]{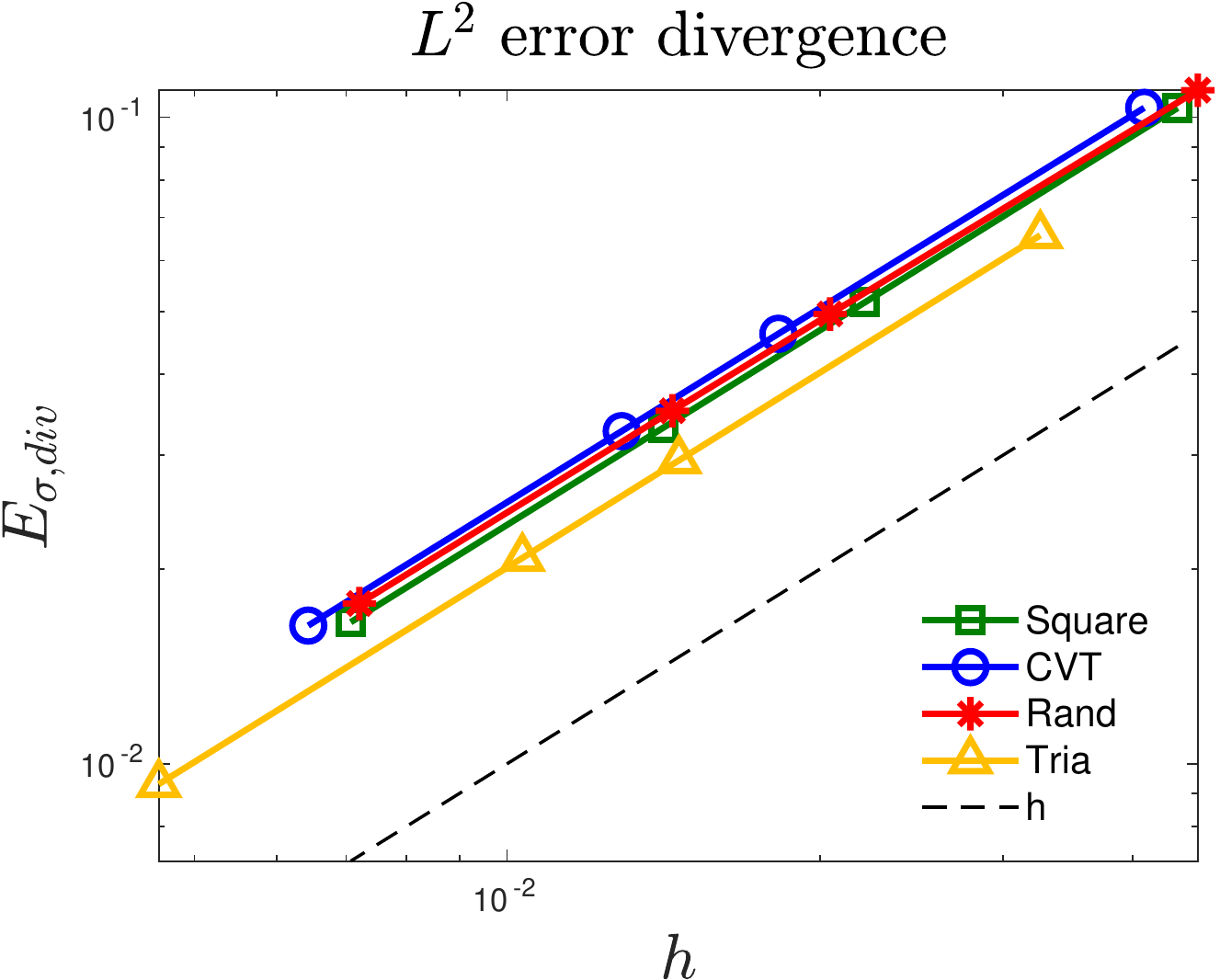}\\
	\vspace*{8pt}
	\includegraphics[width=\sizeGraphHybrid\textwidth,trim = 0mm 0mm 0mm 0mm, clip]{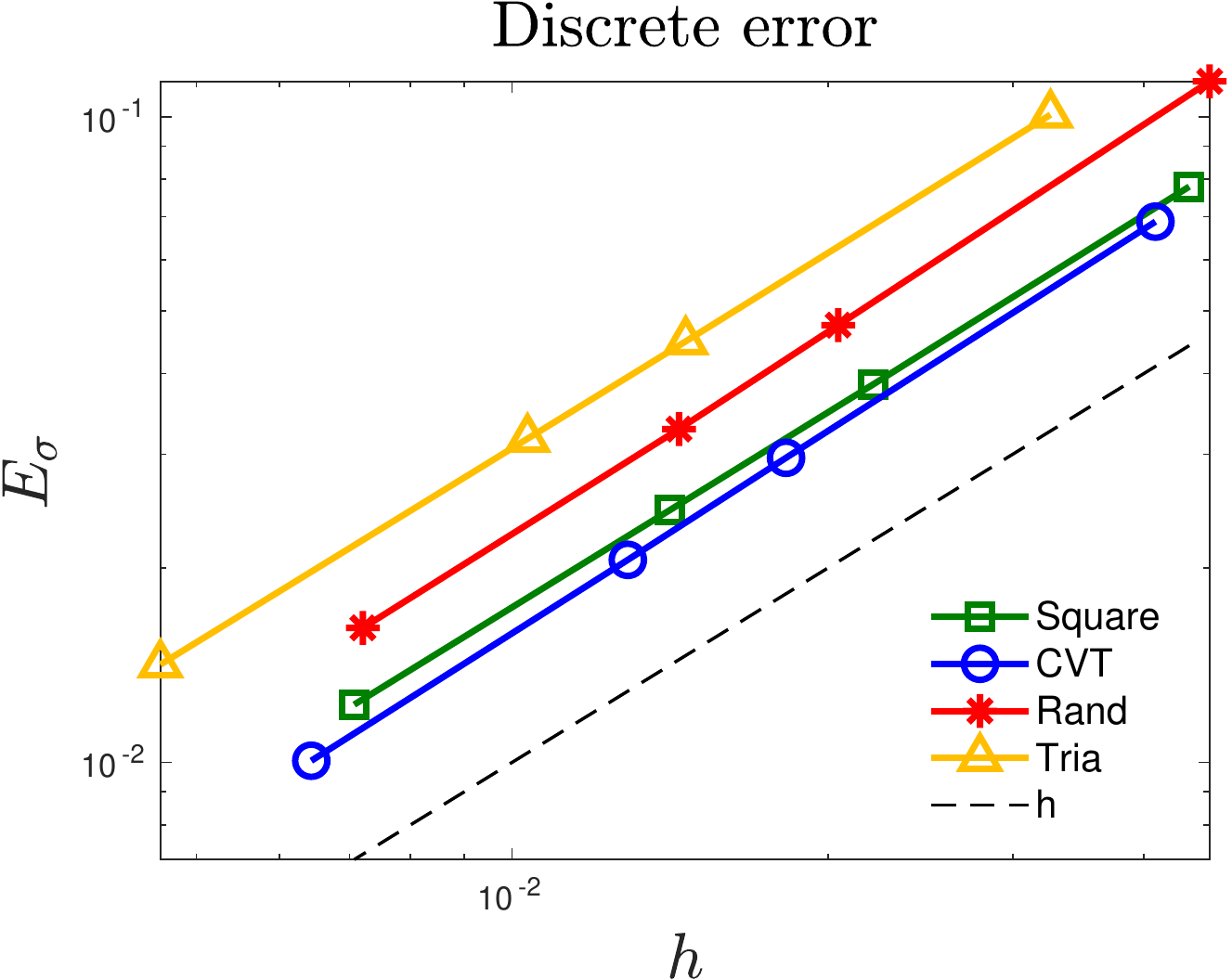}
	\hspace*{10pt}
	\includegraphics[width=\sizeGraphHybrid\textwidth,trim = 0mm 0mm 0mm 0mm, clip]{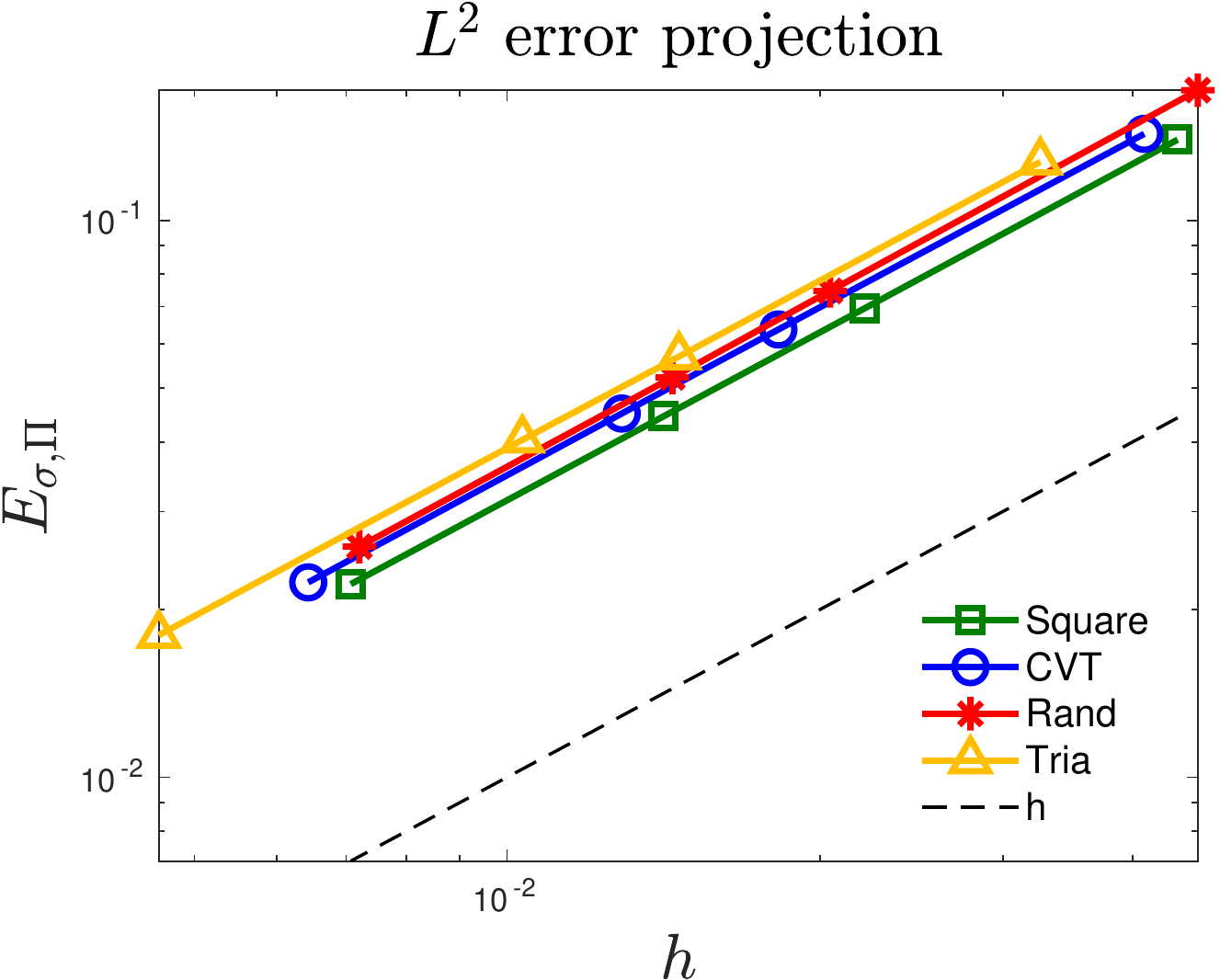}
	\caption{Convergence results. $h$-convergence results for the test case 2D and for all meshes.}\label{fig:convergence2D}
\end{figure}

\begin{figure}[!ht]
	\centering
	\includegraphics[width=\sizeGraphHybrid\textwidth,trim = 0mm 0mm 0mm 0mm, clip]{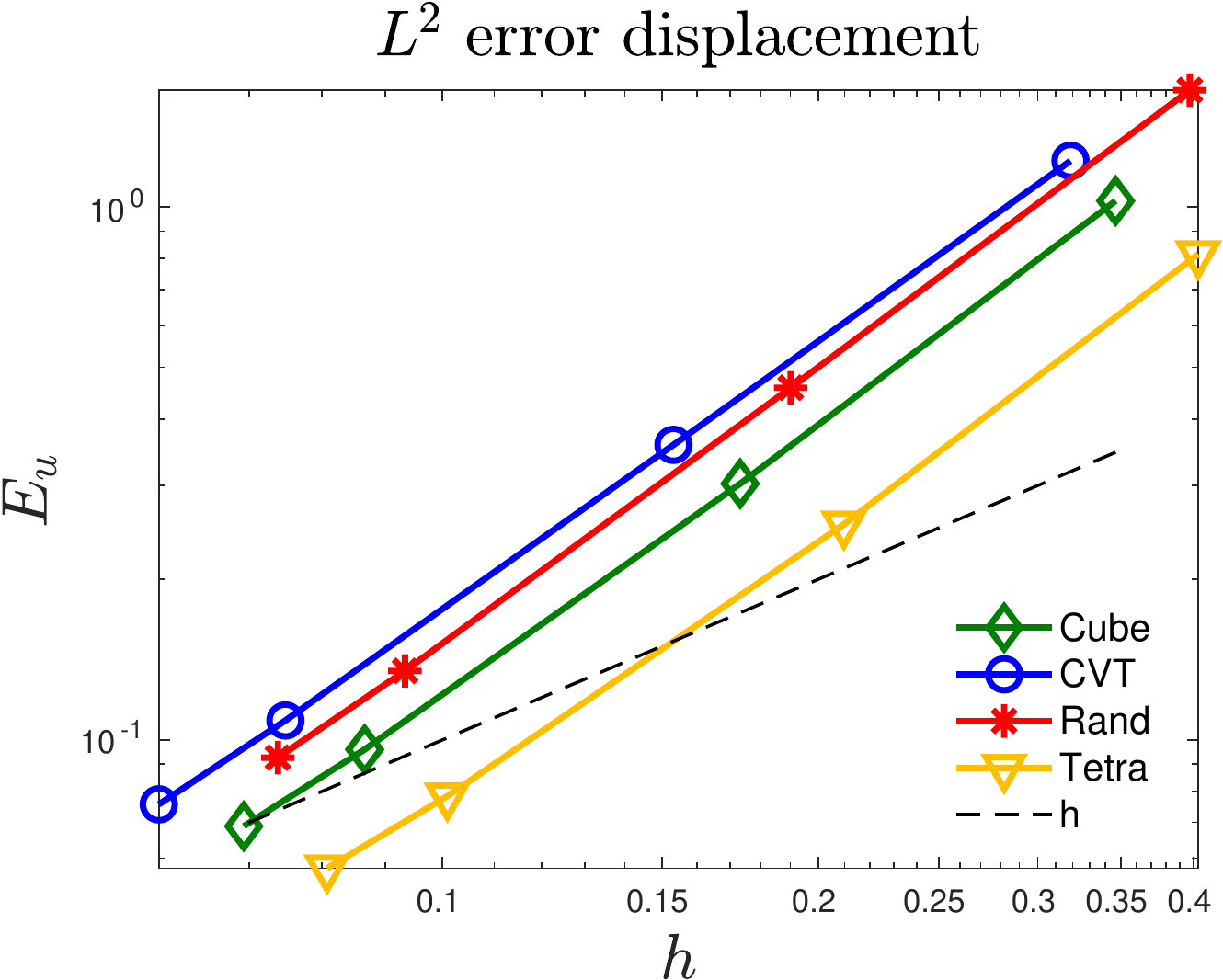}
	\hspace*{10pt}
	\includegraphics[width=\sizeGraphHybrid\textwidth,trim = 0mm 0mm 0mm 0mm, clip]{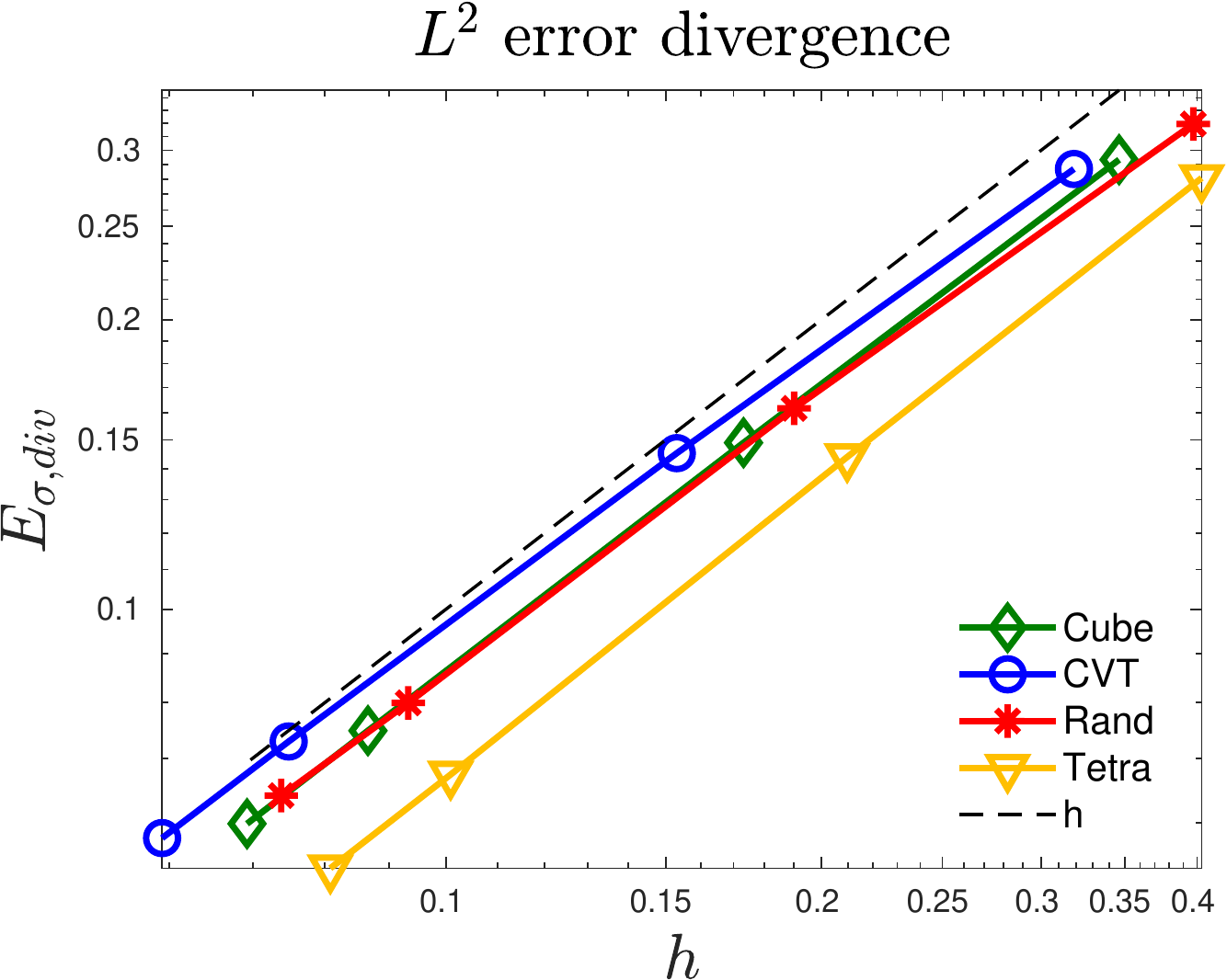}\\
	\vspace*{8pt}
	\includegraphics[width=\sizeGraphHybrid\textwidth,trim = 0mm 0mm 0mm 0mm, clip]{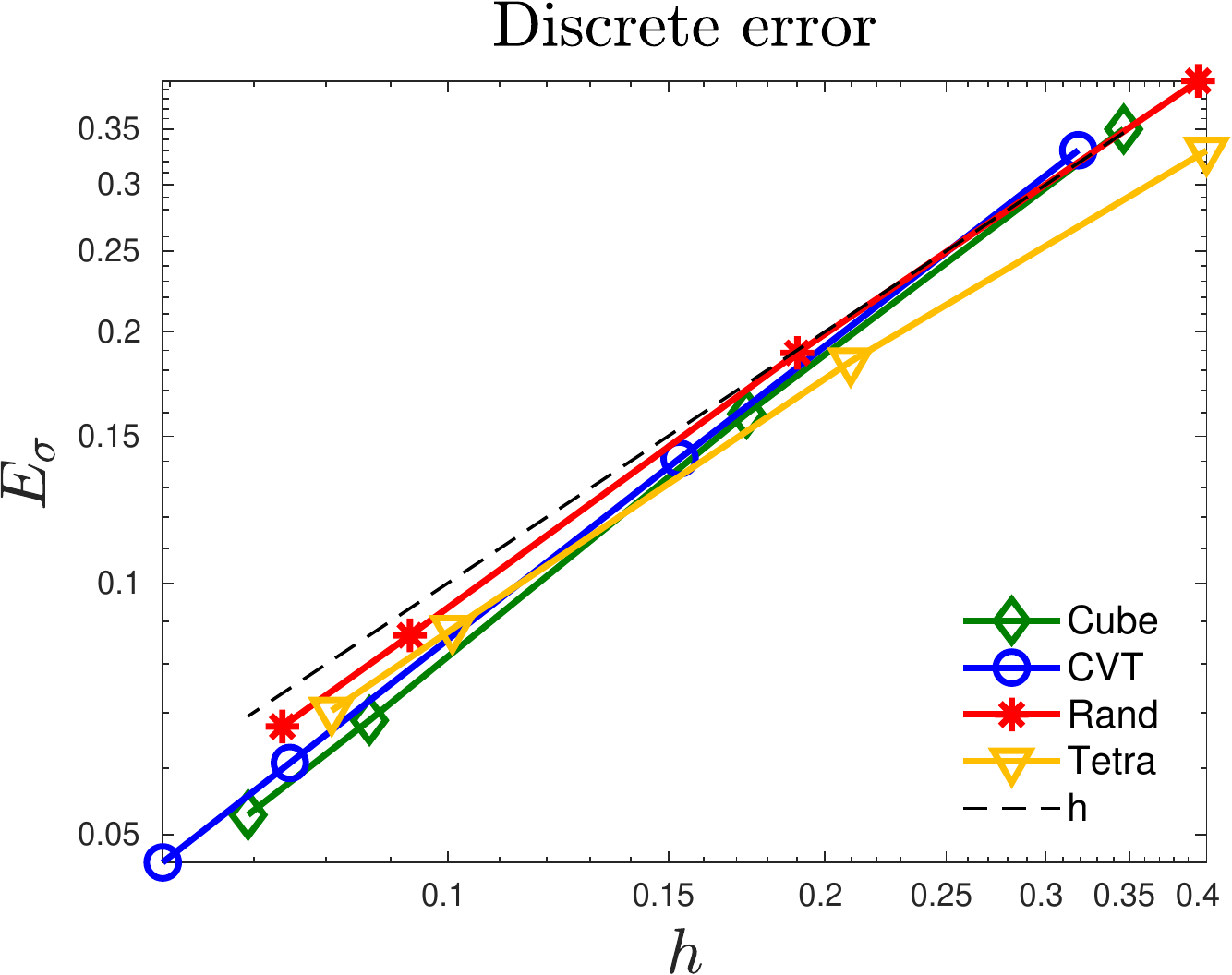}
	\hspace*{10pt}
	\includegraphics[width=\sizeGraphHybrid\textwidth,trim = 0mm 0mm 0mm 0mm, clip]{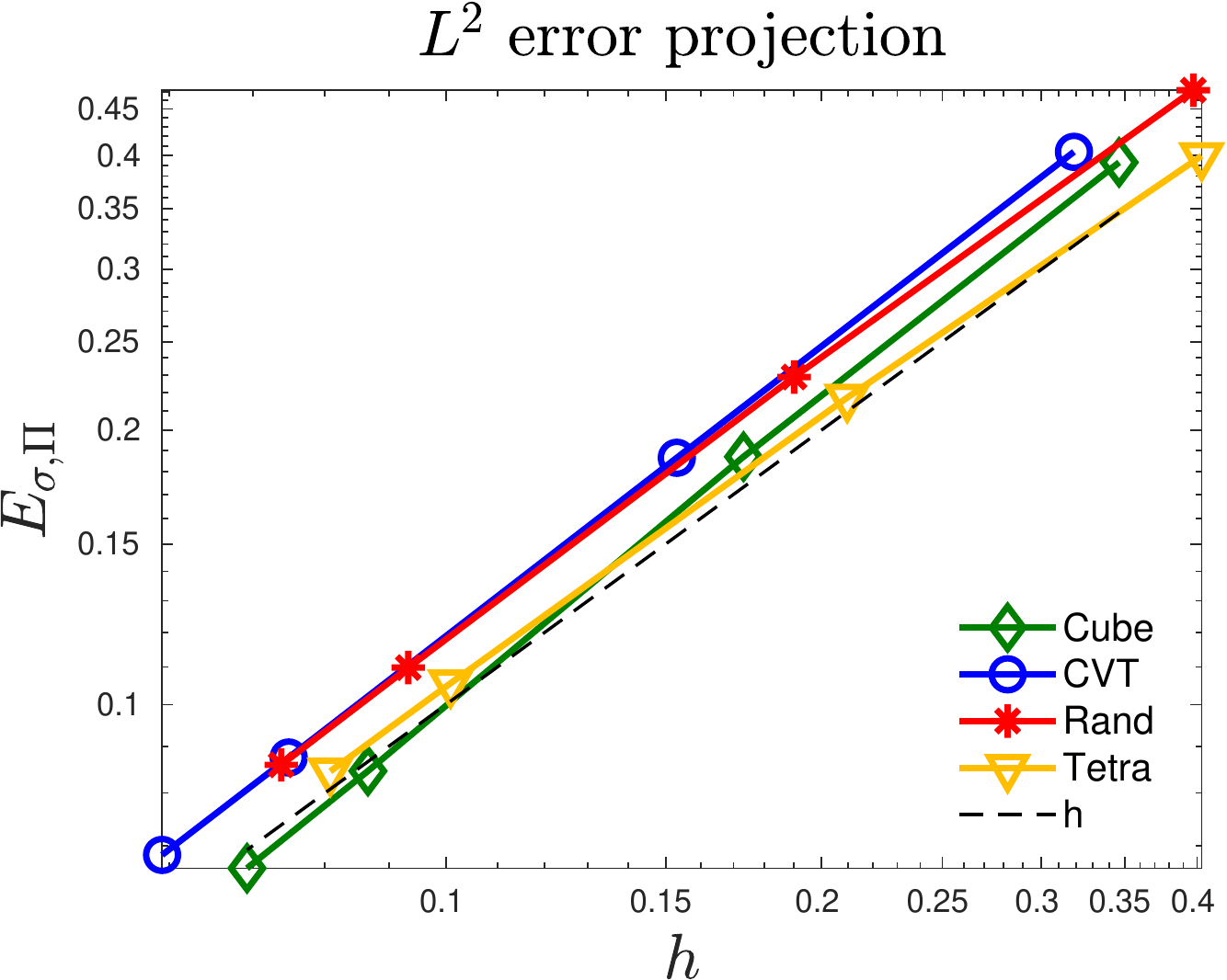}
	\caption{Convergence results. $h$-convergence results for the test case 3D and for all meshes.}\label{fig:convergence3D}
\end{figure}
Fig.~\ref{fig:convergence2D} and Fig.~\ref{fig:convergence3D} report the $h$-convergence of the proposed method for test case 2D and 3D, respectively. Relative errors are displayed. As expected, the hybridization leads to an asymptotic convergence rate equal to 1 for all the error norms and meshes (in fact, the hybridized schemes are equivalent to the original Hellinger-Reissner methods of Refs. \refcite{ARTIOLI2017155} and \refcite{DLV}). Moreover, the convergence graphs are very close to each others, which confirms the good robustness of the proposed VE method with respect to the mesh choice.
%%
%% Post-processing results
%%
\subsection{Post-processing results}
The present section has two goals. First of all we numerically confirm the superconvergence result predicted by Theorem~\ref{theorem:superconvergence}. Then, we exhibit the accuracy of our post-processed displacement field.

\paragraph{Superconvergence.}
We consider the following error quantities:
\begin{itemize}
	\item[$\bullet$] $L^2$ error norm of the $\Pi_{RM}$-projection of the displacement field (cf. \eqref{eq:proj_rm}):
$ ||\bar\bbu_h - \bbu_h ||_0$. 

	According to Theorem~\ref{theorem:superconvergence}, the expected behaviour of such an error is $O(h^2)$ for sufficiently regular problems.
	\item[$\bullet$] $L^2$ error norm of the projection onto piecewise constants of the displacement field:
$||\Pi_{0}\bbu - \bbu_h ||_0$.

	 By our convergence analysis, it is straightforward to see that 
	such a quantity is $O(h)$.
\end{itemize}
%\begin{figure}[!ht]
%	\centering
%	\renewcommand{\thesubfigure}{}
%	\subfigure[]{\includegraphics[width=\sizeGraphHybrid\textwidth,trim = 0mm 0mm 0mm 0mm, clip]{FiguresHybrid/NearlyIncompressible_errorL2DisplacementAndProjectionSolutionOnRigidBodyMotion_2DAll}}
%	\renewcommand{\thesubfigure}{}
%	\subfigure[]{\includegraphics[width=\sizeGraphHybrid\textwidth,trim = 0mm 0mm 0mm 0mm, clip]{FiguresHybrid/NearlyIncompressible_errorL2DisplacementAndProjectionSolutionOnConstant_2DAll}}
%	\caption{Superconvergence results. $h$-convergence results for test case 2D for all meshes.}~\label{fig:superconvergence2D}
%\end{figure}
%\begin{figure}[!ht]
%	\centering
%	\renewcommand{\thesubfigure}{}
%	\subfigure[]{\includegraphics[width=\sizeGraphHybrid\textwidth,trim = 0mm 0mm 0mm 0mm, clip]{FiguresHybrid/Compressible_errorL2DisplacementAndProjectionSolutionOnRigidBodyMotion_3DAll}}
%	\renewcommand{\thesubfigure}{}
%	\subfigure[]{\includegraphics[width=\sizeGraphHybrid\textwidth,trim = 0mm 0mm 0mm 0mm, clip]{FiguresHybrid/Compressible_errorL2DisplacementAndProjectionSolutionOnConstant_3DAll}}
%	\caption{Superconvergence results. $h$-convergence results for test case 3D for all meshes.}~\label{fig:superconvergence3D} 	
%\end{figure}
\begin{figure}[!ht]
\centering
\includegraphics[width=\sizeGraphHybrid\textwidth,trim = 0mm 0mm 0mm 0mm, clip]{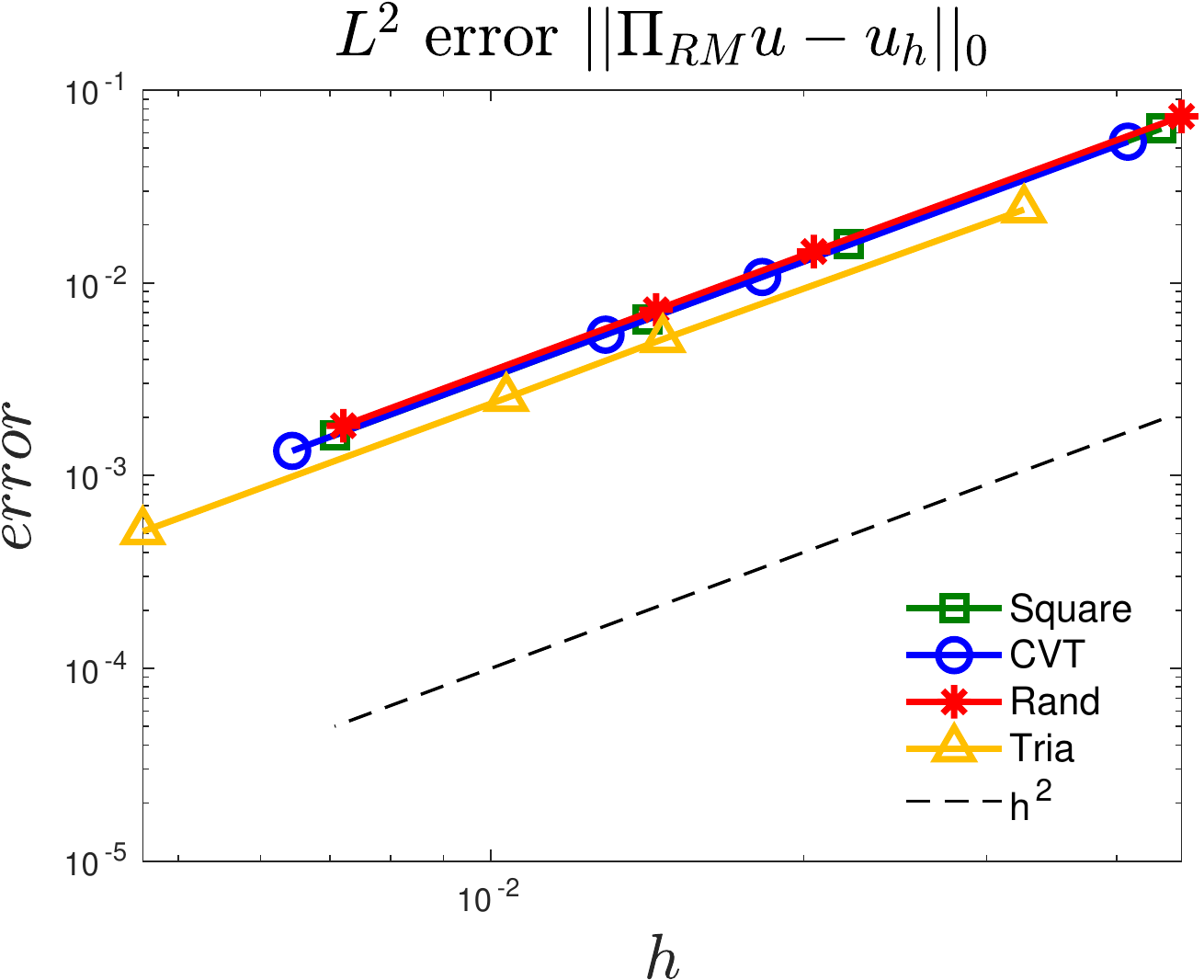}
\hspace*{10pt}
\includegraphics[width=\sizeGraphHybrid\textwidth,trim = 0mm 0mm 0mm 0mm, clip]{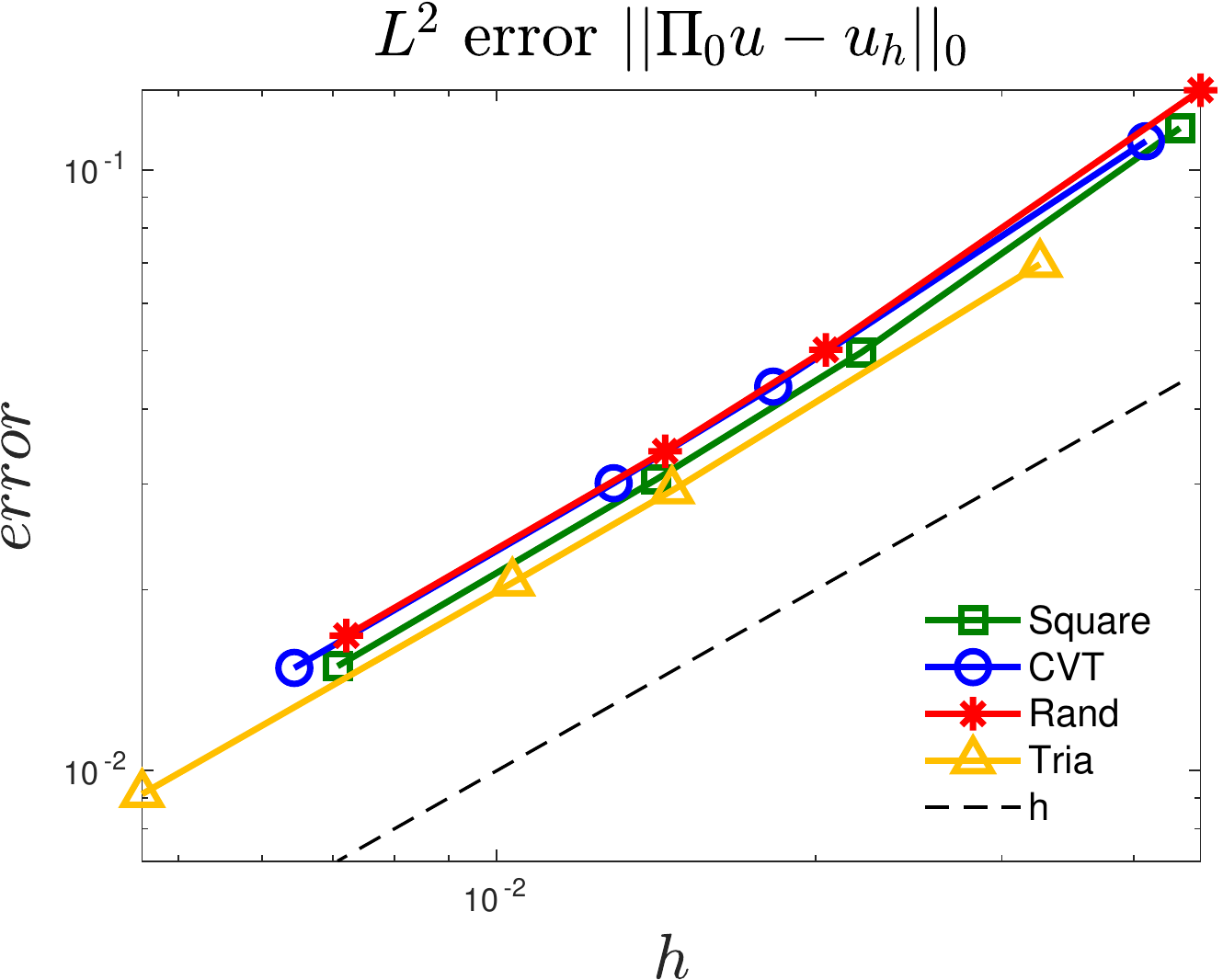}
\caption{Superconvergence results. $h$-convergence results for test case 2D for all meshes.}~\label{fig:superconvergence2D}
\end{figure}
\begin{figure}[!ht]
\centering
\renewcommand{\thesubfigure}{}
\includegraphics[width=\sizeGraphHybrid\textwidth,trim = 0mm 0mm 0mm 0mm, clip]{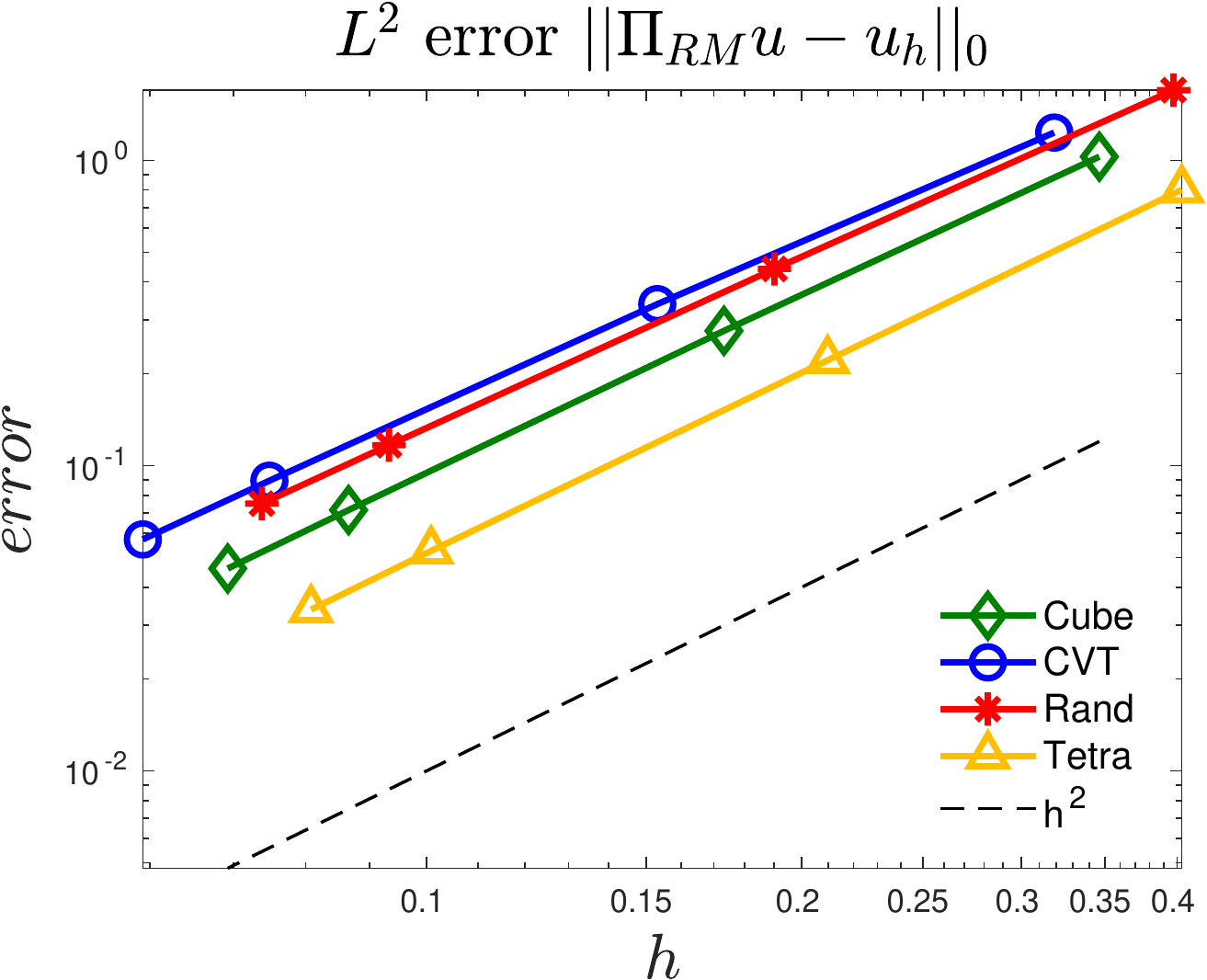}
\hspace*{10pt}
\includegraphics[width=\sizeGraphHybrid\textwidth,trim = 0mm 0mm 0mm 0mm, clip]{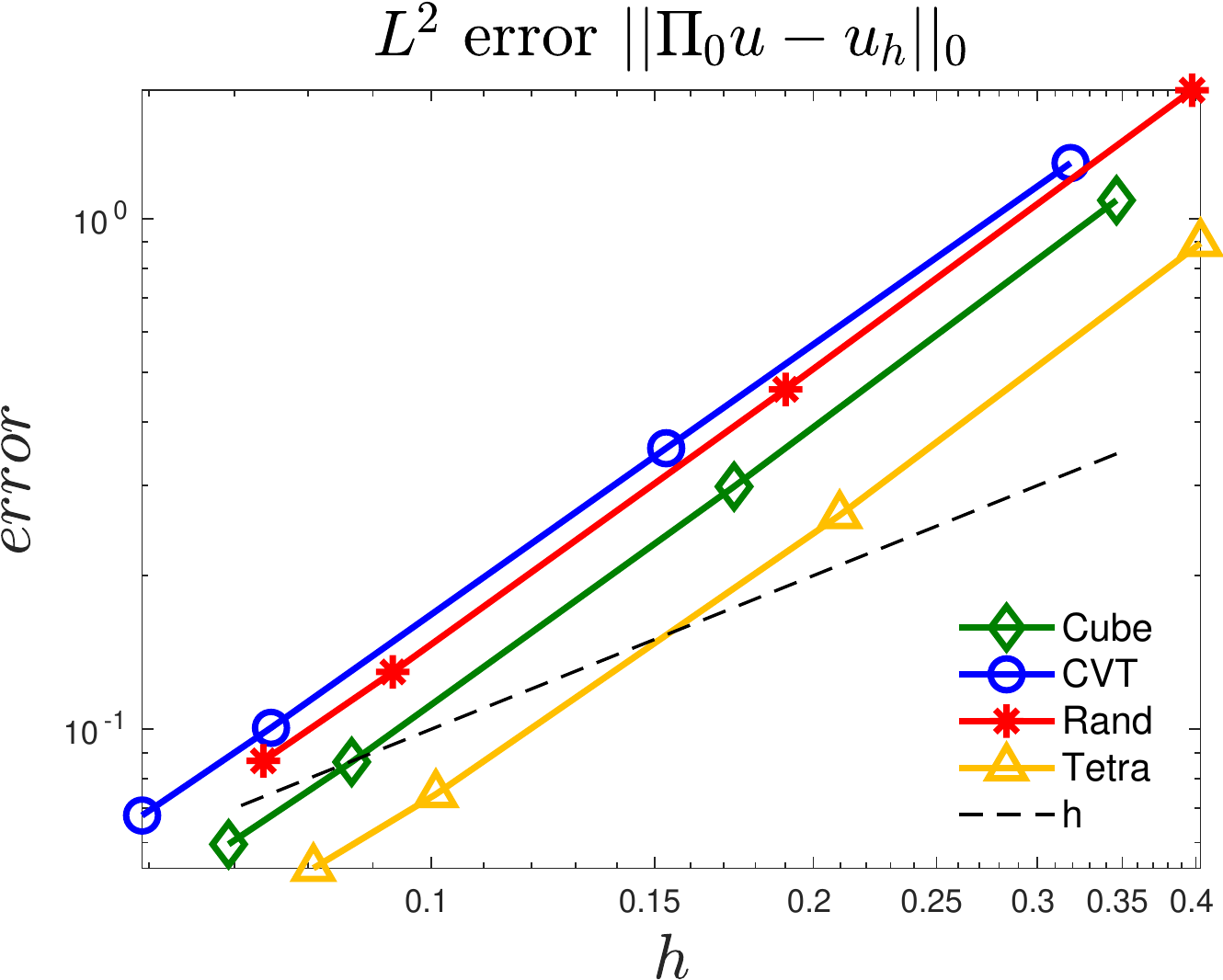}
\caption{Superconvergence results. $h$-convergence results for test case 3D for all meshes.}~\label{fig:superconvergence3D} 	
\end{figure}
In Fig.~\ref{fig:superconvergence2D} and Fig.~\ref{fig:superconvergence3D} we show the convergence graphs for the errors above. Again, relative errors are displayed. The convergence rate for the error norm $E_{\bbu_{RM}}$ is 2, in accordance with the theoretical results, see~\eqref{eq:superconvergence}. Instead, the error norm $E_{\bbu_0}$ does not exhibit the same behaviour: it is only $O(h)$. However, from these graphs we can also appreciate the robustness of the VEM with respect to element distortions. Indeed, the convergence lines for the four meshes (2D and 3D) are very close to each others.
\paragraph{Post-processing.} 
%	The second numerical results of this section are about our postprocessing.
Since the VE post-processed displacement is not explicitly known inside the element, we introduce the following error measures
\begin{equation*}
E^0_{\bbu_h^*}:=||\bbu - \Pi^{\nabla}\bbu_h^*||_0\quad \mbox{and} \qquad
E^1_{\bbu_h^*}:=|\bbu - \Pi^{\nabla}\bbu_h^*|_{1,\Th},
\end{equation*}
where the projection operator $\Pi^{\nabla}$ is defined by~\eqref{eq:h_projection_post-processing}.  
However, we remark that on simplices the function $\bbu^*_h$ is indeed computable: it corresponds to the vectorial version of the  
non-conforming Finite Element post-processed solution detailed in Ref.~\refcite{CrouzeixRaviart}. In such a case, the operator $\Pi^{\nabla}$ is simply the identity. 
%\begin{figure}[!ht]
%	\centering
%	\renewcommand{\thesubfigure}{}
%	\subfigure[]{\includegraphics[width=\sizeGraphHybrid\textwidth,trim = 0mm 0mm 0mm 0mm, clip]{FiguresHybrid/NearlyIncompressible_errorL2DisplacementPostProcessing_2DAll}}
%	\renewcommand{\thesubfigure}{}
%	\subfigure[]{\includegraphics[width=\sizeGraphHybrid\textwidth,trim = 0mm 0mm 0mm 0mm, clip]{FiguresHybrid/NearlyIncompressible_errorBrokenH1DisplacementPostProcessing_2DAll}}
%	\caption{Post-processing. Convergence plots for the error $E^0_{\bbu_h^*}$ and $E^1_{\bbu_h^*}$ for test case 2D.}~\label{fig:post-processing2D}
%\end{figure}
%\begin{figure}[!ht]
%	\centering
%	\renewcommand{\thesubfigure}{}
%	\subfigure[]{\includegraphics[width=\sizeGraphHybrid\textwidth,trim = 0mm 0mm 0mm 0mm, clip]{FiguresHybrid/Compressible_errorL2DisplacementPostProcessing_3DAll}}
%	\renewcommand{\thesubfigure}{}
%	\subfigure[]{\includegraphics[width=\sizeGraphHybrid\textwidth,trim = 0mm 0mm 0mm 0mm, clip]{FiguresHybrid/Compressible_errorBrokenH1DisplacementPostProcessing_3DAll}}
%		\caption{Post-processing. Convergence plots for the error $E^0_{\bbu_h^*}$ and $E^1_{\bbu_h^*}$ for test case 3D.}~\label{fig:post-processing3D}
%\end{figure}
\begin{figure}[!ht]
	\centering
	\includegraphics[width=\sizeGraphHybrid\textwidth,trim = 0mm 0mm 0mm 0mm, clip]{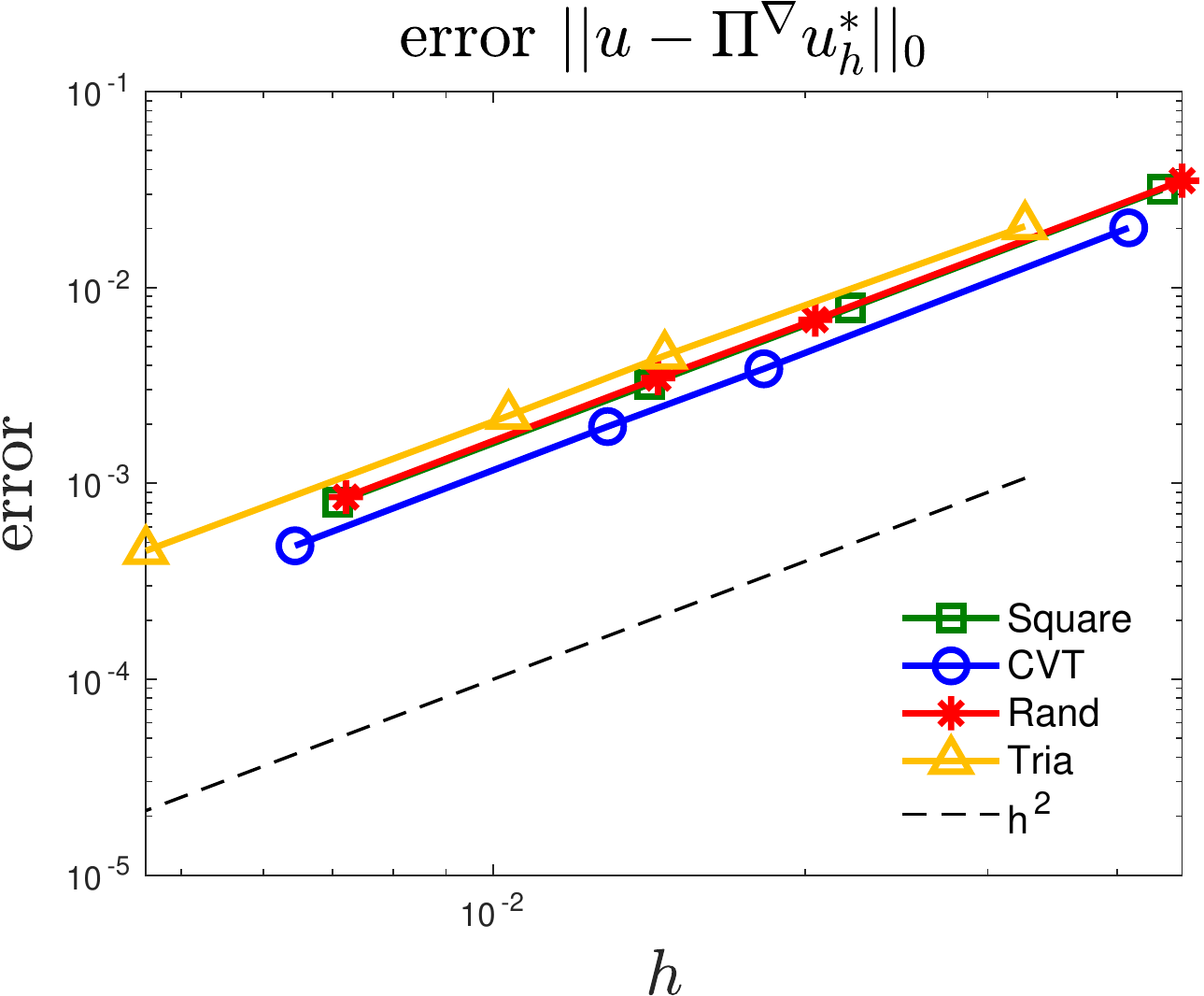}
	\hspace*{10pt}
	\includegraphics[width=\sizeGraphHybrid\textwidth,trim = 0mm 0mm 0mm 0mm, clip]{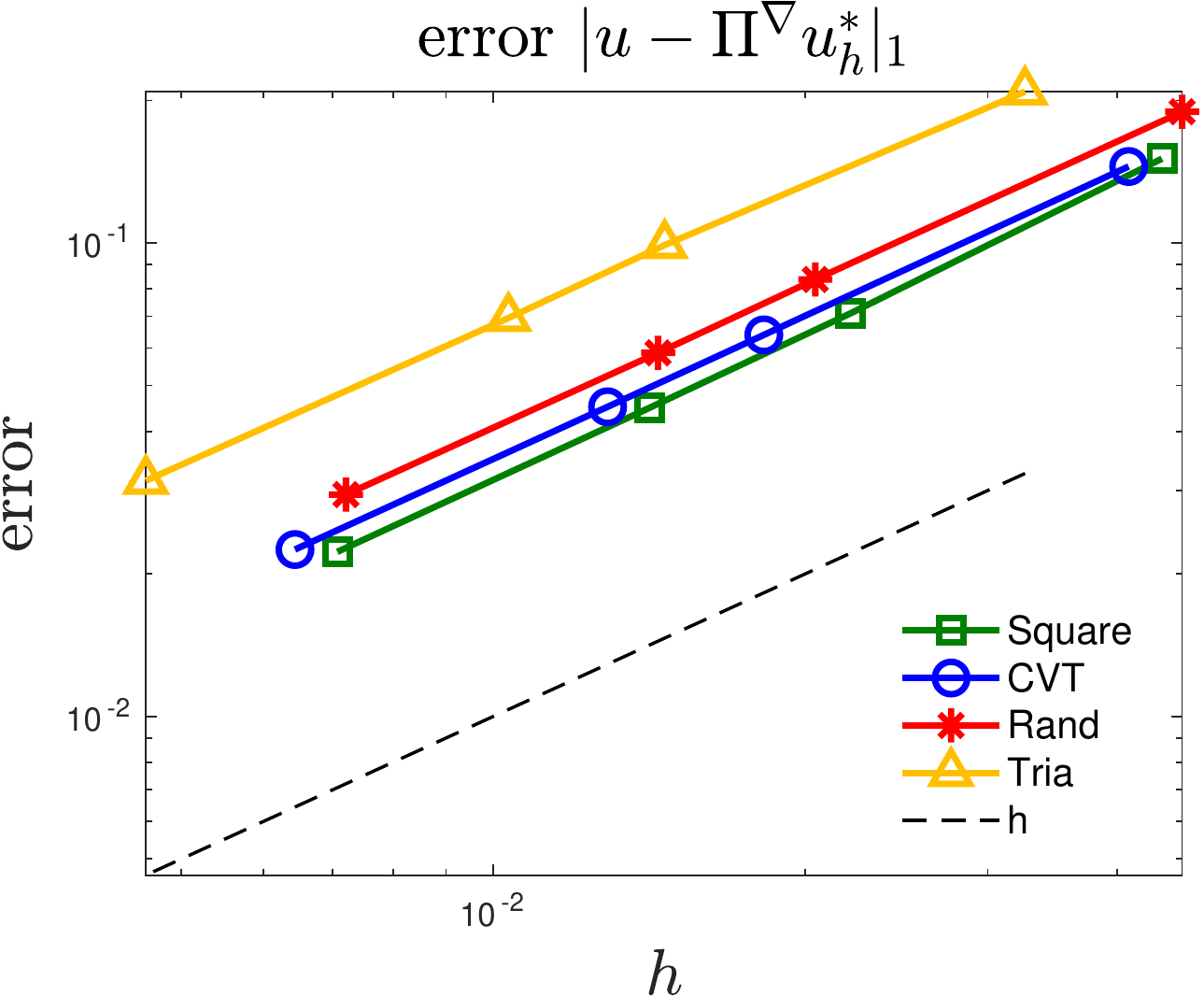}
	\caption{Post-processing. Convergence plots for the error $E^0_{\bbu_h^*}$ and $E^1_{\bbu_h^*}$ for test case 2D.}~\label{fig:post-processing2D}
\end{figure}
\begin{figure}[!ht]
	\centering
	\includegraphics[width=\sizeGraphHybrid\textwidth,trim = 0mm 0mm 0mm 0mm, clip]{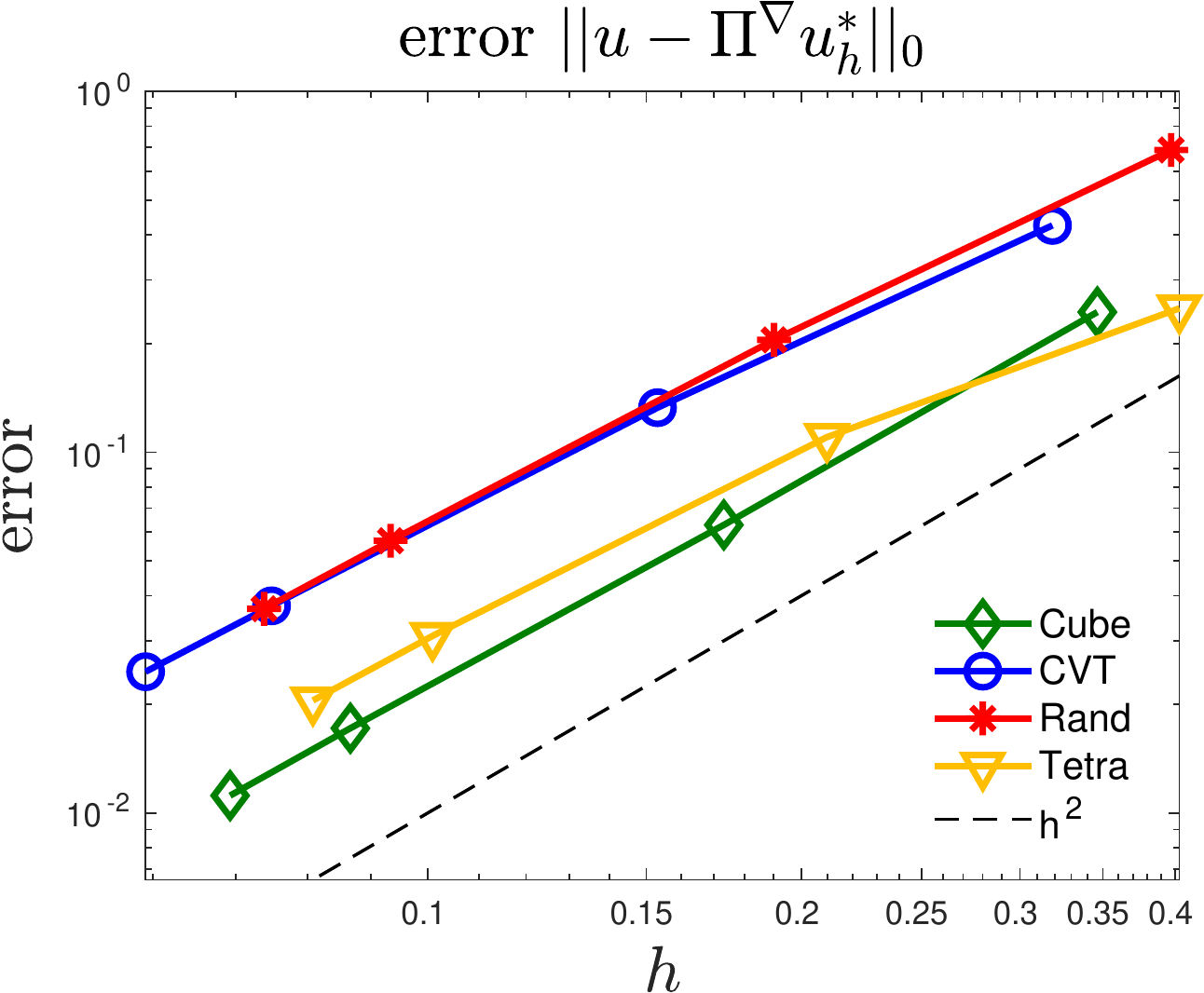}
	\hspace*{10pt}
	\includegraphics[width=\sizeGraphHybrid\textwidth,trim = 0mm 0mm 0mm 0mm, clip]{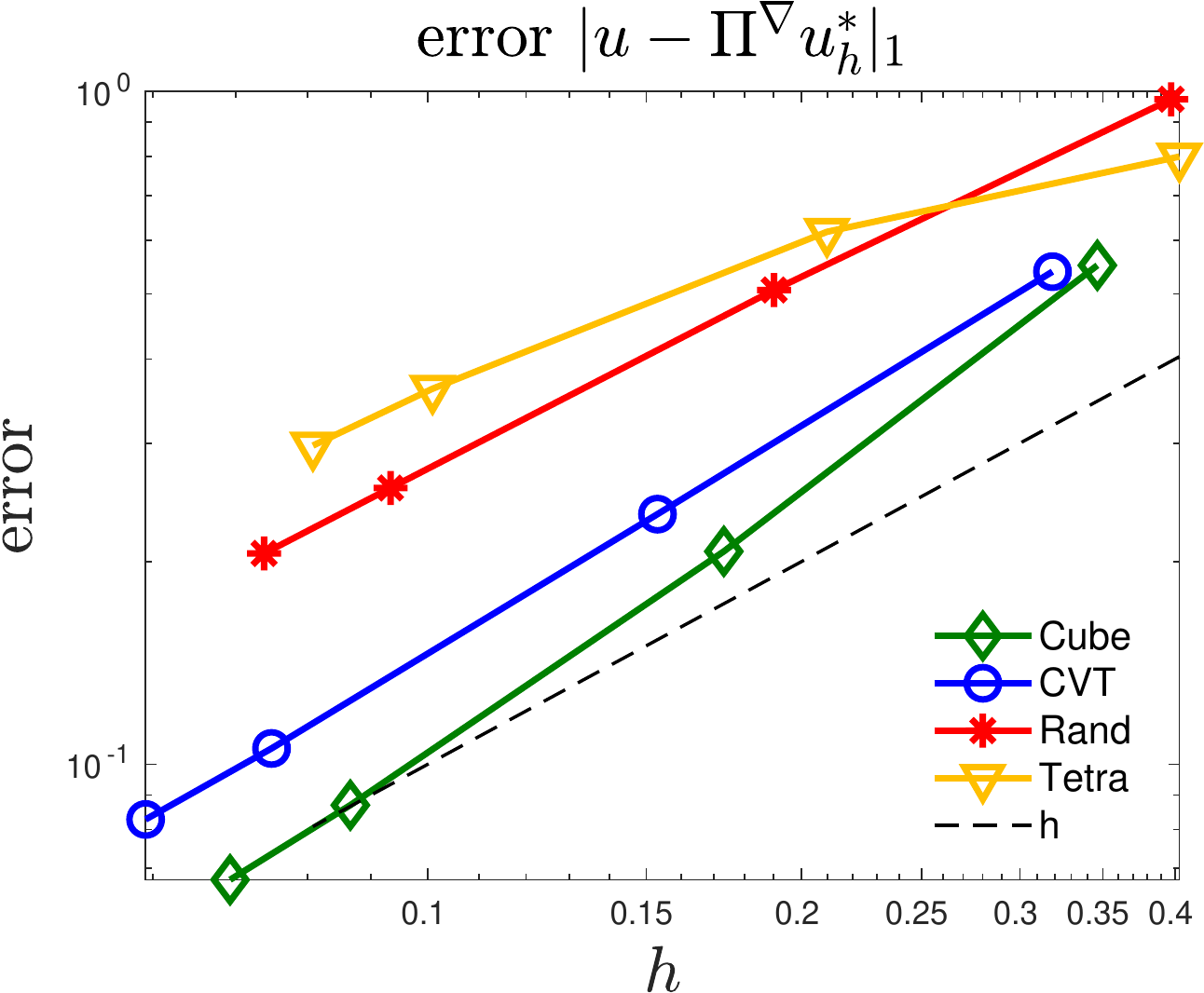}
	\caption{Post-processing. Convergence plots for the error $E^0_{\bbu_h^*}$ and $E^1_{\bbu_h^*}$ for test case 3D.}~\label{fig:post-processing3D}
\end{figure}

In Fig.~\ref{fig:post-processing2D} and Fig.~\ref{fig:post-processing3D} we report the convergence lines for the errors $E^0_{\bbu_h^*}$ and $E^1_{\bbu_h^*}$ for both test cases. Relative errors are displaced. The convergence rate for the error $E^0_{\bbu_h^*}$ is approximately 2, while for the error $E^1_{\bbu_h^*}$ is 1, as expected by Theorem~\ref{theorem:post-processing_result}. Although estimate \eqref{eq:H1inequalityMainTheorem} has been proved only for quasi-uniform meshes, our numerical tests suggests that the same convergence behaviour occurs for more general situations (e.g., \texttt{Rand} meshes are not quasi-uniform but a first order convergence rate takes place). 
 Moreover, the convergence lines of the each mesh are close to each others, showing, one more time, the VEM robustness with respect to the deformation of the mesh. 
%%
%% Comparison of solving time
%%
\subsection{Comparison of solving time}
Our last numerical test concerns the effect of the hybridization on the solution time of the resulting linear system. Accordingly, we qualitatively  compare the solving times between the standard low-order VEM approach, cf. Refs.~\refcite{ARTIOLI2017155,DLV} and the hybridizated scheme procedure (see Sec.~\ref{section:hybrid}, and in particular subsection \ref{ss:static_cond}). 
% petsc-user-ref l'ho tolto dalle referenze ma si può sempre aggiungere
We use the open-source library PETSc, see Ref.~\refcite{petsc-web-page}. In particular, we use the direct solver MUMPS: LU factorization for the standard method; Cholesky for the hybridized one. Moreover, we run our test only on one processor in order to have the same setting for both the cases. 

\begin{table}[!ht]
	\tbl{\label{t:Time1_2D}Comparison of solving time between standard approach and hybridization technique for test case 2D.}% We use \texttt{Square} and \texttt{Tria} meshes.}
	{\begin{tabular}{@{}cccccccc@{}} 
			\toprule
			&  & \texttt{Square}\hphantom{0000000000000} &   & \texttt{Tria}  \hphantom{0000000000000}&\\ \colrule
			{Step} & {Standard} & {Hybrid} & {Standard}  & {Hybrid}  \\ 
			{1} & {0.07}  & {0.07}\hphantom{0} {(33.87\%)} & {0.15} & \hphantom{0}{0.17}\hphantom{0} {(39.48\%)} \\
			{2} & {0.37}  & {0.33}\hphantom{0} {(41.05\%)} & {0.89} & \hphantom{0}{0.99}\hphantom{0} {(41.98\%)} \\
			{3} & {0.81}  & {0.93}\hphantom{0} {(47.70\%)} & {2.05} & \hphantom{0}{2.15}\hphantom{0} {(45.65\%)} \\
			{4} & {4.39}  & {5.11}\hphantom{0} {(55.59\%)} & {17.34}\hphantom{0} & {14.26}\hphantom{0} {(54.96\%)} \\
%			\botrule
%	\end{tabular}}
%	%\end{table}
%	%\begin{table}
%	\tbl{\label{t:Time2_2D}Comparison of solving time between standard approach and hybridization technique for test case 2D. We use \texttt{CVT} and \texttt{Rand} meshes.}
%	{\begin{tabular}{@{}cccccccc@{}} 
			\toprule
			&  & \hphantom{000}\texttt{CVT}\hphantom{0000000000000} &   & \texttt{Rand}  \hphantom{0000000000000}&\\ \colrule
			{Step} & {Standard} & {Hybrid} & {Standard}  & {Hybrid}  \\ 
			{1} & \hphantom{0}{0.16} & \hphantom{0}{0.21}\hphantom{0} {(37.88\%)} & \hphantom{0}{0.15} & \hphantom{0}{0.37}\hphantom{0} {(30.17\%)} \\
			{2} & \hphantom{0}{1.30} & \hphantom{0}{1.22}\hphantom{0} {(53.51\%)} & \hphantom{0}{1.10} & \hphantom{0}{1.22}\hphantom{0} {(51.44\%)} \\
			{3} & \hphantom{0}{3.63} & \hphantom{0}{2.96}\hphantom{0} {(58.00\%)} & \hphantom{0}{2.86} & \hphantom{0}{2.92}\hphantom{0} {(53.47\%)} \\
			{4} & {30.55} & {18.50}\hphantom{0} {(71.00\%)} &{23.13} &{16.78}\hphantom{0} {(65.81\%)} \\
			\botrule
	\end{tabular}}
\end{table}
\begin{table}[!ht]
	\tbl{\label{t:Time1}Comparison of solving time between standard approach and hybridization technique for test case 3D.}% We use \texttt{Cube} and \texttt{Tetra} meshes.}
	{\begin{tabular}{@{}cccccccc@{}} 
			\toprule
			&  & \texttt{Cube}\hphantom{0000000000000} &   & \texttt{Tetra}  \hphantom{0000000000000}&\\ \colrule
			{Step} & {Standard} & {Hybrid} & {Standard}  & {Hybrid}  \\
			{1} & \hphantom{000}{0.11} & \hphantom{00}{0.11}\hphantom{0} {(38.02\%)} & \hphantom{000}{0.12} & \hphantom{000}{0.11}\hphantom{0} {(32.14\%)} \\
			{2} & \hphantom{000}{5.74} & \hphantom{00}{3.09}\hphantom{0} {(82.06\%)} & \hphantom{000}{3.80} & \hphantom{000}{2.28}\hphantom{0} {(70.37\%)} \\
			{3} & \hphantom{0}{971.33} & {209.53}\hphantom{0} {(97.15\%)} & \hphantom{0}{568.12} & \hphantom{0}{284.78}\hphantom{0} {(97.43\%)} \\
			{4} & {4178.47} & {903.67}\hphantom{0} {(98.78\%)} & {3393.64} & {1409.92}\hphantom{0} {(98.33\%)} \\
%			\botrule
%	\end{tabular}}
%	%\end{table}
%	%\begin{table}[ht]
%	\tbl{\label{t:Time2}Comparison of solving time between standard approach and hybridization technique for test case 3D. We use \texttt{CVT} and \texttt{Rand} meshes.}
%	{\begin{tabular}{@{}cccccccc@{}} 
			\toprule
			&  & \hphantom{0}\texttt{CVT}\hphantom{0000000000000} &   & \hphantom{0}\texttt{Rand}  \hphantom{0000000000000}&\\ \colrule
			{Step} & {Standard} & {Hybrid} & {Standard}  & {Hybrid}  \\
			{1} & \hphantom{00000}{0.86} & \hphantom{0000}{0.68}\hphantom{0} {(63.22\%)} & \hphantom{00000}{1.22} & \hphantom{0000}{0.91}\hphantom{0} {(67.32\%)} \\
			{2} & \hphantom{0000}{97.88} & \hphantom{000}{53.43}\hphantom{0} {(94.88\%)} & \hphantom{000}{161.21} & \hphantom{000}{72.13}\hphantom{0} {(95.04\%)} \\
			{3} & \hphantom{0}{29062.80} & \hphantom{0}{6877.68}\hphantom{0} {(99.56\%)} & \hphantom{0}{32015.50} & {14565.00}\hphantom{0} {(99.70\%)} \\
			{4} & {128626.00} & {41000.70}\hphantom{0} {(99.86\%)} & {172781.00} & {81037.80}\hphantom{0} {(99.91\%)} \\
			\botrule
	\end{tabular}}
\end{table}
%In Tables~\ref{t:Time1_2D} and~\ref{t:Time2_2D} for the 2D case (resp., in Tables~\ref{t:Time1} and~\ref{t:Time2} for the 3D case), we show a comparison between the solving time for the standard VE method and the time of the hybridization procedure (static condensation and solving time) for each mesh refinement step.
In Table~\ref{t:Time1_2D} for the 2D case (resp., in Table~\ref{t:Time1} for the 3D case), we show a comparison between the solving time for the standard VE method and the time of the hybridization procedure (static condensation and solving time) for each mesh refinement step. Moreover, in the column ``Hybrid'', we also show the percentage of time used to solve the linear system~\eqref{eq:finalLinearSystem}. We can notice that, refining the meshes, the hybridization procedure has better performance (in time) than the standard procedure (the only exception is the 2D square mesh case, where probably the very particular structure of the matrix greatly helps in dealing with the linear system for the standard procedure).  
Furthermore, focusing only on the hybridization technique, we observe that the time improvement becomes more and more effective as the solving process time dominates over the one needed to deal with the static condensation (this occurs for larger and larger systems). All the quantities are expressed in seconds.

Finally, in Table \ref{tab:ratio2d3d} we display the ratios between the time needed to solve the hybridized problem and the original one, for both the 2D and the 3D cases, and considering the finest meshes. This quantity can be considered as an indicator of the gain when adopting the hybridization technique. As it can be seen, 3D problems exhibit the greatest improvement.

%%%%%%%%%%%%%%%%%%%%%%%%%%%%%%%%%%%%
%%%%%%%%%%%%% separate %%%%%%%%%%%%%%%%%%
%%%%%%%%%%%%%%%%%%%%%%%%%%%%%%%%%%%%
%\begin{table}[!htb]
%	\centering
%	\begin{tabular}{|c|c|}
%		\hline
%		\texttt{Square} & 1.16 \\
%		\texttt{Tria}   & 0.82 \\
%		\texttt{CVT}    & 0.60 \\
%		\texttt{Rand}   & 0.72 \\
%		\hline
%	\end{tabular}
%	\caption{Ratio between the time hybridization and the standard approach for 2D case with the finest meshes for each mesh type.}
%	\label{tab:ratio2d}
%\end{table}
%
%\begin{table}[!htb]
%	\centering
%	\begin{tabular}{|c|c|}
%		\hline
%		\texttt{Cube}   & 0.21 \\
%		\texttt{Tetra}  & 0.41 \\
%		\texttt{CVT}    & 0.31 \\
%		\texttt{Rand}   & 0.46 \\
%		\hline
%	\end{tabular}
%	\caption{Ratio between the time hybridization and the standard approach for 3D case with the finest meshes for each mesh type.}
%	\label{tab:ratio3d}
%\end{table}

%%%%%%%%%%%%%%%%%%%%%%%%%%%%%%%%%%
%%%%%%%%%%%%% unite %%%%%%%%%%%%%%%%%%
%%%%%%%%%%%%%%%%%%%%%%%%%%%%%%%%%%
%\begin{table}[ht]
%	\centering
%	\begin{tabular}{|c|c|cc|c|c|}
%		\cline{1-2}\cline{5-6}
%		\multicolumn{2}{|c|}{2D}&&&\multicolumn{2}{|c|}{3D}\\
%		\cline{1-2}\cline{5-6}
%		\texttt{Square} & 1.16 & & &\texttt{Cube}   & 0.21\\
%		\texttt{Tria}   & 0.82 & & &\texttt{Tetra}  & 0.41\\
%		\texttt{CVT}    & 0.60 & & &\texttt{CVT}    & 0.31\\
%		\texttt{Rand}   & 0.72 & & &\texttt{Rand}   & 0.46\\
%		\cline{1-2}\cline{5-6}
%	\end{tabular}
%	\caption{Ratio between the hybridization and the standard approach time for 2D and 3D cases with the finest meshes for each mesh type.}
%	\label{tab:ratio2d3d}
%\end{table}
\begin{table}[!ht]
	\tbl{\label{tab:ratio2d3d}Ratio between the hybridization and the standard approach time for 2D and 3D cases with the finest meshes for each mesh type.}
	{\begin{tabular}{@{}ccccccc@{}} 
			\toprule &\hphantom{0000000000000}{2D}& & & & {3D}\hphantom{0000000000000} &\\
			\colrule
			&\texttt{Square}&  1.16 &\hspace{2cm} & \texttt{Cube}  & 0.21&\\
			&\texttt{Tria} &  0.82 &  & \texttt{Tetra}  & 0.41&\\
			&\texttt{CVT} &  0.60 &  & \texttt{CVT} &  0.31&\\
			&\texttt{Rand} & 0.72 &  & \texttt{Rand}  & 0.46&\\
			\botrule
	\end{tabular}}
\end{table}

%%%%%%%%%%%%%%%%%%%%%%%%%%%%%%%%%%%%%%%%%%
%%
%% Conclusions
%%
%%%%%%%%%%%%%%%%%%%%%%%%%%%%%%%%%%%%%%%%%%

%\input{conclusions}

%%%%%%%%%%%%%%%%%%%%%%%%%%%%%%%%%%%%%%%%%%
%%
%% Biblio
%%
%%%%%%%%%%%%%%%%%%%%%%%%%%%%%%%%%%%%%%%%%%
\bibliographystyle{ws-m3as}
\bibliography{biblioSistemata}

\end{document}